\documentclass[12pt,a4paper,onecolumn]{article}

\title{Semilinear Dirichlet problem for subordinate spectral Laplacian}

\author{Ivan Bio\v{c}i\'{c}}
\date{}

\usepackage[utf8]{inputenc}
\usepackage{fullpage}

\usepackage{amsmath}
\usepackage{amsthm}
\usepackage{amsfonts}
\usepackage{mathtools}
\usepackage{bbm} 
\usepackage[dvipsnames]{xcolor} 
\usepackage{amssymb} 
\usepackage{enumerate}
\usepackage{enumitem}
\usepackage{comment}
\usepackage{aligned-overset}
\usepackage{yfonts}
\usepackage{cancel,xcolor}

\usepackage{hyperref}
\hypersetup{
	colorlinks=true,
	linkcolor=magenta,
}
\DeclareMathOperator\supp{supp}

\newtheorem{thm}{Theorem}[section]
\newtheorem{prop}[thm]{Proposition}
\newtheorem{cor}[thm]{Corollary}
\newtheorem{defn}[thm]{Definition}
\newtheorem{lem}[thm]{Lemma}
\theoremstyle{definition}
\newtheorem{rem}[thm]{Remark}

\newtheorem*{assumption*}{\assumptionnumber}
\providecommand{\assumptionnumber}{}
\makeatletter
\newenvironment{assumption}[2]
{%
	\renewcommand{\assumptionnumber}{(\textbf{#1#2})}%
	\begin{assumption*}%
		\protected@edef\@currentlabel{(\textbf{#1#2})}%
	}
	{%
	\end{assumption*}
}
\makeatother

\newcommand{\R}{\mathbb{R}}
\newcommand{\N}{\mathbb{N}}
\newcommand{\p}{\mathbb{P}}
\newcommand{\ex}{\mathbb{E}}
\newcommand{\subsub}{\subset\subset}
\newcommand{\1}{\mathbf 1}
\newcommand{\G}{\mathbb{G}}

\newcommand{\LLL}{L^1(D,\de(x)dx)}
\newcommand{\BB}{\mathcal{B}}
\newcommand{\MM}{\mathcal{M}}

\newcommand{\KK}{\mathcal{K}}
\newcommand{\DD}{\mathcal{D}}
\newcommand{\PP}{\mathcal{P}}
\newcommand{\II}{\mathcal{I}}

\newcommand{\Lo}{\phi(-\left.\Delta\right\vert_{D})}
\newcommand{\Loz}{\phi^*(-\left.\Delta\right\vert_{D})}
\newcommand{\LoI}{\left[\Lo\right]^{-1}}
\newcommand{\wLo}{\widetilde{\phi}(-\left.\Delta\right\vert_{D})}

\newcommand{\pLo}{\phi_p(-\left.\Delta\right\vert_{D})}
\newcommand{\pLoz}{\phi^*_p(-\left.\Delta\right\vert_{D})}

\newcommand{\GDfi}{G^\phi_{ D}}
\newcommand{\GDFI}{\GDfi}
\newcommand{\GDfz}{G^{\phi^*}_{ D}}
\newcommand{\GD}{G_{ D}}
\newcommand{\PDfi}{P^\phi_{ D}}
\newcommand{\PDFI}{\PDfi}
\newcommand{\PD}{P_{D}}
\newcommand{\de}{\delta_D}
\newcommand{\wh}{\widehat}
\newcommand{\wt}{\widetilde}
\newcommand{\diam}{\textrm{diam}}
\newcommand{\dist}{\textrm{dist}}
\newcommand{\uu}{\textswab{u}}
\newcommand{\vv}{\textswab{v}}

\numberwithin{equation}{section}

\begin{document}
	
	\maketitle
	\begin{abstract}
		We study semilinear problems in bounded $C^{1,1}$ domains for non-local operators with a boundary condition. The operators cover and extend the case of the spectral fractional Laplacian. We also study harmonic functions with respect to the non-local operator and boundary behaviour of Green and Poisson potentials.
	\end{abstract}
	
	\bigskip
	\noindent {\bf AMS 2020 Mathematics Subject Classification}: Primary 35J61, 35R11; Secondary 35C15, 31B10, 31B25, 31C05, 60J35

	\bigskip\noindent
	{\bf Keywords and phrases}: semilinear differential equations, non-local operators, subordinate killed Brownian motion, harmonic functions
	
	\section{Introduction}
	Let $D\subset \R^d$, $d\ge 2$, be a bounded $C^{1,1}$ domain, $f:D\times \R\to \R$ a function, and $\zeta$ a signed measure on $\partial D$. In this article we study the semilinear problem
	\begin{equation}\label{e:4-intro}
		\begin{array}{rcll}
			\Lo u(x)&=& f(x,u(x))& \quad \text{in } D,\\
			\frac{u}{\PDFI\sigma}&=&\zeta&\quad \text{on }\partial D,
		\end{array}
	\end{equation}
	where $\phi:(0,\infty)\to (0,\infty)$ is a complete Bernstein function without drift satisfying a certain weak scaling conditions. The boundary condition will be described below whereas the operator $\Lo$ can be written in its spectral form as well as a principal value integral:
	$$
	\Lo u(x)=\sum_{j=1}^\infty \phi(\lambda_j)\wh u_j\varphi_j=\textrm{P.V.}\int_{D}(u(x)-u(y))J_D(x,y)\, dy+\kappa(x)u(x), \quad x\in D.
	$$
	Here $(\lambda_j,\varphi_j)_{j\in\N}$ are eigenpairs of the Dirichlet Laplacian in $D$, and  the singular kernel $J_D$ as well as the function $\kappa$ are completely determined by the function $\phi$. This said, $\Lo$ is a non-local operator of elliptic type which in the case $\phi(\lambda)=\lambda^{\alpha/2}$, $\alpha\in (0,2)$,  is the spectral fractional Laplacian $(-\left.\Delta\right\vert_{D})^{\alpha/2}$. The operator $-\Lo$ can be also viewed as the infinitesimal generator of the subordinate killed Brownian motion, where the subordinator has $\phi$ as its Laplace exponent.
	
	The notion of the boundary condition is a bit abstract but, at this point, let us say that it can be understood as a limit at the boundary of $u/\PDFI\sigma$  in the pointwise sense, or in the weak sense of \eqref{eq:distri solution boundary}, depending on the smoothness of the boundary datum, where $\PDFI\sigma$ is a reference function defined as the Poisson potential of the $d-1$ dimensional  Hausdorff measure on $\partial D$.
	
	Motivated by the recent articles \cite{aba+dupaig,semilinear_bvw}, see also the preprint \cite{AGCV19}, we consider solutions of \eqref{e:4-intro} in the weak dual sense, see Definition \ref{d:semilinear problem}, and we prove that the solutions have a special form of a sum of the Green and the Poisson potential, see Theorem \ref{t:linear problem}.
	
	Semilinear problems for the Laplacian have been studied for a long time now. In the monograph \cite{marcus_veron} it is said that this study is at least 50 years old now. However, the study of semilinear problems for non-local operators is quite recent and mostly oriented to the problems driven by the fractional Laplacian, see e.g. \cite{FQ12, CFQ, Aba15a, Aba17, BC17,BCBF16, BCBF18, bogdan_et_al_19, F20}. For more general operators than the fractional Laplacian see e.g. \cite{remarks_on_nonlocal,KKLL} for linear problems, and \cite{semilinear_bvw} for semilinear problems. However, there are just a few articles discussing the semilinear Dirichlet problem for the spectral fractional Laplacian, see \cite{dhifli2012,aba+dupaig}. To the best of our knowledge, this article is the first one to study semilinear problems for spectral-type operators more general than the spectral fractional Laplacian.
	
	A typical difference between the local and the non-local setting is that in the non-local setting even solutions of the linear Dirichlet problem can explode at the boundary whereas in the local setting this does not happen. To be more precise, there exists a harmonic function with respect to $\Lo$ which explodes at the boundary, e.g. the reference function $\PDFI\sigma$ is such one for which we prove the explosion rate, see \eqref{eq:Poisson integral sharp bounds}.
	
	The main goal of this article is to generalize results from \cite{aba+dupaig} where the semilinear problem was studied for the spectral fractional Laplacian, and to generalize results from \cite{semilinear_bvw} to a slightly different type of a non-local operator in the special case of $C^{1,1}$ bounded domain. To achieve this goal, we intensively  use the potential-theoretic and analytic properties of the killed Brownian motion subordinated by a subordinator with the Laplace exponent $\phi$, the process that gives $-\Lo$ as its infinitesimal generator as it is shown in the article. Some of these properties are well known for a long time and belong to the general potential theory. However, some properties are pretty recently proved such as the sharp bounds for the potential kernel and the jumping kernel, the (boundary) Harnack principle, etc., see \cite{ksv_minimal2016,ksv_potential_SKBM_2020}.
	
	Let us now describe the central results of the article which are given in Section \ref{s:semilinear problem}. For the nonlinearity $f$ in \eqref{e:4-intro} in our results we  assume that	
	\begin{assumption}{F}{}
		$f:D\times \R\to\R$ is continuous in the second variable, and there exist a locally bounded function $\rho:D\to[0,\infty]$ and a non-decreasing function $\Lambda:[0,\infty)\to[0,\infty)$ such that $|f(x,t)|\le \rho(x)\Lambda(|t|)$, $x\in D$, $t\in \R$.
	\end{assumption}
	\noindent First, in Proposition \ref{p:Kato's inequality for t+}, we prove Kato's inequality for $\Lo$  using which we develop a method of sub- and supersolution for $\Lo$ in Theorem \ref{t:super/subsolution implies existence}. This theorem directly generalizes \cite[Theorem 32]{aba+dupaig} to our setting of more general non-local operators and also extends \cite[Theorem 3.6]{semilinear_bvw} to slightly different non-local operators. In Theorem \ref{t:t1 semilinear problem} we  prove  the existence of a solution when the nonlinearity $f$ is non-positive and when the boundary measure $\zeta$ is non-negative. This theorem comes as a generalization of \cite[Theorem 8]{aba+dupaig} to our setting of more general non-local operators. Moreover, we consider a more general boundary condition which can also be a measure whereas in \cite[Theorem 8]{aba+dupaig} only continuous functions where considered. The nonlinearity in our theorem is also slightly  more general than the one in  \cite[Theorem 8]{aba+dupaig}. A similar result in a different non-local setting can be found in \cite[Theorem 3.10]{semilinear_bvw}. By the method of monotone iterations, in Theorem \ref{t:semilin non-negative monotone linearity} we find a solution to the semilinear problem when both $f$ and $\zeta$ are non-negative. Finally, for a signed $f$ and a signed $\zeta$, in Theorem \ref{t:semilinear signed datum} we find a solution by the technique used in \cite[Theorem 2.4]{bogdan_et_al_19}. After each theorem, we give a comment on the existence (and non-existence) of a solution in the spectral fractional Laplacian case for the power-like nonlinearity $f$, see Remarks \ref{r:semilin negative rem}, \ref{r:semilin positive rem} and \ref{r:semilin bog rem}.
	
	Let us now give a short summary of the rest of the article. In Section \ref{s:prelims} we introduce assumptions on $\phi$ and recall the known results on the Green kernel. We connect the operator $\Lo$ to the subordinate killed Brownian motion as its infinitesimal generator, give a pointwise characterization of $\Lo$, and study the regularity of the Green potentials. The last part of the section deals with Poisson potentials and harmonic functions. In Proposition \ref{p:Poisson kernel definition} we prove the existence of the Poisson kernel as a normal derivative of the Green function and in Theorem \ref{t:PDFI harmonic} we prove an integral representation formula for non-negative harmonic functions for $\Lo$. We finish the section by proving that harmonic functions are continuous, and by Theorem \ref{t:harmonic functions} in which we show that non-negative harmonic functions are those which satisfy the mean-value property with respect to the subordinate killed Brownian motion.
	
	Section \ref{s:boundary behaviour} deals with the boundary behaviour of potential integrals. Here we emphasize Theorem \ref{t:boundary of GDFIU} which gives the boundary behaviour of the Green potentials. This theorem generalizes \cite[Proposition 7]{dhifli2012} to our setting of more general non-local operators and more general functions.  Furthermore, this theorem with Proposition \ref{p:boundary oper. of Poisson} shows that in some cases the boundary condition \eqref{e:4-intro} can be understood as a limit at the boundary in the pointwise sense. Finally, the section also contains Proposition \ref{p:boundary oper. Poisson L1} and Proposition \ref{p:boundary operator on GDFi} which show that the boundary condition in \eqref{e:4-intro} can be viewed as a limit at the boundary in the weak sense.
	
	Section \ref{s:linear problem} contains the basic properties of the linear Dirichlet problem where we prove that every weak solution of the Dirichlet problem is a sum of the Green and the Poisson potential, see Theorem \ref{t:linear problem}.
	
	Section \ref{s:semilinear problem} contains already described main results.
	
	The article also contains the \hyperref[s:appendix]{Appendix} where  we first provide a proof of the Green function sharp estimate in our setting, see Lemma \ref{apx:l:Green fun est}, modelled upon \cite[Theorem 3.1]{ksv_minimal2016}. We also give a technical proof of Theorem \ref{t:boundary of GDFIU} modelled upon the proof of \cite[Proposition 4.1]{semilinear_bvw}, as well as prove Lemma \ref{apx:l:boundary operator of GDFI 2} which is an additional and a bit lengthy calculation providing an interpretation of the boundary condition. In Subsection \ref{ss:heat kernel regularity} we prove that the heat kernel of a killed Brownian motion upon exiting a $C^{1,\alpha}$ domain is differentiable up to the boundary - a fact that appears to be known but for which we could not find an exact reference.
	\paragraph{Notation.}
	For an open set $D\subset \R^d$: $C(D)$ denotes the set of all continuous functions on $D$, $C^k(D)$ denotes $k$-times ($k\ge 1$) continuously differentiable functions on $D$,  $C^\infty(D)$ infinitely differentiable functions on $D$, and $C_c^\infty(D)$ infinitely differentiable functions with  compact support on $D$, where e.g. we write $C^k(\overline D)$ for the set of functions in $C^k(D)$ whose all derivatives of order less than $k$ have a continuous extension to $\overline D$. For $\alpha\in(0,1]$, by $C^{1,\alpha}(\overline D)$ ($C^{1,\alpha}(D)$) we denote functions in $C^1(\overline D)$ ($C^1(D)$) whose first partial derivatives are uniformly H\"older continuous (locally H\"older continuous) in $D$ with exponent $\alpha$.
	
	Further, $L^1(D,\mu)$ is the set of all integrable functions on $D$, and $L^1_{loc}(D,\mu)$ the set of all locally integrable functions on $D$, with respect to the measure $\mu$ on $D$. If $\mu$ is the standard Lebesgue measure, we write $L^1(D)$ and $L_{loc}^1(D)$. The set $L^2(D)$ denotes square-integrable functions with respect to the Lebesgue measure. The set $H_0^1(D)$ denotes the closure of $C_c^\infty(D)$ with respect to the Sobolev norm in the Sobolev space $W^{1,2}(D)$ - the set of $L^2(D)$ functions whose weak partial derivatives belong to $L^2(D)$. The set $\BB(\R^d)$ denotes Borel sets in $\R^d$. The set $\MM(\partial D)$ and $\MM(D)$ denote Radon (signed) measures on $\partial D$ and $D$, respectively. We assume that all functions in the article are Borel functions, and all (signed) measures are Borel measures. Furthermore, in what follows  when we say $\nu$ is a measure,  we mean that $\nu$ is a non-negative measure on $\R^d$. By $|\nu|$ we denote the total variation of a signed measure $\nu$, and the Dirac measure of a point $x\in\R^d$ is denoted by $\delta_x$.
	
	The boundary of the set $D$ is denoted by $\partial D$. Notation $U\subsub D$ means that $U$ is a  nonempty  bounded open set such that $U\subset \overline U\subset D$ where $\overline U$ denotes  the  closure of $U$. By $|x|$ we denote the Euclidean norm of $x\in\R^d$ and $B(x,r)$ denotes  the  ball around $x\in\R^d$ with radius  $r>0$. For $A,B\subset \R^d$ let $\mathrm{dist}(A,B)=\inf\{|x-y|:x\in A,\,y\in B\}$ and $\diam D=\sup\{|x-y|:x,y\in D\}$. Unimportant constants in the article will be denoted by small letters $c$, $c_1$, $c_2$, $\dots$, and their labeling starts anew in each new statement. By a big letter $C$ we denote some more important constants, where e.g. $C(a,b)$ means that the constant $C$ depends only on parameters $a$ and $b$. All constants are positive finite numbers. For two positive functions $f$ and $g$ we write $f\asymp g$ ($f\lesssim g$, $f\gtrsim g$) if there exist a finite positive constant $c$ such that $c^{-1} f\le g\le c\,f$ ($f\le c\,g$, $f\ge c g$). Finally, $a\wedge b=\min\{a,b\}$ and $a\vee b=\max\{a,b\}$.

	\section{Preliminaries}\label{s:prelims}
	Let $(W_t)_{t\ge0}$ be a Brownian motion in $\R^d$, $d\ge2$, with the characteristic exponent $\xi\mapsto |\xi|^2$, $\xi\in\R^d$. Let $D$ be a non-empty open set, and $\tau_D\coloneqq \inf\{t>0:W_t\notin D\}$ the first exit time from the set $D$. We define the killed process $W^D$ upon exiting the set $D$ by
	\begin{align*}
		W^D_t\coloneqq\begin{cases}
			W_t,&t<\tau_D,\\
			\partial, &t\ge \tau_D,
		\end{cases}
	\end{align*}
	where $\partial$ is an additional point added to $\R^d$ called the cemetery.
	
	Let $S$ be a subordinator independent of $W$, i.e. $S$ is an increasing L\'evy process such that $S_0=0$, with the Laplace exponent
	\begin{align}\label{eq:Laplace exponent defn}
		\lambda\mapsto \phi(\lambda)=b\lambda+\int_{0}^\infty (1-e^{-\lambda\,t})\mu(dt),
	\end{align}
	where $b\ge0$ and the measure $\mu$ satisfies $\int_0^\infty (1\wedge t)\mu(dt)<\infty$. The measure $\mu$ is called the L\'evy measure and $b$ the drift of the subordinator. 
	
	The process $X=((X_t)_{t\ge0},(\p_x)_{x\in D})$ defined by $X_t\coloneqq W^D_{S_t}$ is called the subordinate killed Brownian motion. Here $\p_x$ denotes the probability under which the process $X$ starts from $x\in D$, and by $\ex_x$ we denote the corresponding expectation.

	\subsection{Assumptions}
	The first assumption that we impose throughout the article concerns the set $D$. Although some results will be valid for general open sets, we always assume that $D$ is a bounded $C^{1,1}$ domain.
	
	The second assumption concerns the Laplace exponent $\phi$, i.e. the subordinator $S$. A function of the form \eqref{eq:Laplace exponent defn} is called a Bernstein function, see \cite[Theorem 3.2]{bernstein}, and such functions characterize subordinators, see \cite[Chapter 5]{bernstein}.
	
	We impose the following assumption on $\phi$ throughout the article.
	\begin{assumption}{WSC}{}\label{A1}
		The function $\phi$ is a complete Bernstein function, i.e. the L\'evy measure $\mu(dt)$ has a completely monotone density $\mu(t)$, and  $\phi$ satisfies the following weak scaling condition at infinity: There exist $a_1,a_2>0$ and $\delta_1, \delta_2  \in(0,1)$ satisfying
		\begin{align}\label{eq:scaling condition}
			a_1\lambda^{\delta_1}\phi(t)\le \phi(\lambda t)\le a_2\lambda^{\delta_2}\phi(t),\quad t\ge1, \lambda\ge 1.
		\end{align}
	\end{assumption}

	The best-known subordinator with the property \ref{A1} is the $\alpha$-stable subordinator where $\phi(\lambda)=\lambda^{\alpha/2}$, for some $\alpha\in(0,2)$, which satisfies exact (and even global) scaling condition \eqref{eq:scaling condition}. However, there are many other interesting subordinators that fall into our setting. For a short list of these, see e.g. \cite[p. 3]{KSV_bhpinf}.
	
	Allow us to give some comments on the assumptions above. Since we assume \ref{A1}, the function $\phi^*(\lambda)\coloneqq \frac{\lambda}{\phi(\lambda)}$ is a complete Bernstein function, too, see \cite[Proposition 7.1]{bernstein}, and $\phi^*$ is called the conjugate Bernstein function of $\phi$. We easily see that \eqref{eq:scaling condition} also holds for $\phi^*$ but with different constants $a_1$, $a_2$, and $\delta_1$ and $\delta_2$, thus \ref{A1} also holds for $\phi^*$. By $\nu(dt)=\nu(t)dt$ we denote the L\'evy measure of $\phi^*$.

	In what follows we discuss properties of $\phi$ and the same will hold for $\phi^*$ or, to be more precise, for the counterparts of the function $\phi^*$.
	By a simple calculation, the scaling condition \eqref{eq:scaling condition} implies that $b=0$ in \eqref{eq:Laplace exponent defn} and the well-known bound
	\begin{align}\label{eq:scaling and the derivative}
		\phi'(\lambda)\asymp \frac{\phi(\lambda)}{\lambda}, \quad\lambda\ge1,
	\end{align}
	where, in fact, the upper bound holds for every Bernstein function and every $\lambda>0$, and the lower bounds follows from \eqref{eq:scaling condition}. We use \eqref{eq:scaling and the derivative} many times throughout the article. The L\'evy measure $\mu(dt)$ is infinite, see \cite[p. 160]{bernstein}, and the density $\mu(t)$ cannot decrease too fast, i.e. there is $c=c(\phi)>1$ such that
	$$ \mu(t)\le c\mu(t+1), \quad t\ge 1,$$
	see \cite[Lemma 2.1]{ksv_twosided}. Moreover, it holds that
	\begin{align}\label{eq:Levy density upper bound}
		\mu(t)\le (1-2e^{-1})^{-1}\frac{\phi'(t^{-1})}{t^2},\enskip t>0,\quad\text{and}\quad\mu(t)\ge c\,\frac{\phi(t^{-1})}{t^2},\enskip 0<t\le M,
	\end{align}
	for $M>0$ and $c=c(\phi,M)>0$, see \cite[Eq. (2.13)]{ksv_minimal2016} and \cite[Proposition 3.3]{mimica}.
	
	The potential measure $U$ of the subordinator $S$, defined by $U(A)\coloneqq \int_0^\infty \p(S_t\in A)dt$, $A\in \BB(\R)$, has a decreasing density $\uu$  which satisfies $\int_0^1 \uu(t)dt<\infty$, see \cite[Theorem 11.3]{bernstein}. In addition, it holds that
	\begin{align}\label{eq:potent density upper bound}
		\uu(t)\le (1-2e^{-1})^{-1}\frac{\phi'(t^{-1})}{t^2\phi(t^{-1})^2},\enskip t>0,\quad\text{and}\quad\uu(t)\ge c\frac{\phi'(t^{-1})}{t^2\phi(t^{-1})^2},\enskip 0<t\le M,
	\end{align}
	for $M>0$ and $c=c(\phi,M)>0$,	see \cite[Eq. (2.11)]{ksv_minimal2016} and \cite[Proposition 3.4]{mimica}. The potential density of $\phi^*$ will be denoted by $\vv$.
	
	It is worth noting that for a general Bernstein function a version of a  global scaling condition holds
	\begin{align}\label{eq:simple global scaling}
		1\wedge \lambda \le \frac{\phi(\lambda \, t)}{\phi(t)}\le 1\vee  \lambda,\quad \lambda>0,t>0,
	\end{align}
	which we get directly from \eqref{eq:Laplace exponent defn}.

	In \cite{ksv_minimal2016,ksv_potential_SKBM_2020} important aspects of the potential theory of the process $X$ were developed such as the scale invariant Harnack principle and the boundary Harnack principle. Our assumption \ref{A1}  implies (A1)-(A4) but not (A5) from \cite{ksv_minimal2016,ksv_potential_SKBM_2020} so each time we use a result from \cite{ksv_minimal2016,ksv_potential_SKBM_2020} we will explain how the assumption (A5) can be avoided.

	In the article the case $d = 1$ is excluded since it would require a somewhat different potential theoretic methods.
	
	\subsection{Green function}
	Let us denote the transition density of the Brownian motion $W$ by
	\begin{align}\label{eq:trans.dens.BM}
		p(t,x,y)=(4\pi t)^{-d/2}e^{-\frac{|x-y|^2}{4t}},\quad x,y\in\R^d,\, t>0.
	\end{align}
	Then the transition density of the killed Brownian motion $W^D$ is given by
	\begin{align}\label{eq:trans.dens. KBM}
		p_D(t,x,y)=p(t,x,y)-\ex_x[p(t-\tau_D,W_{\tau_D},y)\1_{\{\tau_D<t\}}],\quad x,y\in\R^d.
	\end{align}
	It is well known that $p_D(t,\cdot,\cdot)$  is symmetric and  it seems that it is  known that $p_D(\cdot,\cdot,\cdot)\in C^1((0,\infty)\times \overline D\times \overline D)$ since $D$ is a $C^{1,1}$ open domain. However, as we were unable to find an exact reference for the regularity up to the boundary of the transition density, we prove it in the Appendix in Lemma \ref{apx:l:regularity of heat kernel}. Furthermore, the following heat kernel estimate holds: There exist constants $T_0=T_0(D)>0$, $c_1=c_1(T_0,D)>0$, $c_2=c_2(T_0,D)>0$, $c_3=c_3(D)>0$, and $c_4=c_4(D)>0$ such that for all $x,y\in D$ and $t\in(0,T_0]$ it holds that
	\begin{align}\label{eq:heat kernel estimate}
		\left[\frac{\de(x)\de(y)}{t}\wedge 1\right]\frac{1}{c_1 	t^{d/2}}e^{-\frac{|x-y|^2}{c_2t}}\le p_D(t,x,y)\le \left[\frac{\de(x)\de(y)}{t}\wedge 1\right]\frac{c_3}{t^{d/2}}e^{-\frac{c_4|x-y|^2}{t}}.
	\end{align}
	We note that the right hand side inequality in \eqref{eq:heat kernel estimate} holds for every $t>0$. For the proofs see \cite[Theorem 3.1 \& Theorem 3.8]{song_sharp_bounds_2004}, cf. \cite[Theorem 1.1]{zhang_heat_kernel} and \cite[Theorem 4.6.9]{davies_HeatKernel}. Moreover, in \cite[Remark 3.3]{song_sharp_bounds_2004} one can find an explanation why the lower bound in \eqref{eq:heat kernel estimate} cannot hold uniformly for every $t>0$.
	
	The semigroup $(P_t^D)_{t\ge 0}$ of the process $W^D$ is given by
	\begin{align}\label{eq:Brownian semigroup}
		P_t^D f(x)=\int_Dp_D(t,x,y)f(y)dy=\ex_x[f(W_t);t<\tau_D]=\ex_x[f(W^D_t)],\enskip f\in L^\infty(D),
	\end{align}
	where $f(\partial)=0$ for all Borel functions on $D$ by convention. It is well known that the semigroup $(P_t^D)_{t\ge 0}$ is  strongly Feller since $D$ is $C^{1,1}$, i.e. $P_t^D(L^\infty)\subset C_b(D)$, and can be uniquely extended to a $L^2(D)$ semigroup. For details see e.g. \cite[Chapter 2]{chung_zhao}.

	The potential kernel of $W^D$ (or the Green function of $W^D$) is defined as
	\begin{align*}
		\GD(x,y)=\int_0^\infty p_D(t,x,y)dt,\quad x,y\in\R^d.
	\end{align*}
	The kernel $\GD$ is symmetric, non-negative, finite off the diagonal and jointly continuous in the extended sense, see \cite[Theorem 2.6]{chung_zhao}, and it is the density of the mean occupation time for $W^D$, i.e. for $f\ge0$ we have
	$$ \int_D \GD(x,y)f(y)dy=\ex_x\left[\int_0^\infty f(W^D_t)dt\right], \quad x\in D.$$
	
	Since $X$ is obtained by subordinating the killed Brownian motion $W^D$, it is well known that the $L^2(D)$ transition semigroup of $X$, denoted by $(Q_t^D)_t$, is given for $t>0$ by
	
	\begin{align*}
		Q_t^D f=\int_0^\infty P_s^Df \,\p(S_t\in ds),\quad f\in L^2(D),
	\end{align*}
	see \cite[Proposition 13.1]{bernstein}. Thus, $Q_t^D$ admits the density 
	\begin{align*}
		q_D(t,x,y)=\int_0^\infty p_D(s,x,y)\;\p(S_t\in ds).
	\end{align*}
	The semigroup $(Q_t^D)_t$ is also strongly Feller since $(P_t^D)_t$ is, see \cite[Proposition V.3.3]{bliedtner}.
	The process $X$ has the potential kernel (i.e. the Green function of $X$) which is given by
	\begin{align}\label{eq:defn Green function 1}
		\GDfi(x,y)=\int_0^\infty q_D(t,x,y)dt=\int_0^\infty p_D(t,x,y)\uu(t)dt,\quad x,y\in\R^d.
	\end{align}
	The kernel $\GDFI$ is symmetric, non-negative, and by the bound \eqref{eq:heat kernel estimate} finite off the diagonal. Moreover, $\GDFI$ is the density of the mean occupation time for $X$, i.e. for $f\ge0$ we have
	$$ \int_D \GDFI(x,y)f(y)dy=\ex_x\left[\int_0^\infty f(X_t)dt\right], \quad x\in D.$$
	The closed form of $\GDFI$ is not known, but in \cite[Theorem 3.1]{ksv_minimal2016} the sharp estimate was obtained, i.e. we have
	\begin{align}\label{eq:Green function sharp estimate}
		\GDfi(x,y)\asymp \left(\frac{\de(x)\de(y)}{|x-y|^{2}}\wedge 1\right)\frac{1}{|x-y|^{d}\phi(|x-y|^{-2})},\quad x,y\in D,
	\end{align}
	where the constant of comparability depends only on $d$, $D$ and $\phi$. We note that the usage of the transience assumption (A5) from \cite{ksv_minimal2016} in \cite[Theorem 3.1]{ksv_minimal2016} can be avoided, see Lemma \ref{apx:l:Green fun est} in  the Appendix for the details. Further, by using the upper bound of \eqref{eq:heat kernel estimate} and the bounds \eqref{eq:Green function sharp estimate}, we can repeat the proof of \cite[Proposition 3.3]{ksv_minimal2016} to get that $\GDFI$ is infinite on the diagonal and jointly continuous in the extended sense in $D\times D$.
	
	By the characterization of Bernstein functions the conjugate Bernstein function $\phi^*$ generates a subordinator $(T_t)_{t\ge0}$, see \cite[Chapter 5]{bernstein}. From the previous subsection it follows that $(T_t)_{t\ge0}$ has a potential measure which also has the decreasing density  which we denote by $V(dt)=\vv(t)dt$, see \cite[Theorem 11.3 \& Corollary 11.8]{bernstein}. We define the potential kernel generated by $\phi^*$ with
	\begin{align}\label{eq:defn Green function 2}
		\GDfz(x,y)=\int_0^\infty p_D(t,x,y)\vv(t)dt,\quad x,y\in\R^d.
	\end{align}
	Since $\phi^*$ satisfies \ref{A1}, $\GDfz$ is also symmetric, finite off the diagonal, jointly continuous in extended sense $D\times D$ and satisfies the sharp bound \eqref{eq:Green function sharp estimate} where $\phi$ is replaced by $\phi^*$. Of course, the kernel $\GDfz$ can be viewed as the potential kernel of the process $(W_{T_t}^D)_{t\ge 0}$.
	
	The kernels $\GD$, $\GDFI$ and $\GDfz$ are also connected by the following well-known factorization.
	\begin{lem}\label{l: GDf(GDf*)=GD}
		For $x,y\in D$ it holds that
		\begin{align}\label{eq: GDf(GDf*)=GD}
			\int_D \GDfi(x,\xi)\GDfz(\xi,y)d\xi=\GD(x,y).
		\end{align}
	\end{lem}
	\begin{proof}
		The claim follows from \cite[Proposition 14.2(ii)]{bernstein} where we set $\gamma=\delta_y$.
	\end{proof}
	
	\subsection{Operator $\Lo$}\label{ss:operator}
	Let ${\{\varphi_j\}}_{j\in\N}$ 
	be a Hilbert basis of $L^2(D)$ consisting of eigenfunctions 
	of the Dirichlet Laplacian $\left.-\Delta\right\vert_D$, 
	associated to the eigenvalues $\lambda_j$, $j\in\N$, i.e.  $\varphi_j\in H^1_0(D)\cap C^{ \infty}(D)\cap C^{1,1}(\overline D)$ and
	\begin{align}\label{eq:defn of eigenvalues}
		\left.-\Delta\right\vert_D\varphi_j=\lambda_j\varphi_j, \quad \text{in $D$},
	\end{align}
	see \cite[Theorem 9.31]{brezis} and \cite[Section 8.11]{gilbarg_pde}. Here \eqref{eq:defn of eigenvalues} can be viewed in various equivalent ways, e.g. as a distributional or a pointwise relation. Also,  $\left.\Delta\right\vert_D$ in \eqref{eq:defn of eigenvalues} can be viewed as the $L^2(D)$-infinitesimal generator of the semigroup $(P_t^D)_t$, i.e.
	\begin{align*}
		\left.\Delta\right\vert_D u=\lim_{t\to 0}\frac{P_t^D u-u}{t}, \quad u\in \DD(\left.\Delta\right\vert_D),
	\end{align*}
	where $\DD(\left.\Delta\right\vert_D)$ is the domain of the generator $\left.\Delta\right\vert_D$ and the limit is taken with respect to $L^2(D)$ norm. We note that $\DD(\left.\Delta\right\vert_D)$ is a class of functions $f\in H_0^1(D)$ such that $\Delta f$ exists in the weak distributional sense and belongs to $L^2(D)$, see \cite[Theorem 2.13]{chung_zhao}.  For more details, see \cite[Chapter 2]{chung_zhao} and \cite[Chapter 9]{brezis}. Note that since $\varphi_j$ is an eigenfunction,  we have
	\begin{align}\label{eq:Pt=e{-lt}}
		P_t^D\varphi_j=e^{-\lambda_j t}\varphi_j,
	\end{align}
	see \cite[Lemma 7.10]{schilling_brownian}.
	Further, since we assume that $D$ is $C^{1,1}$, it is well known that $0<\lambda_1< \lambda_2\le\lambda_3\le \dots$, and by the Weyl's asymptotic law we have 
	\begin{align}\label{eq:Weyl's law}
		\lambda_j \asymp j^{2/d},\quad j\in \N.
	\end{align}
	Also, we choose the basis $\{\varphi_j\}_{j\in \N}$ such that $\varphi_1>0$ in $D$, see \cite[Chapter 9]{brezis}. Hence, another very important sharp estimate for $\varphi_1$ holds:
	\begin{align}\label{eq:varphi=de sharp estimate}
		\varphi_1(x)\asymp \de(x),\quad x\in D.
	\end{align}
	The interior estimate is trivial since $\varphi_1$ is smooth and positive. The boundary bound follows from Hopf's lemma, see e.g. \cite[Hopf's lemma in Section  6.4.2]{evans_pde}.
	
	Consider the Hilbert space
	\begin{align*}
		H_D(\phi)\coloneqq \left\{v=\sum_{j=1}^\infty \widehat v_j\varphi_j\in L^2(D):\|v\|_{H_D(\phi)}^2\coloneqq\sum_{j=0}^\infty\phi(\lambda_j)^{2}\vert\widehat v_j\vert^2<\infty\right\}.
	\end{align*}
	The spectral operator $\Lo:H_D(\phi)\to L^2(D)$ is defined as
	\begin{align}\label{eq:definition of phi(Delta D)}
		\Lo u=\sum_{j=1}^\infty\phi(\lambda_j)\widehat u_j\varphi_j, \quad u\in H_D(\phi).
	\end{align}
	Note that $H_D(\phi)\hookrightarrow L^2(D)$ and we will show in the next proposition that $C_c^\infty(D)\subset H_D(\phi)$, see \eqref{eq:Ccinfty coef bound}. Now it is  obvious that $\Lo$ is  an unbounded operator, densely defined in $L^2(D)$ and  has the bounded inverse $\LoI : L^2(D)\to H_D(\phi)$ given by
	\begin{align}\label{eq:defn of inverse operator}
		\LoI u=\sum_{j=1}^\infty \frac{1}{\phi(\lambda_j)}\widehat{u}_j\varphi_j,\quad u\in L^2(D).
	\end{align}
	
	In the next proposition we prove that a potential relative to $\GDFI$ is the inverse of $\Lo$. The proof is similar to \cite[Lemma 9]{aba+dupaig} but we give the complete proof for the reader's convenience since some elements of the proof will be important in what follows.
	\begin{prop}\label{p:LGDf=-f}
		Let $f\in L^2(D)$. For a.e. $x\in D$ it holds that $\GDfi(x,\cdot)f(\cdot)\in L^1(D)$ and
		\begin{align}\label{eq:LGDf=-f}
			\LoI f(x)=\int_D\GDfi(x,y)f(y)dy.
		\end{align}
	\end{prop}
	\begin{proof}
		First we prove \eqref{eq:LGDf=-f}  for $f=\varphi_1\ge0$. Fubini's theorem yields
		\begin{align}\label{eq:LGD fi = - fi}
			\begin{split}
				\int_D\GDfi(x,y)\varphi_1(y)dy&=\int_0^\infty \uu(t)\int_D p_D(t,x,y)\varphi_1(y)dydt=\int_0^\infty e^{-\lambda_1 t}\varphi_1(x)\uu(t)dt\\&=\frac{1}{\phi(\lambda_1)}\varphi_1(x)=\LoI\varphi_1(x),\quad \text{for a.e. $x\in D $,}
			\end{split}
		\end{align}
		where in the  second equality we used \eqref{eq:Pt=e{-lt}}, in the third \cite[Eq. (5.20)]{bernstein}, and in the last equality \eqref{eq:defn of inverse operator}.
		By the elliptic regularity there exist constants $C=C(d,D)$ and $k=k(d)$ such that $\|\nabla \varphi_j\|_{L^\infty(D)}\le (C\lambda_j)^k\|\varphi_j\|_{L^2(D)}=(C\lambda_j)^k$, see \eqref{eq:eigen estimates 2}. Recall that $\varphi_j\in C^{1,1}(\overline D)$ and that $\varphi_j$ vanishes on the boundary so the mean value theorem implies 
		\begin{align}\label{eq:eigenfunctions bounds}
			\left\|\frac{\varphi_j}{\de}\right\|_{L^\infty(D)}\leq(C\lambda_j)^k.
		\end{align}
		Since $\varphi_1\asymp \de$, by the previous inequality, Fubini's theorem, and the same calculations as in \eqref{eq:LGD fi = - fi}, we have that \eqref{eq:LGDf=-f} holds for every $\varphi_j$, $j\in \N$. By linearity the same is true for the linear span of $\{\varphi_j:j\in\N\}$.
		
		Let 
		\begin{align}\label{eq:GDf=int}
			\G f(x)\coloneqq \int_D \GDfi(x,y)f(y)dy,
		\end{align}
		for $f\in L^2(D)$ and $x\in D$ such that the integral exists. In what was proved, $\G f(x)$ is well defined for every $f\in \textrm{span}\{\varphi_j:j\in\N\}$ and a.e. $x\in D$. Moreover, from $\G\varphi_j=\frac{1}{\phi(\lambda_j)}\varphi_j=\LoI \varphi_j$ it follows that for $f\in \textrm{span}\{\varphi_j:j\in\N\}$ we have
		\begin{align}
			\| \G f\|^2_{H_D(\phi)}=\|f\|^2_{L^2(D)}.
		\end{align}
		Hence, the map $f\mapsto \G f$ uniquely extends to a linear isometry from $L^2(D)$ to $H_D(\phi)$ which coincides with $\LoI$. Further, a consequence of \eqref{eq:LGD fi = - fi} is that $\GDFI(x,\cdot)\in L^1(D)$ for a.e. $x\in D$ since by Fubini's theorem
		$$\int_D\left(\int_D\GDFI(x,y)dy\right)\varphi_1(x)dx =\frac{1}{\phi(\lambda_1)}\int_D \varphi_1(y)dy<\infty.$$
		
		Next we prove that \eqref{eq:LGDf=-f} holds a.e. in $D$ for $f=\psi=\sum_{j=1}^\infty\wh\psi_j\varphi_j \in C_c^\infty(D)$. 	
		Take the approximating sequence $f_n=\sum_{j=1}^n\wh\psi_j\varphi_j$, $n\in \N$, and note that $\G f_n=\LoI f_n\to \LoI f=\G f$ in $L^2(D)$ since $f_n\to f$ in $L^2(D)$. Moreover, by integrating by parts $m\in \N$ times we get
		\begin{align*}
			\widehat\psi_j=\int_D\psi(x)\varphi_j(x)dx=\frac{(-1)^m}{\lambda_j^m}\int_D\Delta^m\psi(x)\varphi_j(x)dx,
		\end{align*}
		which implies 
		\begin{align}\label{eq:Ccinfty coef bound}
			|\wh\psi_j|\le \frac{\|\Delta^m\psi\|_{L^2(D)}}{\lambda_j^m}\eqqcolon C(m,\psi)\frac{1}{\lambda^m_j}.
		\end{align}
		Hence, by using \eqref{eq:Weyl's law},  \eqref{eq:eigenfunctions bounds}, and \eqref{eq:Ccinfty coef bound} for large enough $m\in \N$, it follows that $f_n$ converges uniformly in $D$ to $f=\psi$. This implies that $\G f_n=\int_D \GDFI(\cdot,y)f_n(y)dy\to \int_D \GDFI(\cdot,y)f(y)dy$ a.e. in $D$ since $\GDfi(x,\cdot)\in L^1(D)$ for a.e. $x\in D$. Thus, by uniqueness of the  limit $\G f=\LoI f=\int_D \GDFI(\cdot,y)f(y)dy$ a.e. in $D$.
		
		Take now $f\in L^2(D)$, and let $(f_n)_n\subset C_c^\infty(D)$ which converges to $f$  in $L^2(D)$. Hence,  $\G f_n=\LoI f_n\to \LoI f=\G f$ in $L^2(D)$. On the other hand, 
		\begin{align*}
			\int_D\left|\int_D \GDFI(x,y)(f_n(y)-f(y))dy\right|\varphi_1(x)dx&\le \frac{1}{\phi(\lambda_1)}\int_D\varphi_1(y)|f_n(y)-f(y)|dy\\
			&\le \frac{1}{\phi(\lambda_1)}||f_n-f||_{L^2(D)}\to 0,
		\end{align*}
		which shows that $\GDFI(\cdot,y)f(y)\in L^1(D)$ a.e. in $D$ and by taking the subsequence we get $$\G f_n=\int_D \GDFI (\cdot,y )f_n(y)dy\to \int_D \GDFI (\cdot,y )f(y)dy\quad\text{a.e. in $D$,}$$ thus $\LoI f=\int_D\GDFI(\cdot,y)f(y)dy$ a.e. in $D$.
	\end{proof}
	In what follows, for the operator $\G$ from the proof of the previous lemma we will write
	\begin{align}\label{eq:defn of GDf}
		\GDfi f(x)\coloneqq\int_D\GDfi(x,y)f(y)dy=\G f(x),\quad x\in D.
	\end{align}
	
	\begin{rem}\label{r:infin. genera of X}
		Proposition \ref{p:LGDf=-f} implies that $\GDfi\big(L^2(D)\big)=H_D(\phi)$ and that $$\Lo(\GDfi f)=f,\quad f\in L^2(D).$$
		By the general theory of semigroups, this means that $-\Lo$ defined by \eqref{eq:definition of phi(Delta D)} is the infinitesimal generator of the semigroup $(Q^D_t)_t$ and that $H_\phi(D)$ is the domain of $-\Lo$, see e.g. \cite{yosida_functional_analysis,sato_potential_operators}. In particular, $\DD(\left.\Delta\right|_D)\subset\DD\big(-\Lo\big)= H_\phi(D)$ by \cite[Theorem 13.6]{bernstein}.
		
		Further, we note that $C^{1,1}(\overline D)\subset \DD(\left.\Delta\right\vert_D)$. Indeed, since the first partial derivatives of  a $C^{1,1}(\overline D)$ function are Lipschitz functions, the first partial derivatives have the second partial derivatives almost everywhere. Furthermore, these second partial derivatives are in $L^\infty(D)$ since the first ones satisfy the Lipschitz property uniformly in $D$. By \cite[Theorem 2.13]{chung_zhao} we get $C^{1,1}(\overline D)\subset \DD(\left.\Delta\right\vert_D)$. Hence, by the first part of this remark, we have $C_c^\infty(D)\subset C^{1,1}(\overline D)\subset\DD(\left.\Delta\right|_D)\subset\DD\big(-\Lo\big)= H_\phi(D)$. 
	\end{rem}

	For sufficiently regular functions $\Lo u$ can be expressed pointwisely. At this point we only consider $u\in C^{1,1}(D)\cap \DD(\left.\Delta\right|_D)$ but later on in Proposition \ref{p:Lo various ways} we will prove the pointwise representation of $\Lo$ for $u\in C^{1,1}(D)\cap H_\phi(D)$.
	\begin{lem}\label{l:Lo pointwisely}
		Let $u\in C^{1,1}(D)\cap \DD(\left.\Delta\right|_D)$. Then for a.e. $x\in D$
		\begin{align}\label{eq:Lo pointwisely}
			\Lo u(x)=\textrm{P.V.}\int\limits_D [u(x)-u(y)]J_D(x,y)dy + \kappa(x)u(x),
		\end{align}
		where
		\begin{align*}
			J_D(x,y)\coloneqq \int_0^\infty p_D(t,x,y)\mu(t)dt,\quad \kappa(x)\coloneqq\int_0^\infty \left(1-\int_Dp_D(t,x,y)dy\right)\mu(t)dt.
		\end{align*}
		In particular, \eqref{eq:Lo pointwisely} holds for $u\in C_c^\infty(D)$.
	\end{lem}
	
	\begin{rem}\label{r:about JD(x,y)}
		The function $J_D$ is called the jumping density and the function $\kappa$ is called the killing function of the process $X$.
		Obviously, $J_D$ is non-negative and symmetric. It is also finite off the diagonal and satisfies $\int_D \big(1\wedge |x-y|^2\big)J_D(x,y)dy<\infty$ since the following estimate holds
		\begin{align}\label{eq:jumping kernel sharp estimate}
			J_D(x,y)\asymp\left(\frac{\de(x)\de(y)}{|x-y|^{2}}\wedge 1\right)\frac{\phi(|x-y|^{-2})}{|x-y|^{d}},\quad x,y\in D.
		\end{align}
		Here the constant of comparability depends only on $d$, $D$ and $\phi$ and the proof of \eqref{eq:jumping kernel sharp estimate} is essentially the same as the proof of \eqref{eq:Green function sharp estimate}. By applying comments given for the proof of \cite[Proposition 3.5]{ksv_minimal2016} and using similar manipulations as in the proof of Lemma \ref{apx:l:Green fun est} to avoid using (A5) from \cite{ksv_minimal2016}, we easily obtain \eqref{eq:jumping kernel sharp estimate}, so we skip the proof.

		The killing function $\kappa$ is continuous and $\kappa\in\LLL$. Indeed, since the semigroup $P_t^D$ is strongly Feller, $1-P_t^D\1(x)=\p_x(\tau_D\le t)$ is continuous in $x$. Further, for $\varepsilon>0$ such that $\varepsilon<2\de(x)$ it holds that $\p_x(\tau_D\le t)\le\p_x(\tau_{B(x,\varepsilon)}\le t)=\p_0(\tau_{B(0,1)}\le \frac{t}{\varepsilon^2})\le  c_1(\varepsilon)(1\wedge t)$, where the last inequality follows by e.g.  \cite[Theorem 1]{pei_hsu_BrownianExitDistribution}. Now the dominated convergence theorem yields the continuity of $\kappa$. Finally, $\int_D\kappa(x) \varphi_1(x)dx=\phi(\lambda_1)\int_D\varphi_1(x)dx$ by \eqref{eq:Pt=e{-lt}}, so  \eqref{eq:varphi=de sharp estimate} yields $\kappa \in \LLL$.
	\end{rem}
	\begin{proof}[\textbf{Proof of Lemma \ref{l:Lo pointwisely}}]		
		It is known that for all $u\in \DD(\left.\Delta\right|_D)$ it holds that
		\begin{align}\label{eq:pointwise Lo - Bochner way}
			\Lo u=\int_0^\infty \big(u-P_t^D u\big)\mu(t)dt
		\end{align}
		see \cite[Theorem 13.6]{bernstein}, since $\DD(\left.\Delta\right|_D) \subset \DD(-\Lo)= H_\phi(D)$ by Remark \ref{r:infin. genera of X}. The rest of the proof is dedicated to showing that the right hand sides of \eqref{eq:pointwise Lo - Bochner way} and \eqref{eq:Lo pointwisely} are equal. 
		
		Let $u\in C^{1,1}(D)\cap \DD(\left.\Delta\right|_D) $ and $x\in D$.  First, we show that the principal value integral in \eqref{eq:Lo pointwisely} is well defined. Indeed, fix $\delta>0$ such that $\delta<(1\wedge \de(x)/4)$ and let $\varepsilon>0$ such that $\varepsilon<\delta$. We have
		\begin{align*}
			&\int_{D\setminus B(x,\varepsilon)}\big(u(x)-u(y)\big)J_D(x,y)dy\\
			&\qquad\qquad\qquad\qquad=\int_{D\setminus B(x,\varepsilon)}\big(u(x)-u(y)+\nabla u(x)\cdot(y-x)\1_{B(x,\delta)}(y)\big)J_D(x,y)dy\\
			&\qquad\qquad\qquad\qquad\qquad\qquad\qquad-\int_{B(x,\delta)\setminus B(x,\varepsilon)}\nabla u(x)\cdot(y-x)J_D(x,y)dy
			\\&\qquad\qquad\qquad\qquad=I_1-I_2.
		\end{align*}
		By a $C^{1,1}$ version of Taylor's theorem we have 
		\begin{align}\label{eq:zzz bound}
			|u(x)-u(y)+\nabla u(x)\cdot(y-x)\1_{B(x,\delta)}(y)|\le c_1\,(1\wedge |x-y|^2),
		\end{align}
		where $c_1>0$ depends on $\delta$ and $\|u\|_{C^{1,1}(B(x,\de(x)/2))}$. Hence, the integral $I_1$ is finite and converges as $\varepsilon\to0$ by dominated convergence theorem.

		For the second integral, by Fubini's theorem  and  \eqref{eq:trans.dens. KBM}, we have
		\begin{align*}
			I_2&=\int_0^\infty \int_{B(x,\delta)\setminus B(x,\varepsilon)}\nabla u(x)\cdot(y-x)p(t,x,y)dy\mu(t)dt\\&\qquad\qquad-\int_0^\infty \int_{B(x,\delta)\setminus B(x,\varepsilon)}\nabla u(x)\cdot(y-x)\ex_x[p(t-\tau_D,W_{\tau_D},y)\1_{\{\tau_D<t\}}]dy\,\mu(t)dt\\
			&\eqqcolon J_1-J_2.
		\end{align*}
		The integral $J_1$ is zero for all $\varepsilon<\delta$ since the kernel $p(t,x,y)$ is symmetric in $y$ around $x$, and since the region of integration is symmetric around $x$. For the integral $J_2$ note that $|\nabla u(x)\cdot(y-x)|\le c_2\delta$, $y\in B(x,\delta)$, where $c_2=c_2(u)=\max_{B(x,\de(x)/2)}|\nabla u(x)|$, i.e. $c_2$ depends on local properties of $u$ around $x$. Also,
		\begin{align}\label{eq:zzz1 bound}
			p(t-\tau_D,W_{\tau_D},y)\1_{\{\tau_D<t\}}\le \frac{(4\pi)^{-d/2}}{(t-\tau_D)^{d/2}}e^{-\frac{\de(x)^2}{16(t-\tau_D)}}\1_{\{\tau_D<t\}}\le c_3(1\wedge t),\quad y\in B(x,\delta),
		\end{align}
		where $c_3=c_3(d,\de(x))>0$. Thus,
		\begin{align}\label{eq:zzz2 bound}
			\int_{B(x,\delta)\setminus B(x,\varepsilon)}|\nabla u(x)\cdot(y-x)|\,\ex_x[p(t-\tau_D,W_{\tau_D},y)\1_{\{\tau_D<t\}}]dy\le c_4\delta^{d+1} (1\wedge t), \quad t>0,
		\end{align}
		where $c_4=c_4(d,D,u,\de(x))>0$. In other words, we showed that
		\begin{align*}
			|I_2|\le c_6 \delta^{d+1},
		\end{align*}
		where $c_6=c_6(d,D,u,\de(x),\mu)>0$. Moreover, the bounds \eqref{eq:zzz1 bound} and \eqref{eq:zzz2 bound} imply that the integral $J_2$ converges as $\varepsilon\to0$  by the dominated convergence theorem. Hence, $I_2$ converges as $\varepsilon\to 0$. Finally, this means that the principal value integral in \eqref{eq:Lo pointwisely} is well defined.

		Now we prove \eqref{eq:Lo pointwisely}. For the fixed $\delta>0$ from above, by using \eqref{eq:pointwise Lo - Bochner way} we have
		\begin{align*}
			&\Lo u(x)=\int_0^\infty\big(u(x)-u(x)P_t^D\1(x)+u(x)P_t^D\1(x)-P_t^D u(x)\big)\mu(t)dt\\
			&=\int_0^\infty\left(\int_D\big(u(x)-u(y)\big)p_D(t,x,y)dy\right)\mu(t)dt+\kappa(x)u(x)\\
			&=\int_0^\infty\left(\lim_{\varepsilon\searrow0}\int_{D\setminus B(x,\varepsilon)}\big(u(x)-u(y)+\nabla u(x)\cdot (y-x)\1_{B(x,\delta)}(y)\big)p_D(t,x,y)dy\right)\mu(t)dt\\
			&\qquad\qquad-\int_0^\infty\left(\lim_{\varepsilon\searrow0}\int_{B(x,\delta)\setminus B(x,\varepsilon)}\big(\nabla u(x)\cdot (y-x)\big)p_D(t,x,y)dy\right)\mu(t)dt+\kappa(x)u(x)\\
			&=\lim_{\varepsilon\searrow0}\int_{D\setminus B(x,\varepsilon)}\big(u(x)-u(y)\big)J_D(x,y)+\kappa(x)u(x),
		\end{align*}
		where the change of the order of integration, as well as taking the limit outside the integral, was justified by \eqref{eq:zzz bound}, \eqref{eq:zzz1 bound} and \eqref{eq:zzz2 bound}.
	\end{proof}
	
	\begin{rem}
		Lemma	\ref{l:Lo pointwisely} suggest the pointwise definition of the operator $\Lo$, i.e. we define
		\begin{align}\label{eq:Lo pointwisely defn}
			\pLo u(x)=\textrm{P.V.}\int_D [u(x)-u(y)]J_D(x,y)dy + \kappa(x)u(x),
		\end{align}
		for every function $u$ and $x\in D$ for which \eqref{eq:Lo pointwisely defn} is well defined. E.g. this is true for every $x\in D$ if $u\in C^{1,1}(D)\cap \LLL$ by the proof of Lemma \ref{l:Lo pointwisely} and the bound \eqref{eq:jumping kernel sharp estimate}.
	\end{rem}
	
	To conclude the subsection, we bring the well-known factorization of the Dirichlet Laplacian $-\left.\Delta\right\vert_{D}$ which is closely related to Lemma \ref{l: GDf(GDf*)=GD}. Since $\phi^*$ satisfies \ref{A1}, the operator $\Loz$ can be defined in the same way as $\Lo$, and the same properties hold for $\Loz$. In what follows, such comments on the objects defined relative to $\phi$ and relative to $\phi^*$ will be skipped.
	\begin{lem}\label{l: Delta=phi(Delta)phi(Delta)}
		For $\psi\in C_c^\infty(D)$, it holds that 
		\begin{align*}
			\Lo \circ \Loz \psi = \Loz \circ \Lo\psi=(-\left.\Delta\right\vert_{D})\psi,\quad \text{a.e. in $D$.}
		\end{align*}
		Further, $(-\left.\Delta\right\vert_{D})\psi=-\Delta\psi$.
	\end{lem}
	\begin{proof}
		Recall that the operator $\left.\Delta\right\vert_{D}$ is the infinitesimal generator of the semigroup $(P_t^D)_t$ which on $C_c^\infty (D)$ functions acts like the standard Laplacian $\Delta$. Hence, the claim follows from \cite[Corollary 13.25]{bernstein} since $C_c^\infty(D)\subset \DD(\left.\Delta\right\vert_{D})$.
	\end{proof}

	\subsection{Green potentials}
	In this subsection we prove some useful identities related to the Green potentials, develop some integrability conditions and prove two regularity properties for $\GDFI f$.
	
	The next lemma says that the definition of the Green potential $\GDFI f$ in \eqref{eq:defn of GDf} makes sense for $f\in \LLL$, too, and that the operator $f\mapsto \GDFI f$ is bounded from $\LLL$ to itself.

	\begin{lem}\label{l:rate of GDdelta}
		It holds that
		\begin{align}\label{eq:rate of GDdelta}
			\GDfi\de(x)\asymp \de(x),\quad x\in D,
		\end{align}
		where the constant of comparability depends only on $d$, $D$ and $\phi$.
		Further, if $\lambda\in\MM(D)$ such that $\int_D\de(x)|\lambda|(dx)<\infty$ then
		\begin{align}\label{eq:rate of GDfi lambda}
			x\mapsto \GDFI\lambda(x)\coloneqq \int_D\GDFI(x,y)\lambda(dy)\in \LLL,
		\end{align}
		and there is $C=C(d,D,\phi)\ge 1$ such that $\|\GDFI \lambda\|_{\LLL}\le C \int_D\de(x)|\lambda|(dx)$.
	\end{lem}
	\begin{proof}
		Recall that $\varphi_1(x)\asymp \de(x)$, $x\in D$, by \eqref{eq:varphi=de sharp estimate}, thus by \eqref{eq:LGD fi = - fi}
		\begin{align*}
			\GDfi \de(x)\asymp\GDfi \varphi_1(x)=\frac{1}{\phi(\lambda_1)}\varphi_1(x)\asymp\de(x),\quad x\in  D.
		\end{align*}
		The second and the third claim follow from Fubini's theorem and \eqref{eq:rate of GDdelta}.
	\end{proof}
	
	\begin{cor}\label{c:GDf finite if f in L1}
		There is $C=C(d,D,\phi)>0$ such that for every $f\in\LLL$ it holds that $\|\GDFI f\|_{\LLL}\le C \|f\|_{\LLL}$.
	\end{cor}
	
	\begin{rem}\label{r:GDFI in Linfty}
		Let us note  that by using \eqref{eq:Green function sharp estimate} it easily follows that $\GDFI f\in L^\infty(D)$ for $f\in L^\infty(D)$. 
	\end{rem}
	
	\subsubsection{Operator $\Lo$ revisited}
	In the next lemma we prove the boundary estimate of $\Lo\psi$ for $\psi\in C_c^\infty(D)$ which will allow us to define the operator $\Lo$ in the distributional sense.
	\begin{lem}\label{l: Lo psi < C delta}
		For $\psi\in C_c^\infty(D)$ there is $C_1=C_1(d,D,\phi,\psi)>0$ such that
		\begin{align}\label{l: Lo psi < C delta 1}
			|\Lo\psi(x)|\le C_1\de(x),\quad x\in D.
		\end{align}
		In addition, if $\psi\ge 0$, $\psi\not\equiv 0$, then  there is $C_2=C_2(d,D,\phi,\psi)>0$ such that
		\begin{align}\label{l: Lo psi < C delta 2}
			\Lo \psi(x)\le - C_2\de(x),\quad x\in D\setminus \supp \psi.
		\end{align}
	\end{lem}
	\begin{proof}
		Let $\psi\in C_c^\infty(D)$  and note that $\phi(\lambda)\le (1\wedge \lambda)$ by \eqref{eq:simple global scaling}. Thus, from \eqref{eq:Weyl's law}, \eqref{eq:eigenfunctions bounds}, and  \eqref{eq:Ccinfty coef bound} for large enough $m\in \N$, we have
		\begin{align*}
			\frac{|\Lo\psi(x)|}{\de(x)}\le \sum_{j=1}^\infty |\wh\psi_j|\phi(\lambda_j)\left\|\frac{\varphi_j}{\de}\right\|_{L^\infty(D)}\le C_1(d,D,\phi,\psi).
		\end{align*}
		
		For the other bound let $x^*=\arg\max_{x\in D}\psi(x)$, and let $r>0$ such that $B(x^*,2r)\subset \supp\psi$ and $\psi\ge c>0$ on $B(x^*,2r)$. For $x\in D\setminus \supp\psi$, by using the representation \eqref{eq:Lo pointwisely} and the bound \eqref{eq:jumping kernel sharp estimate}, we have
		\begin{align}\label{eq:C_c inequality}
			\Lo\psi(x)=-\int_{\supp\psi}\psi(y)J_D(x,y)dy\le -\int_{B(x^*,r)} c_1\de(x)dy\le -C_2\,\de(x),
		\end{align}
		where $C_2=C_2(d,D,\psi,\phi)>0$. 
	\end{proof}
	
	\begin{defn}\label{d:operator distributional}
		For $f\in \LLL$ we define the distribution $\wLo f$ in $D$ by
		\begin{align*}
			\langle \wLo f,\psi\rangle\coloneqq \langle f,\Lo\psi\rangle\coloneqq\int_D f(x)\Lo \psi (x)dx,\quad \psi\in C_c^\infty(D).
		\end{align*}
	\end{defn}
	
	\begin{rem}\label{r:Lo weak and point}
		Sometimes for $\wLo f$ we say $\Lo f$ in the distributional sense. Notice that Lemma \ref{l: Lo psi < C delta} implies that the integral defining $\wLo f$ is well defined. 
		
		By following the calculations from \cite[Section 3]{bogdan1999potential}, we get that for $f\in  C^{1,1}(D)\cap\LLL$ we have $\wLo f=\pLo f$. 
	\end{rem}

	The next proposition says that the relation from Remark \ref{r:infin. genera of X} can be also extended to $\wLo$.
	
	\begin{prop}\label{p:wLo GDF=f no.2} 
		Let $\mu\in\MM(D)$ such that $\int_D\de(x)|\mu|(dx)<\infty$. Then $\wLo\GDFI\mu=\mu$.
	\end{prop}
	\begin{proof}
		Let $\psi\in C_c^\infty(D)$ and recall that $\Lo\psi\in L^2(D)$ which follows by taking $m\in \N$ large enough in \eqref{eq:Ccinfty coef bound}. Hence, by Proposition \ref{p:LGDf=-f} we have a.e. in $D$
		\begin{align*}
			\psi=\LoI(\Lo\psi)=\GDFI(\Lo\psi).
		\end{align*}
		Thus, by using Lemma \ref{l:rate of GDdelta} and Lemma \ref{l: Lo psi < C delta}, Fubini's theorem gives us
		\begin{align*}
			\langle \wLo \GDFI\mu,\psi\rangle&= \langle \GDFI\mu,\Lo\psi\rangle\\
			&=\int_D\left(\int_D\GDFI(x,y)\mu(dy)\right)\Lo\psi(x)dx\\&=\int_D\left(\int_D\GDFI(x,y)\Lo\psi(x)dx\right)\mu(dy)=\int_D\psi(y)\mu(dy).
		\end{align*}
	\end{proof}
	
	The following proposition connects the spectral, the distributional,  and the pointwise definition of $\Lo$ for nice enough functions.
	\begin{prop}\label{p:Lo various ways}
		If $u \in C^{1,1}(D)\cap H_\phi(D)$, then
		\begin{align*}
			\Lo u=\wLo u=\pLo u
		\end{align*}
		holds a.e. in $D$.
	\end{prop}
	\begin{proof}
		Let $u \in C^{1,1}(D)\cap H_\phi(D)$. Recall that $H_\phi(D)=\GDFI\big(L^2(D)\big)\subset L^2(D)\subset \LLL$, so $u=\GDFI h$ for some $h\in L^2(D)$, and $\Lo u=h$. However, $u\in C^{1,1}(D)$ so $\wLo u=\pLo u$ by Remark \ref{r:Lo weak and point}, and  $\wLo u=h$ by Proposition \ref{p:wLo GDF=f no.2}.
	\end{proof}
	
	\subsubsection{Regularity of Green potentials}
	In  the  two following claims we deal with  the  regularity properties of $\GDFI f$. The first claim says that Green potentials are continuous and this fact is rather simple to see and prove. We also prove that the Green potential of a $C_c^\infty(D)$ function is a $C^{1,1}(\overline{D})$ function, i.e. we prove a smoothness result for a specific class of functions.
	
	\begin{prop}\label{p:continuity of GDf}
		If $f\in\LLL \cap L^\infty_{loc}(D)$, then $\GDFI f\in C(D)$.
	\end{prop}
	\begin{proof}
		Let $x\in D$, $\eta\in (0,\de(x)/2)$ and $(x_n)_n\subset D$ such that $x_n\to x$ and $|x_n-x|<\eta/2$, $n\in\N$. We have
		\begin{align}
			|\GDFI f(x_n)-\GDFI f(x)|&\le\int_{D}|\GDFI(x_n,y)-\GDFI(x,y)||f(y)|dy\nonumber\\
			&\le\int_{D\cap B(x,\eta)^c}|\GDFI(x_n,y)-\GDFI(x,y)||f(y)|dy\label{eq:c Gdf 1}\\
			&\qquad\qquad+\int_{B(x,\eta)}\GDFI(x_n,y)|f(y)|dy\label{eq:c Gdf 2}\\
			&\qquad\qquad\qquad+\int_{B(x,\eta)}\GDFI(x,y)|f(y)|dy.\label{eq:c Gdf 3}
		\end{align}
		The first integral \eqref{eq:c Gdf 1} goes to 0 as $n\to\infty$ by the dominated convergence theorem   since $\GDFI$ is continuous, $f\in \LLL$, and since the bound \eqref{eq:Green function sharp estimate} holds.
		
		For the integrals \eqref{eq:c Gdf 2} and \eqref{eq:c Gdf 3}  note that $M\coloneqq \sup_{y\in B(x,\de(x)/2)}|f(y)|<\infty$ since $f\in L^\infty_{loc}(D)$. Further, by \eqref{eq:Green function sharp estimate} for all $w\in B(x,\eta/2)$ we have
		\begin{align}\label{eq:cont GDf eq1}
			\int_{B(x,\eta)}\GDFI(w,y)|f(y)|dy&\le c_1M\int_{B(w,\frac32 \eta)}\frac{1}{|w-y|^{d}\phi(|w-y|^{-2})}dy\notag\\
			&\le c_2 M \int_0^{\frac32 \eta}\frac{dr}{r\phi(r^{-2})}\le c_3M\int_0^{\frac32 \eta}\frac{\phi'(r^{-2})}{r^3\phi(r^{-2})^2}=\frac{c_3 M}{\phi(\frac{4}{9\eta^2})},
		\end{align}
		where in the last equality we used the substitution $t=\phi(r^{-2})$ and $c_3=c_3(d,D,\phi)>0$.
		Thus, the second and the third integral can be made arbitrarily small.
	\end{proof}
	\begin{rem}\label{r:equicont GDf}
		From Proposition \ref{p:continuity of GDf} it follows that
		\begin{align}\label{eq:uniform of GDf}
			\lim_{\xi \to x}\int_D|\GDFI(\xi,y)-\GDFI(x,y)||f(y)|dy=0,
		\end{align}
		uniformly on compact subsets of $D$.
		
		Indeed, fix a compact set $K\subset D$ and $\varepsilon>0$. First choose $\eta>0$ from Proposition \ref{p:continuity of GDf} such that $\dist(K,\partial D)>2\eta$ and ${(c_3M)}/\phi(\frac{4}{9\eta^2})<\varepsilon/3$, where $M=\sup_{y\in K+B(0,\eta)}|f(y)|$, see \eqref{eq:cont GDf eq1}. Thus, we tamed the integrals \eqref{eq:c Gdf 2} and \eqref{eq:c Gdf 3}.
		For the integral \eqref{eq:c Gdf 1} notice that the convergence $\lim_{\xi\to x}\GDFI(\xi,y)=\GDFI(x,y)$ is uniform in $x\in K$ and $y\in D\cap B(x,\eta)^c$ since $\GDFI$ is jointly continuous and since $\GDFI$ continuously vanishes at the boundary by \eqref{eq:Green function sharp estimate}. Hence, \eqref{eq:uniform of GDf} holds uniformly on compact sets.
	\end{rem}
	
	\begin{prop}\label{p:GDftwice diff}
		If $f\in C_c^\infty(D)$, then  $\GDFI f\in C^{1,1}(\overline D)$.
	\end{prop}
	\begin{proof}
		By Proposition \ref{p:LGDf=-f} we have $\GDFI f=\sum_{j=1}^\infty \frac1{\phi(\lambda_j)}\wh f_j \varphi_j$ a.e. in $D$. However, $\GDFI f\in C(D)$ by Proposition \ref{p:continuity of GDf}. Also, recall that there is $c_1=c_1(m,f)>0$ such that $|\wh f_j|\le c_1\lambda_j^{-m}$, $j\in\N$, by \eqref{eq:Ccinfty coef bound}, hence in the light of \eqref{eq:eigen estimates} and \eqref{eq:eigen estimate 1}, for large enough $m\in \N$ we have
		\begin{align*}
			\left\|\sum_{j=1}^\infty \frac1{\phi(\lambda_j)}\wh f_j \varphi_j\right\|_{C^{1,1}(\overline D)}\le \sum_{j=1}^\infty\frac{c_2}{\phi(\lambda_j)\lambda_j^m}(1+\lambda_j)^{d/4+1}<\infty
		\end{align*}
		by \eqref{eq:Weyl's law} and by \eqref{eq:simple global scaling}, where $c_2=c_2(d,D,m,f)>0$.
		
		In other words,  $\GDFI f=\sum_{j=1}^\infty \frac1{\phi(\lambda_j)}\wh f_j \varphi_j$ everywhere in $D$ and $\GDFI f\in C^{1,1}(\overline D)$.
	\end{proof}
	
	\subsection{Poisson kernel and harmonic functions}
	Recall that the Poisson kernel of the Brownian motion (i.e. of the Dirichlet Laplacian) can be defined as 
	\begin{align}\label{eq:Poisson Dirichlet Laplacian}
		P_{D}(x,z)=-\frac{\partial}{\partial \mathbf{n}}\GD(x,z),\quad x\in D,z\in\partial D,
	\end{align}
	 since we assume that $D$ is a $C^{1,1}$ bounded domain, 
	see \cite[Section 2.2.4]{evans_pde}. Here $\frac{\partial}{\partial \mathbf{n}}$ denotes the derivate in the direction of the  outer  normal. In this subsection we study the Poisson kernel of the process $X$ which we define as the normal derivative of the Green kernel of the process $X$ and we study harmonic functions relative to $\Lo$, or, as we show at the end of the subsection, relative to $X$.
	
	\begin{prop}\label{p:Poisson kernel definition}
		The function
		\begin{align}\label{eq:Poisson kernel definition}
			\PDfi(x,z)\coloneqq -\frac{\partial}{\partial {\mathbf{n}}}\GDfi(x,z),\quad x\in D,z\in\partial D,
		\end{align}
		is well defined and $(x,z)\mapsto\PDfi(x,z)\in C(D\times \partial D)$. Moreover,
		\begin{align}\label{eq:Poisson sharp bounds}
			\PDfi(x,z)\asymp \frac{\de(x)}{|x-z|^{d+2}\phi(|x-z|^{-2})},\quad x\in D,z\in\partial D,
		\end{align}
		where the constant of comparability depends only on $d$, $D$ and $\phi$. Finally, it holds that
		\begin{align}\label{eq:Green Poisson identity}
			\int_D \GDfz(x,\xi)\PDfi(\xi,z)d\xi=P_{D}(x,z),\quad x\in D,z\in\partial D.
		\end{align}	
	\end{prop}
	\begin{proof}
		Let $x\in D$ and $z\in\partial D$. For $y\in D$ we have
		\begin{align*}
			\frac{\GDfi (x,y)}{\de(y)}=\int_0^\infty \frac{1}{\de(y)}p_D(t,x,y)\uu(t)dt.
		\end{align*}
		In what follows, we always consider $y\in D$ which is in the direction of the normal derivative  in  $z$,  close enough to $z$ so that $\de(x)\le 2|x-y|$.
		
		Recall that $p_D\in C^1((0,\infty)\times \overline D\times \overline D)$ since $D$ is $C^{1,1}$, see Lemma \ref{apx:l:regularity of heat kernel}, hence $-\frac{\partial}{\partial \mathbf{n}}p_D(t,x,z)=\lim\limits_{y\to z}\frac{p_D(t,x,y)}{\de(y)}$ exists. Further, from \eqref{eq:heat kernel estimate}, there exist constants $c_1,c_2>0$ (depending on $D$) such that for all $t>0$, and $x,y\in D$ we have 
		\begin{align}\label{eq:eq1 Poisson estimate}
			\frac{p_D(t,x,y)\uu(t)}{\de(y)}\le c_1 \frac{\de(x)}{t^{d/2+1}}e^{-\frac{c_2|x-y|^2}{t}}\uu(t)\le  c_1\frac{\de(x)}{t^{d/2+1}}e^{-\frac{c_2\de(x)^2}{4t}}\uu(t).
		\end{align}
		Recall that $\uu$ is decreasing and that $\int_0^1\uu(t)dt<\infty$, hence the right hand side of \eqref{eq:eq1 Poisson estimate} is in $L^1\big((0,\infty),dt\big)$. By using  the dominated convergence theorem we conclude  that $\PDfi(x,z)$ is well defined and
		\begin{align*}
			\PDfi(x,z)=\lim_{y\to z}\frac{\GDfi (x,y)}{\de(y)}=-\int_0^\infty\frac{\partial}{\partial \mathbf{n}}p_D(t,x,z)\uu(t)dt.
		\end{align*}
		Moreover, \eqref{eq:Poisson sharp bounds} immediately follows from the definition of $\PDfi$ and \eqref{eq:Green function sharp estimate}.
		
		Now we show that $\PDfi$ is jointly continuous on $D\times\partial D$. Let $(x_n)_n\subset D$ such that $x_n\to x\in D$ and such that $\de(x_n)\ge \de(x)/2$. Also, take $(z_n)_{n}\subset \partial D$ such that $z_n\to z\in\partial D$. By taking the limit $y\to z$ in the first inequality in \eqref{eq:eq1 Poisson estimate} without the term $\uu(t)$, we obtain for all $n\in\N$ and all $t\in(0,\infty)$
		\begin{align}\label{eq:eq2 Poisson estimate}
			0\le -\frac{\partial}{\partial \mathbf{n}}p_D(t,x_n,z_n)\le c_1 \frac{\de(x_n)}{t^{d/2+1}}e^{-\frac{c_2|x_n-z_n|^2}{t}}\le c_1\frac{\de(x_n)}{t^{d/2+1}}e^{-\frac{c_2\de(x_n)^2}{t}},
		\end{align}
		which also holds for $z$ instead of $z_n$. Since $\frac{\partial}{\partial \mathbf{n}}p_D(t,x,z)\in C((0,\infty)\times \overline D\times \partial D)$, see Lemma \ref{apx:l:regularity of heat kernel}, by using the dominated convergence theorem with the bound derived from \eqref{eq:eq2 Poisson estimate} we get
		\begin{align*}
			|\PDfi(x,z)-\PDfi(x_n,z_n)|\le\int_0^\infty\left|\frac{\partial}{\partial \mathbf{n}}p_D(t,x_n,z_n)-\frac{\partial}{\partial \mathbf{n}}p_D(t,x,z)\right|\uu(t)dt\to 0,\quad \textrm{as }n\to\infty.
		\end{align*}

		We are left to prove \eqref{eq:Green Poisson identity}. Obviously, Lemma \ref{l: GDf(GDf*)=GD} implies
		\begin{align*}
			-\frac{\partial}{\partial \mathbf{n}}\left(\int_D \GDfz(x,\xi)\GDfi(\xi,\cdot)d\xi\right)(z)=P_{D}(x,z),\quad x\in D,z\in\partial D.
		\end{align*}
		We need to justify that the normal derivative can go inside the integral. To this end, let $x\in D$, $z\in\partial D$, and $\varepsilon>0$ such that $\de(x)>3\varepsilon$. Again, we only consider $y\in D$ which is in the direction of the normal derivative. For $|z-y|\le \varepsilon/2$ we have
		\begin{align*}
			\int_D \GDfz(x,\xi)\frac{\GDfi(\xi,y)}{\de(y)}d\xi&=	\int_{D\cap B(z,\varepsilon)^c} \GDfz(x,\xi)\frac{\GDfi(\xi,y)}{\de(y)}d\xi+\int_{D\cap B(z,\varepsilon)} \GDfz(x,\xi)\frac{\GDfi(\xi,y)}{\de(y)}d\xi\\
			&\eqqcolon I_1+I_2.
		\end{align*}
		For the integral $I_1$ by the sharp bounds \eqref{eq:Green function sharp estimate} we have
		\begin{align}\label{eq:eq3 Green over delta}
			\frac{\GDfi(\xi,y)}{\de(y)}\lesssim\frac{\de(\xi)}{|\xi-y|^{d+2}\phi(|\xi-y|^{-2})}.
		\end{align}
		Thus, if $\xi \in D\cap B(z,\varepsilon)^c$, we have $\frac{\GDfi(\xi,y)}{\de(y)}\le  c_3\de(\xi)$, where $c_3=c_3(\phi,D,d,\varepsilon)>0$. Further, $\GDfz\de\asymp \de$ by Lemma \ref{l:rate of GDdelta}, hence the integral $I_1$ converges to $$\int_{D\cap B(y,\varepsilon)^c}\GDfz(x,\xi) \PDfi(\xi,z)d\xi,$$ as $y\to z$.
		
		The integral $I_2$ we break  into  two additional integrals
		\begin{align*}
			I_2&=\int_{D\cap B(z,\varepsilon)} \GDfz(x,\xi)\frac{\GDfi(\xi,y)}{\de(y)}d\xi\\
			&\le \int_{B\left(y,\frac{\de(y)}{2}\right)}\GDfz(x,\xi)\frac{\GDfi(\xi,y)}{\de(y)}d\xi+\int_{D\cap B\left(y,\frac{\de(y)}{2}\right)^{c}\cap B(y,2\varepsilon)}\GDfz(x,\xi)\frac{\GDfi(\xi,y)}{\de(y)}d\xi\\&\eqqcolon J_1+J_2.
		\end{align*}
		Recall that $3\varepsilon\le \de(x)$ so $\frac16 |x-z|\le |x-\xi|\le 2|x-z|$ for all $\xi \in B(y,2\varepsilon)$. Hence, \eqref{eq:Green function sharp estimate} applied on $\GDfz$ implies
		\begin{align}\label{eq:GDFIZ upper bound near boundary}
			\GDfz(x,\xi)\le c_4 \de(\xi), \quad \xi \in B(y,2\varepsilon)\cap D,
		\end{align}
		where $c_4=c_4(d,D,\phi^*,|x-z|)>0$ and is independent of $\varepsilon$ in the sense if $\varepsilon\to0$, the constant $c_4$ remains the same.
		
		For $J_1$ note that $\de(\xi)\le \frac32\de(y)$ for $\xi\in B(y,\de(y)/2)$ so by using the bounds \eqref{eq:Green function sharp estimate} and \eqref{eq:GDFIZ upper bound near boundary} we have
		\begin{align*}
			J_1&\le c_5  \int_{B\left(y,\frac{\de(y)}{2}\right)}\frac{\de(\xi)}{\de(y)}\frac{1}{|\xi-y|^{d}\phi(|\xi-y|^{-2})}\le c_6 \int_{B\left(y,\frac{\de(y)}{2}\right)}\frac{1}{|\xi-y|^{d}\phi(|\xi-y|^{-2})}\\
			&\le c_7 \int_0^{\de(y)/2}\frac{\phi'(r^{-2})}{r^3\phi(r^{-2})^2}dr\le c_8 \frac{1}{\phi(4/\de(y)^2)},
		\end{align*}
		where $c_8$ is independent of $y$ and $\varepsilon$. In the second to last inequality we used \eqref{eq:scaling and the derivative} and for the last one we used the substitution $t=\phi(r^{-2})$.
		
		For $J_2$ note that $\de(\xi)\le \de(y)+|y-\xi|\le 3|\xi-y|$, for $\xi\in B(y,\de(y)/2)^c$, hence by the sharp bounds \eqref{eq:Green function sharp estimate} we have
		\begin{align*}
			J_2\le c_{9}\int_{B\left(y,\frac{\de(y)}{2}\right)^c\cap B(y,2\varepsilon)}\frac{1}{|\xi-y|^{d}\phi(|\xi-y|^{-2})}\le c_{10}\int_{\de(y)/2}^{2\varepsilon}\frac{\phi'(r^{-2})}{r^3\phi(r^{-2})^2}dr\le c_{11} \frac{1}{\phi(\frac{1}{4\varepsilon^2})},
		\end{align*}
		where $c_{11}$ is independent of $y$ and $\varepsilon$. Hence, for sufficiently small $\varepsilon$ the integral $I_2$ can be made sufficiently small. Thus, \eqref{eq:Green Poisson identity} holds.
	\end{proof}

	\begin{rem}
		We emphasize that the assumption of the regularity of $\partial D$ was essential in the proof of the previous proposition. Prior to the proof, regularity was used for obtaining the sharp bounds for $\GDFI$ and $\GDfz$ and for proving the regularity of $p_D$. This led to showing the well-definiteness of $\PDFI$, to the sharp bounds for $\PDFI$, and to the identity \eqref{eq:Green Poisson identity}. In the remainder of the section, the regularity of $\partial D$ will be also heavily used but we omit comments like this one from now on. 
	\end{rem}

	Now we deal with harmonic functions with respect to the operator $\Lo$. Our first goal is to show the integral representation of positive harmonic functions which we show in Theorem \ref{t:PDFI harmonic}. After that, in Theorem \ref{t:PD conti} we show the continuity of harmonic functions and at the end of the subsection we connect harmonic functions with functions that satisfy  a  certain mean-value property with respect to $X$, see Theorem \ref{t:harmonic functions}.
	
	\begin{defn}\label{d:harmonic function}
		A function $h\in \LLL$ is called harmonic in $D$ if $\wLo h=0$ in $D$.
	\end{defn}
	First,  we  present  a connection between harmonic functions  and  classical harmonic functions.
	\begin{prop}\label{p:Poisson kernel harmonic}
		A function $h\in \LLL$ is harmonic in $D$ if and only if $\GDfz h$ is a classical harmonic function in $D$. In particular, for every $z\in\partial D$, the function $x\mapsto \PDfi(x,z)$ is harmonic in $D$.
	\end{prop}
	\begin{proof}
		The first part of the claim follows by the following calculation. Take $\psi \in C_c^\infty(D)$. Then by using Lemma \ref{l: Delta=phi(Delta)phi(Delta)}, Proposition \ref{p:LGDf=-f}, and Fubini's theorem we have
		\begin{align*}
			\int_D h(x)\Lo \psi(x)dx&=\int_D h(x)\left[\Loz^{-1}\circ (-\Delta)\psi(x)\right]dx\\
			&=\int_D h(x)\GDfz((-\Delta)\psi)(x)dx\\
			&=-\int_D \GDfz h(x)\,\Delta\psi(x)dx,
		\end{align*}
		i.e. $h$ is harmonic if and only if $\GDfz h$ is a classical harmonic function in $D$.
		
		If $z\in\partial D$, then $\PDFI(\cdot,z)\in \LLL$ by the bound \eqref{eq:Poisson sharp bounds}, see also the beginning of the proof of Theorem \ref{t:PDFI harmonic} with $\zeta=\delta_z$. The second claim now follows from \eqref{eq:Green Poisson identity} and the fact that the kernel $\PD(\cdot,z)$ is  a  classical harmonic function.
	\end{proof}
	
	\begin{thm}\label{t:PDFI harmonic}
		If a non-negative function $h\in \LLL$ is harmonic in $D$, then there exists a finite non-negative  measure $\zeta\in\MM(\partial D)$ such that 
		\begin{align}\label{eq:representation of haramonic}
			h(x)=\int_{\partial D}\PDfi(x,z)\zeta(dz),\quad \text{for a.e. }x\in D.
		\end{align}
		Moreover, there is $C=C(d,D,\phi)>0$ such that 
		\begin{align}\label{eq:harmonic L1 bound}
			\|h\|_{\LLL}\le C\|\zeta\|_{\MM(\partial D)}.
		\end{align}
		Conversely, every function of the form \eqref{eq:representation of haramonic} is harmonic in $D$.
	\end{thm}
	\begin{proof}
		Let $h$ be represented as \eqref{eq:representation of haramonic}. Since $\PDfi(x,\cdot)\in C(\partial D)$ for fixed $x\in D$ by Proposition \ref{p:Poisson kernel definition}, hence bounded, the function $h$ is well defined. Further, since $\de(x)\le |x-z|$, $z\in \partial D$, from \eqref{eq:Poisson sharp bounds} and Fubini's theorem we get
		\begin{align*}
			\int_D h(x)\de(x)dx&\le c_1\int_{\partial D}\int_D \frac{\de(x)^2}{|x-z|^{d+2}}\frac{1}{\phi(|x-z|^{-2})}dx\,\zeta(dz)\\
			&\le c_1 \int_{\partial D}\int_{B(z,\diam D)}\frac{1}{|x-z|^{d}\phi(|x-z|^{-2})}dx\,\zeta(dz)\\
			&\le c_2\int_{\partial D}\frac{\zeta(dz)}{\phi(\diam D^{-2})}<\infty,
		\end{align*}
		where $c_2=c_2(d,D,\phi)>0$, i.e. $h\in\LLL$ and $\|h\|_{\LLL}\le C\|\zeta\|_{\MM(\partial D)}$. Take now $\psi \in C_c^\infty(D)$. Fubini's theorem and Proposition \ref{p:Poisson kernel harmonic} yield
		\begin{align*}
			\int_D \PDfi\zeta(x)\Lo\psi(x)dx=\int_{\partial D}\left(\int_D\PDfi(x,z)\Lo\psi(x)dx\right)\zeta(dz)=0,
		\end{align*}
		i.e. $h$ is harmonic in $D$.
		
		Conversely, let $h$ be a non-negative harmonic function in $D$. Then $\GDfz h$ is a classical non-negative harmonic function in $D$ by Proposition \ref{p:Poisson kernel harmonic}. By the representation of  non-negative classical harmonic functions there is a non-negative finite measure $\zeta\in\MM(\partial D)$ such that
		\begin{align}\label{eq:representation eq0}
			\GDfz h(x)=\int_{\partial D}P_{D}(x,z)\zeta(dz),\quad \text{ for a.e. $x\in D$}.
		\end{align}
		Applying \eqref{eq:Green Poisson identity} to the right hand side of \eqref{eq:representation eq0} we get 
		\begin{align}\label{eq:representation eq1}
			\int_D \GDfz(x,\xi)h(\xi)d\xi=\int_D\GDfz(x,\xi)\left[\int_{\partial D}\PDfi(\xi,z)\zeta(dz)\right]d\xi, \quad \text{ for a.e. $x\in D$}.
		\end{align}
		By using Proposition \ref{p:wLo GDF=f no.2} in \eqref{eq:representation eq1} we obtain
		\begin{align*}
			h(\xi)=\int_{\partial D}\PDfi(\xi,z)\zeta(dz),\quad \text{for  a.e. $\xi$ in $D$.}
		\end{align*}
	\end{proof}
	
	Motivated by the previous theorem, we introduce the definition of the Poisson integral.
	\begin{defn}\label{d:Poisson integral}
		For a finite signed measure $\zeta\in \MM(\partial D)$ we define the Poisson integral of $\zeta$ by
		\begin{align*}
			\PDFI\zeta(x)\coloneqq \int_{\partial D}\PDFI(x,z)\zeta(dz),\quad x\in D.
		\end{align*}
	\end{defn}
	Note that  the  finiteness of the (signed) measure $\zeta$ in the previous definition is  a  necessary and sufficient condition  for  for the integral defining $\PDFI \zeta$  to be  finite, see \eqref{eq:Poisson sharp bounds}. If $\zeta\in L^1(\partial D)$, we slightly abuse the notation in Definition \ref{d:Poisson integral} where we set $\PDFI \zeta(x)=\int_{\partial D}\PDFI(x,z)\zeta(z)\sigma(dz)$, where $\sigma$ is the $d-1$ dimensional  Hausdorff measure on $\partial D$. Since the set $D$ is $C^{1,1}$, the measure $\sigma$ is finite so we can define the Poisson integral of $\sigma$
	\begin{align}\label{eq:defn of PDFIsigma}
		\PDfi\sigma(x)= \int_{\partial D}\PDfi(x,z)\sigma(dz), \quad x\in D,
	\end{align}
	which will be of great importance for the boundary condition of the semilinear problem.
	
	We finish the subsection  with  two properties of harmonic functions of the form $\PDFI\zeta$.
	\begin{thm}\label{t:PD conti}
		A non-negative harmonic function in $D$ is continuous in $D$ (after a modification on the Lebesgue null set). Furthermore,  for every finite (signed) measure $\zeta\in\MM(\partial D)$, we have $\PDFI\zeta\in C(D)$.
	\end{thm}
	\begin{proof}
		Let $h\in \LLL$ be a non-negative harmonic function in $D$. 
		By Theorem 	\ref{t:PDFI harmonic} there exists a finite non-negative  measure $\zeta\in\MM(\partial D)$ such that $h=\PDFI\zeta$ a.e. in $D$. In Proposition \ref{p:Poisson kernel definition} it was proved that the function $\PDFI(\cdot,\cdot)$ is continuous in the first variable and that the sharp bounds \eqref{eq:Poisson sharp bounds} hold, so we can use the dominated convergence theorem to get $\PDFI\zeta \in C(D)$. 
	\end{proof}
	
	In the theory of Markov processes, harmonicity of a function is considered relative to the process itself, i.e. it is said that a function $f:D\to[-\infty,\infty]$ is harmonic in $D$ with respect to $X$ if for every $U\subsub D$ and $x\in U$
	\begin{align}\label{eq:harmonic function}
		h(x)=\ex_x[h(X_{\tau^X_U})]
	\end{align}
	holds, where $\tau^X_U=\inf\{t>0:X_t\notin U\}$ and where we implicitly assume $\ex_x[|h(X_{\tau^X_U})|]<\infty$ for every $x\in U\subsub D$. The relation \eqref{eq:harmonic function} is often referred to as the mean-value property of the function $f$ with respect to $X$. In order not to confuse, if $f$ is harmonic in $D$ with respect to $X$, we will say that $f$ satisfies the mean-value property with respect to $X$. We note that $\ex_x[|h(X_{\tau^X_U})|]<\infty$ for every $x\in U\subsub D$ implies that $f\in \LLL$, see the proof of \cite[Lemma 3.6]{ksv_minimal2016} where instead of the inequality $U^{D,B}(x,y)\le G_X(x,y)$ use $U^{D,B}(x,y)\le \GDFI(x,y)$.
	
	The connection between non-negative functions that satisfy the mean-value property with respect to $X$ and non-negative functions that satisfy the mean-value property with respect to $W^D$ is known due to \cite[Theorem 3.6]{song_vondra_JTP2006} which we cite in the next claim.
	
	\begin{thm}\label{t:harmonic and classical}
		If a non-negative function $h$ satisfies the mean-value property in $D$ with respect to $X$, then $s\coloneqq \GDfz h$ satisfies the mean-value property in $D$ with respect to $W^D$. Conversely, if a non-negative  function  $s$ satisfies the mean-value property in $D$ with respect to $W^D$, then 
		\begin{align}\label{eq:harmonic pointwise rep}
			h(x)\coloneqq\int_0^\infty \left(s(x)-P_t^Ds(x)\right)\nu(t)dt=\pLoz s(x), \quad x\in D,
		\end{align}
		satisfies the mean-value property in $D$ with respect to $X$, $h$ is continuous and $\GDfz h=s$.
	\end{thm}
	\begin{proof}
		Everything follows from \cite[Theorem 3.6]{song_vondra_JTP2006} except the second equality in \eqref{eq:harmonic pointwise rep}. To finish the proof, it follows from the proof of \cite[Lemma 3.4]{song_vondra_JTP2006} that
		$$|s(x)-P_t^Ds(x)|\le c(1\wedge t),\quad x\in K,$$
		where $K$ is any compact subset of $D$ and $c=c(d,D,s_{|K})>0$. Also, $s\in C^\infty(D)$ since it is a classical harmonic function so by the same calculations as in Lemma \ref{l:Lo pointwisely} we get that
		$$\int_0^\infty \left(s(x)-P_t^Ds(x)\right)\nu(t)dt=\pLoz s.$$
	\end{proof}
	
	The following theorem says that non-negative harmonic functions and non-negative functions with the  mean-value property with respect to $X$ are essentially the same.
	\begin{thm}\label{t:harmonic functions}
		If a non-negative function $h\in \LLL$ is harmonic in $D$, then (after a modification on the Lebesgue null set) $h$ satisfies the mean-value property with respect to $X$. Conversely, if $h\ge0$ satisfies the mean-value property with respect to $X$, then $h$ is harmonic in $D$.
	\end{thm}
	\begin{proof}
		Let $h\ge0$ be harmonic in $D$. Theorem \ref{t:PDFI harmonic} implies that we can modify $h$ such that $h=\PDFI\zeta$ in the whole $D$ for some non-negative and finite $\zeta\in \MM(\partial D)$. This also means that $h\in C(D)$ by Theorem \ref{t:PD conti}. Since $\GDfz h=\PD\zeta$ in $D$ by \eqref{eq:Green Poisson identity}, the claim follows from Theorem \ref{t:harmonic and classical} because $P_D\zeta$ is a (smooth) classical harmonic function, hence it satisfies the mean-value property with respect to $W^D$.
		
		Conversely, if $h\ge 0$ satisfies the mean-value property with respect to $X$, then $\GDfz h$ satisfies the mean-value property with respect to $W^D$ by Theorem \ref{t:harmonic and classical}. By the classical theory of harmonic functions, $\GDfz h$ is a classical harmonic function in $D$. Proposition \ref{p:Poisson kernel harmonic} now implies that $h$ is harmonic in $D$.
	\end{proof}
	
	\section{Boundary behaviour of potential integrals}\label{s:boundary behaviour}
	In this section we study the boundary behaviour of Poisson and Green integrals which will serve as a foundation for  the  understanding of the boundary condition of the (semi)linear problem and  the  understanding of the connection between weak and distributional solutions in the next section. However, these problems are also interesting in  themselves.  We emphasize that the essential assumption for all the results in this section is that $D$ is a $C^{1,1}$ bounded domain and the regularity of $\partial D$ is heavily used in every proof in this section.  First we give a sharp bound for $\PDFI \sigma$, 
	\begin{lem}\label{l:Poisson integral sharp bounds}
		It holds that
		\begin{align}\label{eq:Poisson integral sharp bounds}
			\PDfi\sigma(x)\asymp \frac{1}{\de(x)^2\phi(\de(x)^{-2})},\quad x\in D,
		\end{align}
		where the constant of comparability depends only on $d$,  $D$ and $\phi$.
	\end{lem}
	
	\begin{proof}
		In Proposition \ref{p:Poisson kernel definition} we have proved that
		$$ \PDfi(x,z)\asymp \frac{\de(x)}{|x-z|^{d+2}\phi(|x-z|^{-2})},\quad x\in D,z\in\partial D,$$
		where the constant of comparability depends only on $d$,  $D$ and $\phi$. Also, in the following calculations, it is easy to check that every comparability constant remains to depend only on $d$, $D$ and $\phi$.
		
		For the upper bound, note that $\de(x)\le|x-z|$, $z\in \partial D$ so by using \eqref{eq:simple global scaling} we have $\de(x)^{2}\phi(\de(x)^{-2})\le |x-z|^2\phi(|x-z|^{-2})$, thus
		\begin{align*}
			\PDfi\sigma(x)&\asymp \int_{\partial D}\frac{\de(x)}{|x-z|^{d+2}\phi(|x-z|^{-2})}\sigma(dz)\le \frac{1}{\de(x)^{2}\phi(\de(x)^{-2})},
		\end{align*}
		since $\int_{\partial D}|x-z|^{-d}\asymp \de(x)^{-1}$, $x\in D$.
		
		For the lower bound fix $x\in D$ and choose $\Gamma=\{z\in\partial D: |x-z|\le 2\de(x)\}$. Recall that $\phi$ is increasing so
		\begin{align*}
			\PDfi\sigma(x)&\asymp \int_{\partial D}\frac{\de(x)}{|x-z|^{d+2}\phi(|x-z|^{-2})}\sigma(dz)\ge \frac{1}{4\de(x)^{2}\phi(\de(x)^{-2})}\de(x)\int_{\Gamma}\frac{\sigma(dz)}{|x-z|^d}\\
			&\asymp \frac{1}{\de(x)^{2}\phi(\de(x)^{-2})},
		\end{align*}
		since $\int_{\Gamma}|x-z|^{-d}\asymp \de(x)^{-1}$, $x\in D$, by reducing to the flat case, see Lemma \ref{apx:l:approx of surface cap}.
	\end{proof}
	\begin{rem}\label{r:relation to classical Poisson}
		
		For the classical Poisson kernel $\PD$, defined in \eqref{eq:Poisson Dirichlet Laplacian}, it is well known that $\PD(x,z)\asymp \frac{\de(x)}{|x-z|^d}$, for $x\in D$ and $z\in \partial D$. Moreover, since $\PD$ is the density of $W_{\tau_D}$, we have $\PD\sigma(x)=\ex_x[\1({W_{\tau_D}})]=1$. In particular, by the sharp bound \eqref{eq:Poisson integral sharp bounds} and by the scaling condition \eqref{eq:scaling condition}, $\PDFI\sigma$ explodes when approaching the boundary of $D$  whereas $\PD\sigma$ obviously does not.
	\end{rem}
	\begin{rem}\label{r:bound PDFIsigma over PDFI}
		In what follows we will need the following inequality
		\begin{align}\label{eq:bound PDFIsigma over PDFI}
			\frac{\PDFI(x,z)}{\PDFI\sigma(x)}\lesssim \frac{\de(x)}{|x-z|^d}, \quad x\in D,
		\end{align}
		which holds by the sharp bounds \eqref{eq:Poisson sharp bounds} and \eqref{eq:Poisson integral sharp bounds},  and since by \eqref{eq:simple global scaling} it holds that $\de(x)^{2}\phi(\de(x)^{-2})\le |x-z|^2\phi(|x-z|^{-2})$, for $x\in D$ and $z\in\partial D$.
	\end{rem}
	
	 The  two following propositions deal with the boundary behaviour of Poisson integrals. They generalize \cite[Proposition 25 \& Theorem  26]{aba+dupaig} to our more general non-local setting.
	
	\begin{prop}\label{p:boundary oper. of Poisson}
		Let $\zeta\in C(\partial D)$. It holds
		\begin{align*}
			\lim_{D\ni x\to z\in\partial D}\frac{\PDfi\zeta(x)}{\PDfi\sigma(x)}=\zeta(z)
		\end{align*}
		uniformly on $\partial D$.
	\end{prop}
	\begin{proof}
		Note that $\zeta$ is uniformly continuous since $D$ is bounded and let $M=2\sup_{z\in\partial D} |\zeta(z)|$. For $\varepsilon>0$ choose $\eta>0$ such that if $y,z\in\partial D$ and $|y-z|<\eta$, then $|\zeta(y)-\zeta(z)|\le \varepsilon$. For $z\in\partial D$ let $\Gamma_z=\{y\in\partial D:|y-z|<\eta\}$. Now if $|x-z|\le \frac\eta2$, then by using \eqref{eq:bound PDFIsigma over PDFI} we have
		\begin{align*}
			\left|\frac{\PDfi\zeta(x)}{\PDfi\sigma(x)}-\zeta(z)\right|&\le\frac{1}{\PDfi\sigma(x)}\int_{\partial D}\PDfi(x,y)\left|\zeta(y)-\zeta(z)\right|\sigma(dy)\\
			&\le c_1\de(x)\int_{\Gamma_z}\frac{\left|\zeta(y)-\zeta(z)\right|}{|x-y|^d}\sigma(dy)+c_1\de(x)\int_{\partial D\setminus \Gamma_z}\frac{\left|\zeta(y)-\zeta(z)\right|}{|x-y|^d}\sigma(dy)\\
			&\le c_2\varepsilon+c_1\de(x)M\sigma(\partial D)\left(\frac{\eta}{2}\right)^{-d},
		\end{align*}
		where in the last inequality for the first term we used $\de(x)\asymp \int_{\partial D}|x-y|^{-d}\sigma(dy)$, hence $c_2=c_2(d,D,\phi)>0$. Now the claim follows by taking $x$ close enough to $z$.
	\end{proof}
	
	\begin{prop}\label{p:boundary oper. Poisson L1}
		For any $\zeta\in L^1(\partial D)$
		and any $\varphi\in C(\overline\Omega)$ it holds that
		\begin{align*}
			\frac1t\int_{\{\de(x)\le t\}}
			\frac{\PDfi\zeta(x)}{\PDfi\sigma(x)}\varphi(x)dx
			\xrightarrow{t\downarrow 0} \int_{\partial D}\varphi(y)\zeta(y)d\sigma(y).
		\end{align*}
		
	\end{prop}
	\begin{proof}
		We can repeat the proof of \cite[Theorem 26]{aba+dupaig} almost to the letter. Indeed, take $\varphi\in C(\overline D)$ and note the $h_1$ of \cite{aba+dupaig} is our $\PDFI\sigma$, and $\phi$ of \cite{aba+dupaig} is our $\varphi$. We repeat the proof up to the definition of
		\begin{align*}
			\Phi(t,y):=\frac1t\int_{\{\de(x)<t\}}\frac{\PDFI(x,y)}{\PDFI\sigma(x)}\,\varphi(x)\:dx.
		\end{align*}
		Now we use Remark \ref{r:bound PDFIsigma over PDFI} and the boundedness of $\varphi$ to obtain
		\begin{align*}
			|\Phi(t,y)|\le c_1
			\frac{\|\varphi\|_{L^\infty(D)}}{t}
			\int_{\{\de(x)<t\}}\frac{\de(x)}{|x-y|^d}dx\le c_2,
		\end{align*}
		by the reduction to the flat boundary, see \cite[Lemma 40]{aba+dupaig}, where $c_2=c_2(\phi,D,d,\varphi)>0$. The rest of the proof is now the same as in \cite{aba+dupaig}.
	\end{proof}
	
	Now we turn to the boundary behaviour of Green integrals. Here the pointwise limits are harder to get and we must assume some kind of uniformity of the integrating function.
	\begin{thm}\label{t:boundary of GDFIU}
		Let $U:(0,\infty)\to[0,\infty)$ such that 
		\begin{enumerate}[label=(\textbf{U\arabic*})]
			\item \label{U1} integrability condition holds 
			\begin{equation}\label{eq:int-cond}
				\int_{0}^1U(t)t\, dt <\infty;
			\end{equation}
			
			\item \label{U2}  almost non-increasing condition holds, i.e. there exists  $C>0$  such that
			\begin{equation}\label{eq:wus-U}
				U(t)\le C U(s), \quad 0<s\le t\le 1;
			\end{equation}
			
			\item \label{U3}reverse doubling condition holds, i.e. there exists $C>0$   such that
			\begin{equation}\label{eq:doubling-U}
				U(t)\le C U(2t),\quad t\in (0,1); 
			\end{equation}
			\item \label{U4} boundedness away from zero holds, i.e. $U$ is bounded from above on $[c, \infty)$ for each $c >0$.
		\end{enumerate}
		Then $U(\de)\in\LLL$ and
		\begin{align}\label{eq:GDFI U sharp bound}		\GDFI\big(U(\de)\big)(x)&\asymp \frac{1}{\de(x)^2\phi(\de(x)^{-2})}\int\limits_0^{\de(x)}U(t)t\, dt+ \de(x)+\de(x)\int\limits_{\de(x)}^{\diam D} \frac{U(t)}{t^2\phi(t^{-2})}\, dt\,.
		\end{align}
		In particular,
		\begin{align}\label{eq:GD << PD}
			\lim_{D\ni x\to z\in\partial D}\frac{\GDFI [U(\de)](x)}{\PDFI\sigma(x)}=0.
		\end{align}
	\end{thm}
	This theorem generalizes \cite[Proposition 7]{dhifli2012} to more general non-local operators and more general functions since in \cite{dhifli2012} this result was proved in the case of the spectral fractional Laplacian and for functions of the form $U(t)=t^\beta$.
	\begin{proof}[\textbf{Proof of Theorem \ref{t:boundary of GDFIU}}]
		The proof of this claim is very similar to the proof of \cite[Proposition 4.1]{semilinear_bvw} so the details can be found in  the  Appendix in Section \ref{apx:proof of boundary GDFI}.

	\end{proof}
	
	The following proposition appears as \cite[Theorem 27]{aba+dupaig} for the case of the spectral fractional Laplacian but in our more general setting the proof gets a little more complicated, cf. \cite[Eq. (46)]{aba+dupaig} and \eqref{eq:boundary op. GDFI eq2}.
	\begin{prop}\label{p:boundary operator on GDFi}
		Let $\lambda\in\MM(D)$ such that $\int_D\de(x)|\lambda|(dx)<\infty$. Then
		\begin{align}\label{eq:boundary operator on GDFi}
			\frac1t \int_{\{\de(x)\le t\}}\frac{\GDFI\lambda(x)}{\PDFI\sigma(x)}\varphi(x)dx\xrightarrow{t\downarrow 0}0,\quad \varphi\in C(\overline{D}).
		\end{align}
	\end{prop}
	\begin{proof}
		Without loss of generality we may assume that $\lambda$ is a non-negative measure. It is enough to prove that \eqref{eq:boundary operator on GDFi} holds for $\varphi\equiv1$. By using Fubini's theorem it follows that
		\begin{align}\label{eq:boundary op. GDFI eq1}
			\frac1t \int_{\{\de(x)\le t\}}\frac{\GDFI\lambda(x)}{\PDFI\sigma(x)}dx=\int_D\left(\frac1t \int_{\{\de(x)\le t\}}\frac{\GDFI(x,y)}{\PDFI\sigma(x)}dx\right)\lambda(dy).
		\end{align}
		Lemma \ref{apx:l:away-from-bdry} for $U\equiv 1\gtrsim 1/\PDFI\sigma$ and Lemma \ref{apx:l:boundary operator of GDFI 2} imply that there is  $C=C(d,D,\phi)>0$ such that
		\begin{align}\label{eq:boundary op. GDFI eq2}
			\int_{\{\de(x)\le t\}}\frac{\GDFI(x,y)}{\PDFI\sigma(x)}dx\le \begin{cases}
				C t\de(y),&\de(y)< \frac{t}{2},\\
				C\,\wt f(y,t),&\de(y)\ge\frac{t}{2},
			\end{cases}
		\end{align}
		where $0\le \wt f(y,t)\le t\,\de(y)$ in $\{\de(y)\ge \frac t2\}$ and $f(y,t)/t\to0$ as $t\to 0$ for every $y\in D$.
		Hence, \eqref{eq:boundary op. GDFI eq1} and \eqref{eq:boundary op. GDFI eq2} imply
		\begin{align*}
			\frac1t \int_{\{\de(x)\le t\}}\frac{\GDFI\lambda(x)}{\PDFI\sigma(x)}dx\le C\int_{\{\de(y)< \frac{t}{2}\}}\de(y)\lambda(dy)+C\int_{\{\de(y)\ge \frac{t}{2}\}}\frac{\wt f(y,t)}{t}\lambda(dy)
		\end{align*}
		from which the claim of the lemma follows by using the dominated convergence theorem.
	\end{proof}
	
	\section{Linear Dirichlet problem}\label{s:linear problem}
	
	In this section we deal with a linear Dirichlet problem for $\Lo$ and develop some basic properties of a weak solution to the problem. At the end of the section, we connect the weak formulation of the problem with the distributional.
	\begin{defn}\label{d:problem definition}
		Let $\lambda\in \MM(D)$ and $\zeta\in \MM(\partial D)$ such that  
		\begin{align}\label{eq:condition on rate of solution}
			\int_D\de(x)|\lambda|(dx)+|\zeta|(\partial D)<\infty.
		\end{align}
		We say that $u\in L^1_{loc}(D)$ is a weak solution to the problem
		\begin{align}\label{linear problem}
			\begin{cases}
				\Lo u=\lambda,&\textrm{in $D$},\\
				\frac{u}{\PDFI\sigma}=\zeta,&\textrm{on $\partial D$},
			\end{cases}
		\end{align}
		if for every $\psi\in C_c^\infty(D)$ it holds that
		\begin{align}\label{eq:problem - integral definition}
			\int_D u(x)\psi(x)dx=\int_D\GDfi\psi(x)\lambda(dx)-\int_{\partial D}\frac{\partial}{\partial \mathbf{n}}\GDfi\psi(z)\zeta(dz).
		\end{align}
		If in \eqref{eq:problem - integral definition} we have $\le$ ($\ge$) instead of the equality and the inequality holds for every non-negative $\psi\in C_c^\infty(D)$, then we say $u$ is a weak subsolution (supersolution) to the problem \eqref{linear problem}.
	\end{defn}
	
	\begin{rem}
		\begin{enumerate}[label=(\alph*)]
			\item
			Let $\psi\in C_c^\infty(D)$. From the calculations in the proof of Proposition \ref{p:Poisson kernel definition}, see  also \eqref{eq:Poisson kernel definition} and \eqref{eq:Poisson sharp bounds}, it follows that $\frac{\partial}{\partial \mathbf{n}}\GDfi\psi(z)$ is well defined and $$-\frac{\partial}{\partial \mathbf{n}}\GDfi\psi(z)=\int_D\PDFI(y,z)\psi(y)dy, \quad z \in \partial D,$$
			holds, hence $\frac{\partial}{\partial \mathbf{n}}\GDfi\psi\in L^\infty(\partial D)$. Moreover, Lemma \ref{l:rate of GDdelta} implies that $|\GDfi\psi(x)|\lesssim \de(x)$, thus the condition \eqref{eq:condition on rate of solution} ensures that the integrals in \eqref{eq:problem - integral definition} are well defined.
			\item
			If $u$ is a solution to the linear problem \eqref{linear problem}, then by using Fubini's theorem in \eqref{eq:problem - integral definition} we get that
			\begin{align}\label{eq:solution linear problem}
				u=\GDfi\lambda+\PDFI\zeta,\quad \textrm{a.e. in $D$.}
			\end{align}
			This implies that $u\in \LLL$. Indeed, $\GDFI\lambda\in\LLL$ by Lemma \ref{l:rate of GDdelta}, and  $\PDFI\zeta\in \LLL$ by \eqref{eq:harmonic L1 bound}.
			
			Conversely, the function defined in \eqref{eq:solution linear problem} is the solution of linear problem \eqref{linear problem} which we also get by using Fubini's theorem in \eqref{eq:problem - integral definition}.
		\end{enumerate}
	\end{rem}
	The following theorem summarizes the previous remark.
	\begin{thm}\label{t:linear problem}
		Let $\lambda\in \MM(D)$ and $\zeta\in \MM(\partial D)$ such that  
		\eqref{eq:condition on rate of solution} holds. Then the linear problem \eqref{linear problem} has a unique weak solution $u$ for which it holds that $u\in\LLL$ and
		\begin{align*}
			u(x)=\GDFI\lambda(x)+\PDFI\zeta(x),\quad\text{for a.e. $x\in D$.}
		\end{align*}
		Furthermore, there is $C=C(d,D,\phi)>0$ such that
		\begin{align}\label{eq:linear solution norm boundedness}
			\|u\|_{\LLL}\le C\left(\int_D\de(x)|\lambda|(dx)+|\zeta|(\partial D)\right).
		\end{align}
	\end{thm}
	
	In the next corollary we bring a version of a maximum principle for the weak solution.
	\begin{cor}\label{c:maximum principle weak solution}
		Let $\lambda\in \MM(D)$ and $\zeta\in \MM(\partial D)$ such that  
		\eqref{eq:condition on rate of solution} holds. If $\lambda\ge0$ and $\zeta\ge0$, then the unique solution $u$ of the linear problem $\eqref{linear problem}$ satisfies $u\ge0$ a.e. in $D$.
	\end{cor}

	Now we connect the weak and the distributional  formulation of the Dirichlet problem. First,  we  define  the distributional solution.
	
	\begin{defn}\label{r:defn of distributional solution}
		We say that $u\in \LLL$ is a distributional solution of \eqref{linear problem} if for every $\psi\in C^\infty_c(D)$ it holds that
		\begin{align}\label{eq:distri solution}
			\int_D u(x)\Lo\psi(x)dx=\int_D \psi(x)\lambda(dx),
		\end{align}
		and if for every $\varphi \in C(\overline D)$ it holds that
		\begin{align}\label{eq:distri solution boundary}
			\lim_{t\downarrow 0}\frac1t \int_{\{\de(x)\le t\}}\frac{u(x)}{\PDFI\sigma(x)}\varphi(x)dx=\int_{\partial D}\varphi(z)\zeta(dz).
		\end{align}
	\end{defn}
	
	\begin{prop}
		Let $\lambda\in \MM(D)$ and $\zeta\in L^1(\partial D)$ such that \eqref{eq:condition on rate of solution} holds. Then the weak solution of \eqref{linear problem} is also a distributional solution of \eqref{linear problem}.
	\end{prop}
	\begin{proof}
		The weak solution is given by $u=\GDFI \lambda+\PDFI\zeta$ so the relation \eqref{eq:distri solution} follows from Proposition \ref{p:wLo GDF=f no.2} and Theorem \ref{t:PDFI harmonic}. The boundary condition \eqref{eq:distri solution boundary}  follows from Proposition \ref{p:boundary oper. Poisson L1} and Proposition \ref{p:boundary operator on GDFi}.
	\end{proof}

	\section{Semilinear Dirichlet problem}\label{s:semilinear problem}
	
	In this section we study the following semilinear problem.
	\begin{defn}\label{d:semilinear problem}
		Let $f:D\times \R\to\R$ and $\zeta\in \MM(\partial D)$ such that $|\zeta|(\partial D)<\infty$. We say that $u\in L^1_{loc}(D)$ is a weak solution to the problem
		\begin{align}\label{semilinear problem}
			\begin{cases}
				\Lo u(x)=f(x,u(x)),&\textrm{in $D$},\\
				\frac{u}{\PDFI\sigma}=\zeta,&\textrm{on $\partial D$},
			\end{cases}
		\end{align}
		if
		\begin{align}\label{eq:defn semilinear solution}
			\int\limits_D u(x)\psi(x)=\int\limits_D\GDFI\psi(x)f(x,u(x))dx-\int\limits_{\partial D}\frac{\partial}{\partial \mathbf{n}}\GDFI\psi(z)\zeta(dz),\quad \psi\in C_c^\infty(D).
		\end{align}
		If in the equation above we have $\le$ ($\ge$) instead of the equality and the inequality holds for every non-negative $\psi\in C_c^\infty(D)$, then we say $u$ is a weak subsolution (supersolution) to \eqref{semilinear problem}.
	\end{defn}
	Note that if $u$ is a solution to the semilinear problem \eqref{semilinear problem}, then it is implicitly assumed that $x\mapsto f(x,u(x))\in \LLL$ since only then the first integral in \eqref{eq:defn semilinear solution} is well defined. For the sake of brevity, we will frequently use the notation $f_u(x)\coloneqq f(x,u(x))$, $x\in D$, which is also known as  the  Nemytskii operator. Further, in the same way as in the linear case,  we can see that if $u$ is a weak solution of \eqref{semilinear problem}, then by Fubini's theorem used in \eqref{eq:defn semilinear solution} we get
	\begin{align}\label{eq:solution semilinear}
		u=\GDFI f_u+\PDFI \zeta.
	\end{align}
	Conversely, if $u$ satisfies \eqref{eq:solution semilinear}, then  $u$ is a weak solution of \eqref{semilinear problem}.
	
	In the following subsection we prove Kato's inequality in our setting. This will help us to obtain existence and uniqueness results for various nonlinearities $f$ in the semilinear problem,  which we do in the final subsection of the article.

	\subsection{Kato's inequality}
	The proof of  Kato's inequality in our setting, i.e. Proposition \ref{p:Kato's inequality for t+}, is motivated by the proofs of Kato's inequality found in \cite{aba+dupaig,chen_veron} for the case of the spectral fractional Laplacian and the fractional Laplacian, respectively. First we need a lemma.
	\begin{lem}\label{l:Kato's inequality}
		Let $w$ be the weak solution to the linear problem
		\begin{align*}
			\begin{cases}
				\Lo u=h,&\textrm{in $D$},\\
				\frac{u}{\PDFI\sigma}=0,&\textrm{on $\partial D$},
			\end{cases}
		\end{align*}
		for $h\in\LLL$. Let $\Lambda\in C^2(\R)$ be a convex function such that  $\Lambda(0)=0$, and such that $|\Lambda'|\le C$ for some $C>0$. Then
		\begin{align}\label{eq:Kato's inequality}
			\int_D\Lambda(w(x))\Lo \psi(x)dx\le \int_D\Lambda'(w(x))h(x)\psi(x)dx,\quad \psi\in C_c^\infty(D),
		\end{align}
		and
		\begin{align}\label{eq:Kato's inequality Lambda(w)<GDFI}
			\Lambda(w)\le \GDFI\left[\Lambda'(w)h\right]\quad \text{a.e. in $D$.}
		\end{align}
	\end{lem}
	
	\begin{proof}
		Recall that $w=\GDFI h\in\LLL$.
		
		Let $h\in C_c^\infty(D)$. Then by Proposition \ref{p:GDftwice diff} we have $w=\GDFI h\in C^{1,1}(\overline D)$ from which we can calculate $\Lo w$ and $\Lo \Lambda(w)$ pointwisely, see Proposition \ref{p:Lo various ways}. We have
		\begin{align*}
			\Lo&[\Lambda\circ w](x)  =\textrm{P.V.}\ \int_D[\Lambda(w(x))-\Lambda(w(y))]\,J_D(x,y)\:dy+\kappa(x)\,\Lambda(w(x)) \notag\\
			& =\ \Lambda'(w(x))\,\textrm{P.V.}\int_D[w(x)-w(y)]\,J_D(x,y)\:dy+\kappa(x)\,\Lambda(w(x)) \notag\\
			& \qquad-\textrm{P.V.}\int_D\left([w(x)-w(y)]^2\,J_D(x,y)
			\int_0^1\Lambda''(w(x)+t[w(y)-w(x)])(1-t)\,dt\right)dy\notag \\
			&\leq\ \Lambda'(w(x))\,\Lo w(x),
		\end{align*}
		where we have used that $\Lambda''\geq 0$ in $\R$ and 
		that $\Lambda(t)\leq t\Lambda'(t)$, which follows from $\Lambda(0)=0$ and the fact that $\Lambda'$ is non-decreasing. Integrating the previous inequality 
		with respect to $\psi(x)dx$, where $0\le \psi\in C_c^\infty(D)$, we get \eqref{eq:Kato's inequality}. Furthermore, since  both $w$ and $\Lambda(w)$ are in $C^{1,1}(\overline D)$, both sides of the previous inequality are in  $L^2(D)$, see Remark \ref{r:infin. genera of X}, 
		so we can apply Proposition \ref{p:LGDf=-f} 
		to get
		\begin{align*}
			\Lambda(w)=\GDFI[\Lo \Lambda(w)]\le \GDFI\left[\Lambda'(w)h\right]\quad \text{a.e. in $D$,}
		\end{align*}
		i.e. \eqref{eq:Kato's inequality Lambda(w)<GDFI} holds.
		
		Let $h\in \LLL$ and $(h_n)_n\subset C_c^\infty(D)$ such that $h_n\to h$ in $\LLL$ and a.e. in $D$.  By Corollary \ref{c:GDf finite if f in L1} we have $w_n\coloneqq \GDFI h_n\to w$ in $\LLL$ so by considering a subsequence we may assume that $w_n\to w$ a.e., too. From the first part of the proof we know
		\begin{align}
			\int_D\Lambda( w_n(x))&\Lo \psi(x) dx\le \int_D\Lambda'(w_n(x))h_n(x)\psi(x)dx\label{eq:Kato eq1}\\
			\text{and}\qquad\qquad&\Lambda(w_n)\le \GDFI\left[\Lambda'(w_n)h_n\right]\quad \text{a.e. in $D$,}\label{eq:Kato eq1.1}
		\end{align}
		for all $n\in \N$ and all $0\le \psi\in C_c^\infty(D)$.
		
		Now we will take $n$ in \eqref{eq:Kato eq1} and \eqref{eq:Kato eq1.1} to infinity. Recall that $|\Lo \psi|\le C_1\de$ by Lemma \ref{l: Lo psi < C delta}. Also, since $|\Lambda'|\le C$, we have $|\Lambda(t)-\Lambda(s)|\le C|t-s|$. By using these two facts and the fact that both $w_n\to w$ and $h_n\to h$ in $\LLL$, both sides of  \eqref{eq:Kato eq1} converge. Hence, by taking the limit in \eqref{eq:Kato eq1} we obtain
		$$\int_D\Lambda(w(x))\Lo \psi(x)dx\le \int_D\Lambda'(w(x))h(x)\psi(x)dx.$$
		
		Before we take the limit in equality \eqref{eq:Kato eq1.1}, note that $\Lambda\in C^2(\R)$ so $\Lambda(w_n)\to \Lambda(w)$ and $\Lambda'(w_n)\to \Lambda'(w)$ a.e. in $D$. Further, again by  $|\Lambda'|\le C$ and the fact that $h_n\to h$ in $\LLL$ we have
		\begin{align*}
			\left|\GDFI\left[\Lambda'(w_n)h_n\right]-\GDFI\left[\Lambda'(w)h\right]\right|&\le \GDFI\big[|\Lambda'(w)-\Lambda'(w_n)||h|\big]\\
			&\qquad\qquad+\GDFI\big[|\Lambda'(w_n)||h-h_n|\big]\to0,\quad n\to\infty,
		\end{align*}
		where the first term goes to zero by the dominated convergence theorem, and the second by the continuity of $\GDFI$ acting on $\LLL$, i.e. by Lemma \ref{l:rate of GDdelta}.
		This calculation justifies taking the limit in \eqref{eq:Kato eq1.1} to get
		$$\Lambda(w)\le \GDFI\left[\Lambda'(w)h\right]\quad \text{a.e. in $D$}.$$

	\end{proof}

	\begin{rem}
		For $h\in L^\infty(D)$ the inequalities \eqref{eq:Kato's inequality} and \eqref{eq:Kato's inequality Lambda(w)<GDFI} hold for every convex function $\Lambda\in C^2(\R)$ such that $\Lambda(0)=0$ since the assumption $|\Lambda'|\le C$ was used only as a technical tool to justify the usage of the dominated convergence theorem for general $h\in \LLL$. 
	\end{rem}
	In the next proposition we prove Kato's inequality  which says that we can take $\Lambda(t)=t^+=t\vee 0$ in Lemma \ref{l:Kato's inequality}.
	\begin{prop}[Kato's inequality]\label{p:Kato's inequality for t+}
		Let $w$ be the weak solution to the linear problem
		\begin{align*}
			\begin{cases}
				\Lo u=h,&\textrm{in $D$},\\
				\frac{u}{\PDFI\sigma}=0,&\textrm{on $\partial D$},
			\end{cases}
		\end{align*}
		for $h\in\LLL$. Then for every $\psi \in C_c^\infty(D)$, $\psi\ge 0$, it holds that
		\begin{align}\label{eq:Kato's for t^+}
			\int_D w(x)^+\Lo\psi(x)dx&\le \int_{\{w> 0\}}h(x)\psi(x)dx.
		\end{align}
		Moreover, it holds that
		\begin{align}\label{eq:Kato's inequality with GDFI}
			w^+&\le \GDFI\left[\1_{\{w> 0 \}}h\right],\quad \text{a.e. in $D$.}
		\end{align}
	\end{prop}
	\begin{proof}
		First, let us prove \eqref{eq:Kato's for t^+}. Set $\Lambda(t)=t\vee 0$ and $w=\GDFI h$ where $h\in \LLL$. Also, for every $n\in\N$ let $\Lambda_n:\R\to\R$ be defined by
		\begin{align}
			\Lambda_n(t)=\begin{cases}
				0,&t\le 0\\
				\frac{n^2t^3}{6},&t\in (0,\frac1n]\\
				\frac{1}{3n}-t+nt^2-\frac{n^2t^3}{6},&t\in (\frac{1}{n},\frac2n]\\
				t-\frac1n,&t>\frac2n.
			\end{cases}
		\end{align}
		We have that $\Lambda_n\in C^2(\R)$ , $0\le \Lambda_n\le \Lambda$, and $0\le \Lambda'_n\le 1$ in $\R$. Also, $\Lambda_n\to\Lambda$ and $\Lambda_n'\to \1_{(0,\infty)}$ in $\R$ as $n\to\infty$.  Thus, Lemma \ref{l:Kato's inequality} yields
		\begin{align}\label{eq:Kato eq2.2}
			\int_D\Lambda_n(w(x))\Lo \varphi(x)dx\le \int_D\Lambda_n'(w(x))h(x)\varphi(x)dx
		\end{align} 
		and the relation \eqref{eq:Kato's for t^+} follows from \eqref{eq:Kato eq2.2} by using the dominated convergence theorem.
		
		Let us now turn to \eqref{eq:Kato's inequality with GDFI}. Consider again the sequence $\Lambda_n$ defined above. Lemma \ref{l:Kato's inequality} yields
		\begin{align}\label{eq:Kato aux GDFI}
			\Lambda_n(w)\le \GDFI\left[\Lambda_n'(w)h\right],\quad\text{a.e. in $D$ and for all $n\in\N$.}
		\end{align}
		Again, by taking $n\to\infty$ and by using the dominated convergence theorem we get
		\begin{align*}
			w^+&\le \GDFI\left[\1_{\{w> 0 \}}h\right],\quad \text{a.e. in $D$.}
		\end{align*}
		
	\end{proof}
	\begin{rem}\label{r:Kato inequality extension}
		By modifying the proof of the previous proposition we also get
		\begin{align}
			\int_D w(x)^+\Lo\psi(x)dx\le \int_{\{w\ge 0\}}h(x)\psi(x)dx,\label{eq:Kato ineq.1 altern }
		\end{align}
		and 
		\begin{align}
			w^+\le \GDFI\left[\1_{\{w\ge 0 \}}h\right],\quad \text{a.e. in $D$.}
		\end{align}
		Indeed, in the proof we only need to change $\Lambda_n$ to $\tilde\Lambda_n\in C^2(\R)$ such that $\tilde\Lambda_n(t)=\Lambda_n(t+\frac2n)-\frac1n$. For $\tilde\Lambda_n$ it holds that 
		$$-\frac{1}{n}\le \tilde\Lambda_n\le \Lambda, \quad0\le \tilde\Lambda_n'\le1,\quad \lim_n\tilde\Lambda_n= \Lambda,\enskip\text{and}\enskip  \lim_n\tilde\Lambda_n'=\1_{[0,\infty)}$$
		in $\R$.	By repeating the procedure in the proof of the previous proposition we get the claim.
		
	\end{rem}
	
	\begin{rem}
		Note that Kato's inequality was proved only for weak solutions of linear problems with a zero boundary condition whereas the classical Kato's inequality holds for subsolutions even if the considered linearity is a measure, see \cite{brezis_ponce}. To the best of our knowledge it is not clear whether the inequality \eqref{eq:Kato's for t^+} holds for subsolutions since the non-local nature of the operator $\Lo$ causes problems in the calculations in Proposition \ref{p:Kato's inequality for t+}. Even in simpler non-local cases as in \cite{aba+dupaig} and \cite{chen_veron} Kato's inequality was proved only for solutions, see \cite[Lemma 31]{aba+dupaig} and \cite[Proposition 2.4]{chen_veron}.
	\end{rem}
	
	In the next corollary we bring a simple consequence of Kato's inequality which is the fact interesting in itself.
	\begin{cor}\label{c:maximum of solutions is subsolution}
		Let $u$ and $v$ be weak solutions of \eqref{semilinear problem}. Then $\max\{u,v\}$ is a subsolution of \eqref{semilinear problem}.
	\end{cor}
	\begin{proof}
		Applying Proposition \ref{p:Kato's inequality for t+} to the $w\coloneqq u-v$ and $h(x)\coloneqq f(x,u(x))-f(x,v(x))$ we get
		\begin{align*}
			\int_D w^+(x) \Lo\psi(x) dx\le \int_{u> v}\left[f(x,u(x))-f(x,v(x))\right]\psi(x)dx,\quad \psi\in C_c^\infty(D),\, \psi\ge 0.
		\end{align*}
		Since $\max\{u,v\}=v+(u-v)^+=v+w^+$ we have for all non-negative $\psi\in C_c^\infty(D)$
		\begin{align*}
			\int_D \max\{u,v\}(x)&\Lo\psi(x) dx\\
			&\le \int_D f(x,v(x))\psi(x)dx-\int\limits_{\partial D}\frac{\partial}{\partial \mathbf{n}}\GDFI\psi(z)\zeta(dz)\\&\qquad\qquad\qquad\qquad+\int_{u> v}\left[f(x,u(x))-f(x,v(x))\right]\psi(x)dx\\
			&=\int_{u\le v} f(x,v(x))\psi(x)dx-\int\limits_{\partial D}\frac{\partial}{\partial \mathbf{n}}\GDFI\psi(z)\zeta(dz)\\&\qquad\qquad\qquad\qquad+\int_{u> v}f(x,u(x))\psi(x)dx\\
			&=\int_Df(x,\max\{u,v\}(x))\psi(x) dx-\int\limits_{\partial D}\frac{\partial}{\partial \mathbf{n}}\GDFI\psi(z)\zeta(dz).
		\end{align*}
	\end{proof}

	\subsection{Semilinear problem}
	In this subsection we prove  the  existence and uniqueness results for the semilinear problem \eqref{semilinear problem}. As such, the subsection is central  to  the article. 
	
	For the nonlinearity $f$ in the following problems we will almost always assume that the following condition holds.
	\begin{assumption}{F}{}\label{F}
		$f:D\times \R\to\R$ is continuous in the second variable, and there exist  a  locally bounded function $\rho:D\to[0,\infty]$ and  a   non-decreasing function $\Lambda:[0,\infty)\to[0,\infty)$ such that $|f(x,t)|\le \rho(x)\Lambda(|t|)$, $x\in D$, $t\in \R$.
	\end{assumption}
	
	From now on, the function $f$ will be solely used as a nonlinearity in the semilinear problem and the functions $\rho$ and $\Lambda$ are solely used as the functions in the condition \ref{F} for $f$.
	
	Our first result is  the uniqueness theorem for general nonlinearity $f$ which is non-increasing in the second variable.
	\begin{prop}\label{p:uniqueness if nonincreasing}
		If the nonlinearity $f$ in \eqref{semilinear problem} is non-increasing in the second variable, then the weak solution of \eqref{semilinear problem}, if it exists, is unique (up to the modification on the Lebesgue null set).
	\end{prop}
	\begin{proof}
		Let $u$ and $v$ be two solutions of \eqref{semilinear problem}. Then $w\coloneqq u-v$ solves the linear problem
		\begin{align*}
			\begin{cases}
				\Lo w(x)=f(x,u(x))-f(x,v(x)),&\text{in $D$,}\\
				\frac{w}{\PDFI\sigma}=0,&\text{on $\partial D$.}
			\end{cases}
		\end{align*}
		By Kato's inequality \eqref{eq:Kato's inequality with GDFI}, since $f$ is non-increasing in the second variable, we have
		\begin{align}\label{eq:Kato for uniqueness}
			w^+\le \GDFI\left[\1_{\{u>v\}}\cdot \big(f_u-f_v\big) \right] \le 0.
		\end{align}
		Thus, $u\le v$ a.e. in $D$. Reversing the roles of $u$ and $v$ we get $u\ge v$ a.e. in $D$, hence $u=v$ a.e. in $D$.
	\end{proof}

	The next theorem, Theorem \ref{t:super/subsolution implies existence}, deals with a semilinear problem with a zero boundary condition and it is a generalization of \cite[Theorem 32]{aba+dupaig} to our setting of more general non-local operators.  Theorem \ref{t:super/subsolution implies existence} will be of great importance for a general semilinear problem (with a non-zero boundary condition), and it is, in fact, the cornerstone of the proof of Theorem \ref{t:t1 semilinear problem}.  A  somewhat similar role to  a  semilinear problem in slightly different non-local setting plays \cite[Theorem 3.6]{semilinear_bvw}.
	\begin{thm}\label{t:super/subsolution implies existence} 
		Let $f$ satisfy $\ref{F}$. Assume that there exist a  supersolution $\overline u$ and a subsolution $\underline{u}$ to the semilinear problem
		\begin{align}\label{eq:semilinear sub super problem}
			\begin{cases}
				\Lo u(x)=f(x,u(x)),&\textrm{in $D$},\\
				\frac{u}{\PDFI\sigma}=0,&\textrm{on $\partial D$},
			\end{cases}
		\end{align}
		of the form $\underline u=\GDFI \underline h$ and $\overline u=\GDFI \overline h$ such that $\underline u\le \overline u$, $\underline h(x)\le f(x,\underline u(x))$ and $f(x,\overline u(x))\le \overline h(x)$ a.e. in $D$, and such that $\overline{u},\underline{u}\in\LLL \cap L^\infty_{loc}(D)$. Further, assume that $\rho\Lambda(|\underline u|\vee |\overline u|)\in \LLL$.

		Then there exist weak solutions $u_1,u_2\in\LLL$ of \eqref{eq:semilinear sub super problem} such that every solution of \eqref{eq:semilinear sub super problem} with property $\underline{u}\le u \le \overline u$  satisfies
		$$ \underline u\le u_1 \le u\le u_2\le \overline u.$$
		
		Further, every weak solution $u$ of \eqref{eq:semilinear sub super problem} with property $\underline{u}\le u \le \overline u$ is continuous after the modification on a Lebesgue null set.
		
		Additionally, if the nonlinearity $f$ is non-increasing in the second variable, the weak solution of \eqref{eq:semilinear sub super problem} is unique.
	\end{thm}

	\begin{proof}
		\textit{Step 1: existence of a solution to \eqref{eq:semilinear sub super problem}.}
		Define the function $F:D\times \R\to \R$ by
		\begin{align*}
			F(x,t)=\begin{cases}
				f(x,\underline u(x)),&t<\underline u(x),\\
				f(x,t),&\underline{u}\le t \le \overline u,\\
				f(x,\overline u(x)),&\overline u(x)<t,
			\end{cases}
		\end{align*}
		and denote by $F_v(x)\coloneqq F(x,v(x))$. Note that since $f$ is continuous in the second variable, so is $F$. Further, $|F_v|\le \rho\Lambda(|\underline u|\vee |\overline u|)$, hence $F_v\in \LLL$, for all $v\in \LLL$.
		
		Also, the mapping $v\mapsto F_v$ is continuous from $\LLL$ to $\LLL$. Indeed, take $v_n\to v$ in $\LLL$ and let $(v_{n_k})_k$ be a subsequence of $(v_n)_n$ which converges to $v$  a.e. By Lemma \ref{apx:l:Montenegro and Ponce Prop 2.1} the family $\{F_{v_{n_k}}:k\in\N\}$ is uniformly integrable with respect to the measure $\de(x)dx$, hence by Vitali's theorem \cite[Theorem 16.6]{schilling2017measures}, we get $F_{v_{n_k}}\to F_v$ in $\LLL$ because $F$ is continuous in the second variable. However, the limit does not depend on the subsequence $(v_{n_k})_k$ so $v\mapsto F_v$ is continuous.
		
		Next we prove that the operator $\KK:\LLL \to \LLL$ defined by 
		$$\KK v(x)=\int_D\GDFI(x,y)F(y,v(y))dy,\quad x\in D,$$
		is compact. Since $v\mapsto F_v$ is continuous in $\LLL$, Corollary \ref{c:GDf finite if f in L1}  implies that $\KK$ is continuous $\LLL$, too. To have compactness, we are left to prove that $\KK$ maps bounded sets to relatively compact sets. To this end, take a bounded sequence $(v_n)_n\subset \LLL$. Recall  that  $|F_{v_n}|\le \rho\Lambda(|\underline u|\vee |\overline u|)$,  and notice that  $\rho\Lambda(|\underline u|\vee |\overline u|)\in\LLL\cap L^\infty_{loc}(D)$ since $\overline{u},\underline{u}\in\LLL \cap L^\infty_{loc}(D)$, so $(\KK v_n)_n$ are pointwisely bounded by Proposition \ref{p:continuity of GDf} and equicontinuous by Remark \ref{r:equicont GDf}. By Arzel\`{a}-Ascoli theorem, there is a subsequence $(\KK v_{n_k})_k$ of $(\KK v_n)_n$ which converges pointwisely to some $u\in C(D)\cap\LLL$.  Since $\KK v_n=\GDFI F_{v_n}$, Lemma \ref{l:rate of GDdelta} implies that that $\{\KK v_{n_k}:k\in\N\}$ is uniformly integrable with respect to the measure $\de(x)dx$ since $\{F_{v_{n_k}}:k\in\N\}$ is. However, $\KK v_{n_k}\to u$ pointwisely so by Vitali's theorem \cite[Theorem 16.6]{schilling2017measures} we have $\KK v_{n_k}\to u$ in $\LLL$.
		
		This means that $\KK$ is compact so by Schauder's fixed point theorem there is $u\in\LLL$ such that $\KK u =u$ in $D$, i.e. $u$ solves
		\begin{align*}
			\begin{cases}
				\Lo u(x)=F(x,u(x)),&\textrm{in $D$},\\
				\frac{u}{\PDFI\sigma}=0,&\textrm{on $\partial D$}.
			\end{cases}
		\end{align*}
		We need to prove that $\underline u \le u \le \overline u$ in $D$ which would mean that $u$ also solves \eqref{eq:semilinear sub super problem}. For this step we will use Kato's inequality. 
		
		More precisely, applying Proposition \ref{p:Kato's inequality for t+} to $w=u-\overline u=\GDFI\big(F_u-\overline h\big)$ we get
		\begin{align}
			(u-\overline u)^+\le\GDFI\left[\1_{\{u>\overline u\}}\cdot\big(F_u-\overline h\big)\right]\le \GDFI\left[\1_{\{u>\overline u\}}\cdot\big(f_{\overline u}-f_{\overline u}\big)\right]=0,
		\end{align}
		where the second inequality holds since $F(x,u(x))=f(x,\overline u(x))$ on $\{u\ge \overline u\}$ and since we assume $f(x,\overline u)\le \overline h$ a.e. in $D$. This means $u\le \overline u$ a.e. in $D$. Similarly we get that $\underline u\le u$ a.e. in $D$. Hence, we found a solution to the problem \eqref{eq:semilinear sub super problem}.
		
		\noindent
		\textit{Step 2: finding the maximal and the minimal solution.}		
		We adapt a method from \cite[Theorem 1.3]{DancerSweersMaximal} which uses Zorn's lemma.
		
		Let $\mathcal{P}\coloneqq\{u\in \LLL: \underline u\le u\le \overline u\text{ and $u$ solves \eqref{eq:semilinear sub super problem}}\}$. Let $\{u_i\}_{i\in\mathcal{I}}$ be a totally ordered subset of $\mathcal{P}$. Since $u_i\in\PP$ and since we have  $\rho\Lambda(|\underline u|\vee |\overline u|)\in \LLL \cap L^\infty_{loc}(D)$, it follows that $\{u_i\}_{i\in\mathcal{I}}$ is equicontinuous in $D$. In fact, by Remark \ref{r:equicont GDf} the set $\{u_i\}_{i\in\mathcal{I}}$ is equicontinuous on every compact subset of $D$. Hence, the function $u\coloneqq \sup_{i\in \II} u_i$ is continuous and $u$ can be approximated by $\{u_i\}_{i\in \II}$ uniformly on compact subsets of $D$. Moreover, $D$ is $\sigma$-compact so  we can choose an increasing sequence $(u_n)_n\subset \{u_i\}_{i\in\mathcal{I}}$ such that $\lim_n u_n(x)=u(x)$ for all $x\in D$.
		
		By the dominated convergence theorem, since $|f_{u_n}|\le\rho\Lambda(|\underline u|\vee |\overline u|)$,  it easily  follows  by the continuity of $f$ in the second variable that $u=\lim_n u_n=\lim_n\GDFI \big(f_{u_n}\big)=\GDFI \big(f_{u}\big),$ i.e. $u\in \PP$. Now Zorn's lemma implies that there exists the maximal solution $u_2$ of \eqref{eq:semilinear sub super problem}. 		We find the minimal solution $u_1$ 	in the same way.

		\noindent
		\textit{Step 3: continuity of solutions.} 
		We prove that every solution of \eqref{eq:semilinear sub super problem} with property $\underline{u}\le u \le \overline u$ is continuous up to the modification. Indeed, every solution satisfies $u=\GDFI f_u$ a.e. in $D$. Furthermore, since $\underline{u}\le u \le \overline u$ and $\rho\Lambda(|\underline u|\vee |\overline u| )\in \LLL \cap L^\infty_{loc}(D)$, we have $\GDFI f_u\in C(D)$ by Proposition \ref{p:continuity of GDf}. Finally, $\tilde u\coloneqq \GDFI f_u$ is a continuous modification of $u$, hence $f_u=f_{\tilde u}$ a.e. in $D$, hence $\tilde u= \GDFI f_{\tilde u}$ in $D$.
		
		\noindent 
		\textit{Step 4: uniqueness of solution.} In the case when $f$ is non-increasing in the second variable, uniqueness follows from Proposition \ref{p:uniqueness if nonincreasing}.
	\end{proof}

	By using the previous theorem, a  method of sub- and super-solutions and the approximation of harmonic functions,  we solve a semilinear problem  that \ deals with a non-positive nonlinearity $f$ and a non-negative boundary condition $\zeta$. Theorem \ref{t:t1 semilinear problem} generalizes \cite[Theorem 8]{aba+dupaig} to our setting of more general non-local operators. Moreover, we consider  a  more general boundary condition which can also be a measure,  whereas in \cite[Theorem 8]{aba+dupaig} only continuous functions  were  considered. The nonlinearity in our theorem is also slightly  more general  than  the one in  \cite[Theorem 8]{aba+dupaig}. A similar result in a slightly different non-local setting can be found in \cite[Theorem 3.10]{semilinear_bvw}.

	\begin{thm}\label{t:t1 semilinear problem}
		Let $f:D\times \R\to(-\infty,0]$ such that $f(x,0)=0$, $x\in D$, and such that $f$ satisfies \ref{F}. Further, let $\zeta\in \MM(\partial D)$ be a finite non-negative measure such that
		$$\rho\Lambda(\PDFI\zeta)\in\LLL.$$
		Then the problem
		\begin{align}\label{eq:semilinear non-positive problem}
			\begin{cases}
				\Lo u(x)=f(x,u(x)),&\textrm{in $D$},\\
				\frac{u}{\PDFI\sigma}=\zeta,&\textrm{on $\partial D$},
			\end{cases}
		\end{align}
		has a weak solution $u\in C(D) \cap \LLL$.
		
		Additionally, if $f$ is non-increasing in the second variable, the continuous weak solution of \eqref{eq:semilinear non-positive problem} is unique.
	\end{thm}
	
	\begin{proof}		
		Let $(\wt f_k)_k$ be a non-negative sequence of bounded functions such that $\GDFI \wt f_k \uparrow \PDFI \zeta$ in $D$. This sequence exists by \cite[Appendix A.1]{semilinear_bvw} since the semigroup $(Q_t^D)_t$ is strongly Feller, $\GDFI\de\asymp \de$ by Lemma \ref{l:rate of GDdelta}, and since $\PDFI\zeta$ is a continuous function with the mean-value property with respect to $X$, see Theorem \ref{t:PD conti} and Theorem \ref{t:harmonic functions}.
		
		We build a sequence of solutions to the following semilinear problems
		\begin{align}\label{eq:semi non-positive fk}
			\begin{cases}
				\Lo u(x)=f(x,u(x))+\wt f_k,&\textrm{in $D$},\\
				\frac{u}{\PDFI\sigma}=0,&\textrm{on $\partial D$}.
			\end{cases}
		\end{align}
		For every $k\in \N$,  a subsolution to \eqref{eq:semi non-positive fk} is $\underline u=0$ since $f(x,0)=0$ and since $\wt f_k\ge0$. A supersolution to \eqref{eq:semi non-positive fk} is $\overline u=\GDFI \wt f_k$ because $f$ is non-positive. Note that both $\underline u$ and $\overline u$ are bounded functions, so it is trivial to check that the assumptions of Theorem \ref{t:super/subsolution implies existence} are satisfied. Hence,  for every $k\in\N$ there is a solution $u_k\ge 0$ to \eqref{eq:semi non-positive fk} which is also continuous in $D$ and satisfies
		\begin{align}\label{eq:semi solut eq0}
			u_k=\GDFI f_{u_k}+\GDFI \wt f_k,\quad \text{in $D$.}
		\end{align}
		
		Now we find an appropriate subsequence of $(u_k)_k$ which converges to a solution of \eqref{eq:semilinear non-positive problem}. Since $\GDFI\wt f_k$ is continuous and increases to the continuous function $\PDFI\zeta$, by Dini's theorem the convergence is locally uniform so the usual $3\varepsilon$-argument gives equicontinuity of the family $(\GDFI \wt f_k)_k$. Also, since $|f_{u_k}|\le \rho\Lambda(\PDFI\zeta)$, equicontinuity of $(\GDFI(f_{u_k}))_k$ follows by Proposition \ref{p:continuity of GDf} and Remark \ref{r:equicont GDf}. Hence, Arzel\`{a}-Ascoli theorem gives us a subsequence, denoted again by $(u_k)_k$, which converges to a continuous function $u$.
		
		Now we show that $u$ is a solution of \eqref{eq:semilinear non-positive problem}. Obviously, since $u=\lim_{k\to\infty} u_k$ and $0\le u_k\le \GDFI \wt f_k\le \PDFI \zeta<\infty$, $u$ is non-negative and finite. Further, $\GDFI \wt f_k\uparrow \PDFI\zeta$, so we are left to prove that $\GDFI f_{u_k}\to \GDFI f_u$. However, this is easy since $|f_{u_k}|\le\rho\Lambda(\PDFI \zeta)$, so continuity of $f$ in the second variable and  the dominated convergence imply $\GDFI f_{u_k}\to \GDFI f_u$.

		Uniqueness, if $f$ is non-increasing in the second variable, follows from Proposition \ref{p:uniqueness if nonincreasing}.
	\end{proof}
	\begin{rem}\label{r:maximal solution}
		Applying Zorn's lemma argument from the proof of Theorem \ref{t:super/subsolution implies existence} we get that for the problem \eqref{eq:semilinear non-positive problem} there exists a minimal solution $u_1$ and a maximal solution $u_2$ such that for every solution $u$ of \eqref{eq:semilinear non-positive problem} we have
		$$0\le u_1\le u\le u_2\le \PDFI\zeta,\quad \text{in $D$.}$$
	\end{rem}
	We say that $\Lambda:[0,\infty)\to[0,\infty)$ satisfies the doubling condition if there exists $C>0$ such that
	\begin{align}\label{eq:doubling condition}
		\Lambda(2t)\le C \Lambda(t),\quad t\ge1.
	\end{align}
	If $\Lambda$ is non-decreasing, the condition \eqref{eq:doubling condition} implies that for every $c_1>1$ there is $c_2=c_2(C,c_1)>0$ such that
	\begin{align}\label{eq:doubling condition 1}
		\Lambda(c_1t)\le c_2 \Lambda(t),\quad t\ge1.
	\end{align}
	\begin{cor}\label{c:semilinear doubling condition}
		Let $f:D\times \R\to(-\infty,0]$ such that $f(x,0)=0$, $x\in D$. Let $f$ also satisfy \ref{F} such that $\Lambda$ satisfies the doubling condition \eqref{eq:doubling condition}.
		
		If $\rho\Lambda\left(\frac{1}{\de^2\phi(\de^{-2})}\right)\in \LLL$, then the problem
		\begin{align}\label{eq:semilinear non-positive continuous problem}
			\begin{cases}
				\Lo u(x)=f(x,u(x)),&\textrm{in $D$},\\
				\frac{u}{\PDFI\sigma}=\zeta,&\textrm{on $\partial D$},
			\end{cases}
		\end{align}
		has a continuous weak solution for every non-negative function $\zeta\in C(D)$. Additionally, if $f$ is non-increasing in the second variable, the continuous weak solution is unique.
		
		In particular, if $f(x,t)=-|t|^p$, then the equation \eqref{eq:semilinear non-positive continuous problem} has a unique continuous weak solution for $p<\frac{1}{1-\delta_1}$, where $\delta_1$ comes from \eqref{eq:scaling condition}.
	\end{cor}
	\begin{proof}
		Note that for $\zeta\in C(D)$ we have $\PDFI\zeta \le c_1 \frac{1}{\de^2\phi(\de^{-2})}$ by Lemma \ref{l:Poisson integral sharp bounds} since $\zeta$ is bounded on $\partial D$. Thus, from the doubling condition we have
		$$\rho\Lambda(\PDFI\zeta)\le c_2\rho\Lambda\left(\frac{1}{\de^2\phi(\de^{-2})}\right)\in \LLL$$
		so we can apply Theorem \ref{t:t1 semilinear problem} to get the claim.
		
		In the special case $f(x,t)=-|t|^p$ we have $\rho\equiv 1$ and $\Lambda(t)=t^p$ so \eqref{eq:scaling condition} and the reduction to the flat case give us
		$$\int_D\rho\Lambda\left(\frac{1}{\de^2\phi(\de^{-2})}\right)\de dx\asymp \int_D \frac{\de}{\de^{2p}\phi(\de^{-2})^{2p}}dx\lesssim \int_0^1 t^{1-2p+2p\delta_1}dt$$
		which is finite if $p<\frac{1}{1-\delta_1}$.
	\end{proof}
	\begin{rem}\label{r:semilin negative rem}
		Assume that we are in the spectral fractional Laplacian case in the previous corollary, i.e. if $\phi(\lambda)=\lambda^{s}$, for some $s\in(0,1)$. Then we can find a solution of \eqref{eq:semilinear non-positive continuous problem} for $f(x,t)=-|t|^p$ and for every non-negative $\zeta\in C(\partial D)$ if $p<\frac{1}{1-s}$ since $\delta_1=s$ in this case.
		
		Conversely, if $f(x,t)=-|t|^p$ for $p\ge \frac{1}{1-s}$, and we additionally demand that the boundary condition holds pointwisely for a non-negative $\zeta\in C(\partial D)$ such that $\zeta\not\equiv 0$, then the problem \eqref{eq:semilinear non-positive continuous problem} does not have a solution. Indeed, assume that $u$ is a solution to \eqref{eq:semilinear non-positive continuous problem} and that the boundary condition holds pointwisely. Then $u\gtrsim \de^{2s-2}$ near $z\in \partial D$ such that $\zeta(z)>0$ since $\PDFI\zeta\asymp \de^{2s-2}$ near such $z$, see Proposition \ref{p:boundary oper. of Poisson}. Thus, $|u|^p\not\in\LLL$ since $p\ge\frac{1}{1-s}$, i.e. $\GDFI f_u=\infty$ in $D$ by Lemma \ref{l:rate of GDdelta}, which is a contradiction.
	\end{rem}
	
	One of the weaknesses of Theorem \ref{t:super/subsolution implies existence} is that one has to have a supersolution and a subsolution which are strictly Green potentials, i.e. a supersolution and a subsolution cannot consist of Poisson integrals which are annulled by $\Lo$, since only then we may use Kato's inequality \eqref{eq:Kato's inequality with GDFI}. However, in some cases we can exploit some other methods for obtaining a solution to a semilinear problem. For example, in the next theorem we deal with a non-negative nonlinearity $f$ and a non-negative boundary condition $\zeta$ and we use a method of monotone iterations to obtain a solution.
	\begin{thm}\label{t:semilin non-negative monotone linearity}
		Let $f:D\times \R\to[0,\infty)$ satisfy \ref{F}, and let $f$ be a non-decreasing function in the second variable. Let $\zeta$ be a non-negative finite measure on $\partial D$ such that
		\begin{align}\label{eq:semi condition nondecreasing}
			\GDFI\big(\rho\Lambda(2\PDFI \zeta)\big)\le \PDFI\zeta,\quad \text{in $D$.}
		\end{align}
		 Then  there is a continuous non-negative solution to 
		\begin{align}\label{eq:semilinear non-nonegative}
			\begin{cases}
				\Lo u(x)=f(x,u(x)),&\textrm{in $D$},\\
				\frac{u}{\PDFI\sigma}=\zeta,&\textrm{on $\partial D$}.
			\end{cases}
		\end{align}
	\end{thm}
	\begin{proof}
		We use a method of monotone iterations. Let $u_0=0$, and define for $n\ge 1$
		\begin{align*}
			u_{n}=\GDFI\big(f_{u_{n-1}}\big)+\PDFI\zeta.
		\end{align*}
		Since $f$ is non-negative and non-decreasing in the second variable, it follows that $(u_n)_n$ is non-negative and  non-decreasing, too. However, by induction it is easy to see that $0\le u_n\le 2\PDFI \zeta$. Indeed, for $u_0$ this fact is trivial, and for $n\ge 1$ by \eqref{eq:semi condition nondecreasing} we have
		\begin{align*}
			u_n=\GDFI\big(f_{u_{n-1}}\big)+\PDFI\zeta\le  \GDFI\big(\rho\Lambda(2\PDFI \zeta)\big)+\PDFI\zeta\le 2\PDFI\zeta.
		\end{align*}
		This means that $u=\uparrow \lim_{n\to\infty}u_n$ is well defined. Since $f$ is continuous in the second variable by \ref{F} and since the integrability condition \eqref{eq:semi condition nondecreasing} holds, by the dominated convergence theorem we get
		$$ u=\GDFI f_u+\PDFI\zeta,$$
		i.e. we found a solution to \eqref{eq:semilinear non-nonegative}. 
		
		For the continuity of $u$, note that since $u\le 2\PDFI\zeta$,  the condition \eqref{eq:semi condition nondecreasing} implies that $f_u\in\LLL$ in the following way
		$$\int_D f_u(x)\de(x) dx\asymp\int_D f_u(x)\GDFI\de(x) dx=\int_D\GDFI(f_u)(x)\de(x)\le \int_D \PDFI\zeta(x)\de(x)dx<\infty.$$
		 Also, the bound $u\le 2\PDFI\zeta$ implies $f_u\in L_{loc}^\infty(D)$. 
		Now Proposition \ref{p:continuity of GDf} and Theorem \ref{t:PD conti} give $u\in C(D)$.
	\end{proof}
	
	\begin{rem}\label{r:semilin positive rem}
		If we are in the spectral fractional Laplacian case in the previous theorem, i.e. if $\phi(\lambda)=\lambda^{s}$, for some $s\in(0,1)$, then there exists a solution of \eqref{eq:semilinear non-nonegative} for any non-negative $\zeta\in C(\partial D)$ and for the nonlinearity $f(x,t)=m|t|^p$, where  $m>0$ is sufficiently small and $p<\frac{1}{1-s}$. Indeed, in this case $\PDFI\zeta \lesssim \de(x)^{2-2s}$, and $(\PDFI\zeta)^p\in \LLL$ if $p<\frac{1}{1-s}$. Obviously, we chose the parameter $m>0$ so small so that \eqref{eq:semi condition nondecreasing} holds.
		
		Conversely, if $p\ge \frac{1}{1-s}$, then the problem  \eqref{eq:semilinear non-nonegative} does not have a solution for  $f(x,t)=m|t|^p$ for any $m>0$ and for any non-negative $\zeta\in C(\partial D)$ such that $\zeta\not\equiv 0$. Indeed, assume that $u$ solves \eqref{eq:semilinear non-nonegative}. Then $u\ge \PDFI\zeta$ since $f\ge0$ and $\PDFI\zeta\gtrsim \de^{2-2s}$, near $z\in \partial D$ such that $\zeta(z)>0$, see Proposition \ref{p:boundary oper. of Poisson}. Hence for $p\ge \frac{1}{1-s}$ the function $(\PDFI\zeta)^p\not\in \LLL$ which implies $u=\GDFI f_u+\PDFI\gtrsim \GDFI\left((\PDFI\zeta)^p\right)=\infty$ in $D$, by Lemma \ref{l:rate of GDdelta}.
	\end{rem}
	
	To obtain a solution to a semilinear problem with an unsigned nonlinearity $f$ and an unsigned boundary condition $\zeta$ we need some stronger assumptions on the nonlinearity $f$. The following theorem is in  the  spirit same as \cite[Theorem 2.4]{bogdan_et_al_19} and \cite[Corollary 3.8]{semilinear_bvw} which were proved in a different non-local setting.
	\begin{thm}\label{t:semilinear signed datum}
		Let $f:D\times \R\to \R$ satisfy \ref{F} and let $\zeta$ be a finite measure on $\partial D$. Assume that $\GDFI\rho\in C_0(D)$ and $\GDFI\big(\rho\Lambda(2\PDFI|\zeta|)\big)\in C_0(D)$. Assume additionally that: (a) $\Lambda$ is sublinearly increasing, i.e. $\lim_{t\to\infty}\Lambda(t)/t=0$, or  (b) $m>0$ is sufficiently small. Then the semilinear problem
		\begin{align}\label{eq:semilinear bogdan}
			\begin{cases}
				\Lo u(x)=m\,f(x,u(x)),&\textrm{in $D$},\\
				\frac{u}{\PDFI\sigma}=\zeta,&\textrm{on $\partial D$}.
			\end{cases}
		\end{align}
		has a weak continuous solution $u$ such that $|u|\le C+\PDFI|\zeta|$, for some constant $C\ge0$.
		
		If, in addition, $f$ is non-increasing in the second variable, $u$ is a unique weak solution to \eqref{eq:semilinear bogdan}.
	\end{thm}
	\begin{proof}
		The proof follows the proof of \cite[Theorem 2.4]{bogdan_et_al_19} and we repeat the main steps for the reader's convenience.
		
		Define the operator $T$ on $C_0(D)$ by
		\begin{align*}
			T v(x)=\int_D\GDFI(x,y)m\,f(y,v(y)+\PDFI\zeta)dy,\quad v\in C_0(D),\; x\in D.
		\end{align*}
		Our goal is to get a fixed point of the operator $T$ from which we will extract a solution to \eqref{eq:semilinear bogdan}.
		
		Let $r_\rho=\sup_{x\in D}\GDFI\rho(x)<\infty$ and $r_\zeta=\sup_{x\in D}\GDFI\big(\rho\Lambda(2\PDFI|\zeta|)\big)(x)<\infty$. Let $C\ge 0$ and define $K\coloneqq\{v\in C_0(D):\|v\|_\infty\le C\}$. It is easy to show that for $a,b>0$ we have $\Lambda(a+b)\le \Lambda(2a)+\Lambda(2b)$. Hence,
		\begin{align*}
			|f(y,v(y)+\PDFI\zeta(y))|\le \rho(y)\Lambda(|v(y)|+\PDFI|\zeta|(y))\le \rho(y)\Lambda(2C)+\rho\Lambda(2\PDFI|\zeta|(y)),\quad v\in K,
		\end{align*}
		so  $T v\in C_0(D)$ by the upper bound and the same calculations as in Proposition \ref{p:continuity of GDf}. Moreover,
		\begin{align*}
			\|T v\|_\infty&=\sup_{x\in D}|\int_D\GDFI(x,y)mf(y,v(y)+\PDFI\zeta(y))dy|\\
			&\le \sup_{x\in D}\int_D\GDFI(x,y)m\big(\rho(y)\Lambda(2C)+\rho\Lambda(2\PDFI|\zeta|(y))\big)dy\le m\big(r_\rho\Lambda(2C)+r_\zeta\big).
		\end{align*}
		If $m$ is sufficiently small or $\Lambda$ sublinearly increases, there is $C>0$ such that $m\big(r_\rho\Lambda(2C)+r_\zeta\big)\le C$. Fix this $C$.
		We will now use Schauder's fixed point theorem on $T$. By the choice of $C$, we have $T[K]\subset K$. Also, $T$ is a continuous operator on $K$. This is proved by assuming the opposite  as in the proof of \cite[Theorem 3.6 (iii)]{semilinear_bvw} for the operator defined in equation (3.14) therein, see also \cite[Eq. (3.15)]{semilinear_bvw}. Further, the family $\{Tv:v\in K\}$ is equicontinuous in $D$ by the inequality
		\begin{align*}
			|Tv(x)-Tv(\xi)|\le \int_D|\GDFI(x,y)-\GDFI(\xi,y)|m\big(\rho(y)\Lambda(2C)+\rho\Lambda(2\PDFI|\zeta|(y))\big)dy,\quad v\in K,
		\end{align*}
		and by the Remark \ref{r:equicont GDf}. Arzel\`{a}-Ascoli theorem implies that $T[K]$ is precompact in $K$, thus, by Schauder's fixed point theorem there  exists  $u_0\in K$ such that $Tu_0=u_0$. To finish the proof, notice that the function $$u(x)\coloneqq u_0(x)+\PDFI\zeta(x)=\int_D\GDFI(x,y)mf(y,u(y))dy+\PDFI\zeta(x)$$
		solves \eqref{eq:semilinear bogdan}, and it holds that $u\in C(D)$ and  $|u|\le C+\PDFI|\zeta|$.		
	\end{proof}
	
	\begin{rem}\label{r:semilin bog rem}
		In the spectral fractional case where $\phi(\lambda)=\lambda^s$, for some $s\in (0,1)$, when $\zeta\in C(\partial D)$, we have a solution of \eqref{eq:semilinear bogdan} for the nonlinearity $f$ which satisfies $|f(x,t)|\lesssim |t|^p$ if $p<\frac{s}{1-s}$. Indeed, in that case $\PDFI|\zeta|\lesssim \de^{2s-2}$, hence $(\PDFI|\zeta|)^p\in\LLL$ and  $\GDFI\left((\PDFI|\zeta|)^p\right)\in C_0(D)$ by Theorem \ref{t:boundary of GDFIU}, or see \cite[Proposition 7]{dhifli2012}. Note that the range $p<\frac{s}{1-s}$ is worse  than  the one for Corollary \ref{c:semilinear doubling condition} and Theorem \ref{t:semilin non-negative monotone linearity}, see Remarks \ref{r:semilin negative rem} and \ref{r:semilin positive rem}.
	\end{rem}
	
	%
	%
	%
	%
	%
	%
	%
	\appendix
	\section{Appendix}\label{s:appendix}
	
	\subsection{Green function estimate}
	\begin{lem}\label{apx:l:Green fun est}
		Under assumption \ref{A1} it holds that 
		\begin{align}\label{apx:eq:Green function sharp estimate}
			\GDfi(x,y)\asymp \left(\frac{\de(x)\de(y)}{|x-y|^{2}}\wedge 1\right)\frac{1}{|x-y|^{d}\phi(|x-y|^{-2})},\quad x,y\in D,
		\end{align}
		where the constant of comparability depends only on $d$, $D$ and $\phi$.
	\end{lem}
	\begin{proof}
		We slightly modify the proof of \cite[Theorem 3.1]{ksv_minimal2016} where the claim was proved under assumptions (A1)-(A5) from \cite{ksv_minimal2016}. Since \ref{A1} implies (A1)-(A4) from \cite{ksv_minimal2016}, we show that assumption (A5), which assumes that $\int_0^1\phi(\lambda)^{-1}d\lambda<\infty$, can be dropped in our setting. To shorten the proof, we note that every constant of comparability in the proof will depend at most on $d$, $D$ and $\phi$.

		The lower bound proved in \cite{ksv_minimal2016} does not use (A5) so we need to modify just the calculations for the upper bound.
		
		Similarly as in \cite{ksv_minimal2016}, let us define
		\begin{align*}
			I_1(r)&\coloneqq\int_0^{r^2}\left(\frac{\de(x)\de(y)}{t}\wedge 1\right)t^{-d/2}e^{-\frac{c\, r^2}{t}}\uu(t)dt,\\
			I_2(r)&\coloneqq\int_{r^2}^{(2\diam D)^2}\left(\frac{\de(x)\de(y)}{t}\wedge 1\right)t^{-d/2}e^{-\frac{c\, r^2}{t}}\uu(t)dt,\\
			L&\coloneqq \int_{(2\diam D)^2}^\infty e^{-\lambda_1 t}\de(x)\de(y)\uu(t)dt,
		\end{align*}
		where $\lambda_1$ is the first eigenvalue of $-\left.\Delta\right\vert_{D}$, see Subsection \ref{ss:operator}, and the constant $c$ is the constant $ c_4$ from \eqref{eq:heat kernel estimate}. In addition to the bounds \eqref{eq:heat kernel estimate}, there  is another one for all big enough $t>0$:
		$$ p_D(t,x,y)\asymp e^{-\lambda_1 t}\varphi_1(x)\varphi_1(y)\overset{\eqref{eq:varphi=de sharp estimate}}{\asymp}e^{-\lambda_1 t}\de(x)\de(y),\quad x,y\in D,\,t\ge \diam D,$$
		see \cite[Theorem 4.2.5]{davies_HeatKernel} and \cite[Remark 3.3]{song_sharp_bounds_2004}. Hence,
		\begin{align}
			\GDFI(x,y)&=\int_0^\infty p_D(t,x,y)\uu(t)dt=\left(\int_0^{|x-y|^2}+\int_{|x-y|^2}^{(2\diam D)^2}+\int_{(2\diam D)^2}^\infty\right) p_D(t,x,y)\uu(t)dt\nonumber\\
			&\lesssim I_1(|x-y|)+I_2(|x-y|)+L.\label{eq:proof Green bound}
		\end{align}
		Obviously,
		$$L\le \de(x)\de(y)\int_0^\infty e^{-\lambda_1 t}u(t)=\frac{\de(x)\de(y)}{\phi(\lambda_1)}\lesssim \left(\frac{\de(x)\de(y)}{|x-y|^{2}}\wedge 1\right)\frac{1}{|x-y|^{d}\phi(|x-y|^{-2})},$$
		since $|x-y|^{-d}\phi(|x-y|^{-2})^{-1}$ explodes at $x=y$ by \ref{A1}.
		
		For $I_1$ we imitate the calculations for \cite[Eq. (3.7)]{ksv_minimal2016}. Since $\uu(t)\lesssim \frac{\phi'(t^{-1})}{t^2\phi(t^{-1})^2}$  by \eqref{eq:Levy density upper bound} for all $t>0$, and since $t\mapsto\phi'(t^{-1})/\phi(t^{-1})^2$ increases, by the change of variables $c\,r^2/t=s$ we have
		\begin{align*}
			I_1(r)&\lesssim\int_0^{r^2}\left(\frac{\de(x)\de(y)}{t}\wedge 1\right)t^{-d/2}e^{-\frac{c\,r^2}{t}}\frac{\phi'(t^{-1})}{t^2\phi(t^{-1})^2}dt\\
			&\le \frac{\phi'(r^{-2})}{\phi(r^{-2})^2}\int_0^{r^2}\left(\frac{\de(x)\de(y)}{t}\wedge 1\right)t^{-d/2-2}e^{-\frac{c\,r^2}{t}}dt\\
			&\lesssim\left(\frac{\de(x)\de(y)}{r^2}\wedge 1\right)\frac{\phi'(r^{-2})}{r^{d+2}\phi(r^{-2})^2}\int_{c\,}^\infty s^{d/2+1}e^{-s}ds\lesssim\left(\frac{\de(x)\de(y)}{r^2}\wedge 1\right)\frac{1}{r^{d}\phi(r^{-2})},
		\end{align*}
		where the last inequality follows from \eqref{eq:scaling and the derivative}.
		
		The calculation for $I_2$ is slightly different  than  the one for \cite[Eq. (3.8)]{ksv_minimal2016}. Note that $\uu(t)\lesssim \frac{\phi'(t^{-1})}{t^2\phi(t^{-1})^2}\lesssim\frac{1}{t\phi(t^{-1})}\lesssim \frac{t^{\delta_2-1}}{r^{2\delta_2}\phi(r^{-2})}$, for $r^2\le t\le (2\diam D)^2$, where in the last approximate inequality we used \ref{WSC}. Hence
		\begin{align*}
			I_2(r)&\lesssim\frac{1}{r^{2\delta_2}\phi(r^{-2})}\int_{r^2}^{(2\diam D)^2}\left(\frac{\de(x)\de(y)}{t}\wedge 1\right)t^{\delta_2-1-d/2}e^{-\frac{c\,r^2}{t}}dt\\
			&\le\left(\frac{\de(x)\de(y)}{r^2}\wedge 1\right)\frac{1}{r^{2\delta_2}\phi(r^{-2})}\int_{r^2}^{\infty}t^{\delta_2-d/2-1}dt\lesssim\left(\frac{\de(x)\de(y)}{r^2}\wedge 1\right)\frac{1}{r^{d}\phi(r^{-2})}.
		\end{align*}
		
		The claim now follows from \eqref{eq:proof Green bound}.
		
	\end{proof}
	
	\subsection{Boundary behaviour of some integrals}\label{apx:proof of boundary GDFI}
	\begin{lem}\label{apx:l:approx of surface cap}
		For $\Gamma=\{y\in\partial D: |x-y|\le 2\de(x)\}$ it holds that
		\begin{align*}
			\int_{\Gamma}|x-y|^{-d}\asymp \de(x)^{-1},\quad x\in D.
		\end{align*}
	\end{lem}
	\begin{proof}
		Since $D$ is a $C^{1,1}$ set, for small enough $\de(x)$ the boundary part $\Gamma$ can be described as $\Gamma=\{q\in\R^{d-1}:|\de(x)-f(q)|^2+|q|^2\le 4\de(x)^2\}$, for some $C^{1,1}$ function $f$ on $\R^{d-1}$ such that $f(0)=0$ and $\nabla f(0)=0$, whereas $x$ can be viewed as $x=(0,\dots,0,\de(x))$. Hence
		\begin{align*}
			\int_{\Gamma}\frac{\de(x)}{|x-y|^{d}}d\sigma(y)&\asymp\de(x) \int\limits_{\{q\in\R^{d-1}:|\de(x)-f(q)|^2+|q|^2\le 4\de(x)^2\}}\frac{\sqrt{1+|\nabla f(q)|^2}}{(|\de(x)-f(q)|^2+|q|^2)^{d/2}}dq\\
			&\asymp\int\limits_{\{z\in\R^{d-1}:|1-f(\de(x)z)/\de(x)|^2+|z|^2\le 4\}}\frac{1}{\left(\left|1-\frac{f(\de(x)z)}{\de(x)}\right|^2+|z|^2\right)^{d/2}}dz,
		\end{align*}
		where we first used that $|\nabla f|$ is bounded by the Lipschitz property and then the substitution $q=\de(x)z$. Since $f\in C^{1,1}(\R^{d-1})$ such that $f(0)=0$ and $\nabla f(0)=0$, by the dominated convergence theorem the last integral converges to
		$$\int\limits_{\{z\in\R^{d-1}:1+|z|^2\le 4\}}\frac{1}{(1+|z|^2)^{d/2}}dz<\infty$$
		as $\de(x)\to 0$.
		
	\end{proof}
	
	The two following lemmas are in  the  spirit the same as \cite[Lemma A.4 \& Lemma A.5]{semilinear_bvw}. These lemmas from \cite{semilinear_bvw} lead to a result  that  is an analogue of Theorem \ref{t:boundary of GDFIU} in the case of subordinate Brownian motions, see \cite[Proposition 4.1]{semilinear_bvw}.
	
	Let $\epsilon=\epsilon(D) >0$ be such that the map $\Phi:\partial D\times (-\epsilon, \epsilon)\to \R^d$ defined by $\Phi(y,\delta)=y+\delta\mathbf{n}(y)$ defines a diffeomorphism to its image, cf.~\cite[Remark 3.1]{AGCV19}. Here $\mathbf{n}$ denotes the unit interior normal.  Without loss of generality assume that $\epsilon<\mathrm{diam}(D)/20$.

	\begin{lem}\label{apx:l:GDFI close-to-bdry}
		Let $\eta<\epsilon$ and assume that  conditions \ref{U1}-\ref{U4} hold true. Then for any $x\in D$ such that $\delta_D(x)<\eta/2$, 
		\begin{align}\label{eq:close-to-bdry}
			\begin{split}
				\GDFI\big(U(\de)\1_{(\de<\eta)}\big)(x)&\asymp \frac{1}{\de(x)^2\phi(\de(x)^{-2})}\int_0^{\de(x)}U(t)t\, dt+ \de(x)\int_0^{\eta}U(t)t\, dt \\
				&\qquad+\de(x)\int_{3\de(x)/2}^{\eta}\frac{U(t)}{t^2\phi(t^{-2})}\, dt\, .
			\end{split}
		\end{align}
		In particular, $\GDFI\big(U(\delta_D)\1_{(\delta_D<\eta)}\big)(x)<\infty$ if and only if  the integrability condition \eqref{eq:int-cond} holds true. Moreover, all comparability constants depend only on $d$, $D$ and $\phi$ and are independent of $\eta$.
	\end{lem}
	
	\begin{proof}
		Fix some  $r_0<\epsilon$ and fix $x\in D$ as in the statement. Define
		\begin{align}\label{apx:eq: D sets}
			\begin{split}
				D_1&= B(x, \delta_D(x)/2)\\
				D_2&=\{y: \delta_D(y)<\eta\}\setminus B(x, r_0)\\
				D_3&=\{y: \delta_D(y)<\delta_D(x)/2\}\cap B(x, r_0)\\
				D_4&= \{y:3\delta_D(x)/2 < \delta_D(y)<\eta\}\cap B(x,r_0)\\
				D_5&=\{y: \delta_D(x)/2<\delta_D(y)<3\delta_D(x)/2\}\cap (B(x, r_0)\setminus B(x, \delta_D(x)/2)).
			\end{split}
		\end{align}
		Thus we have that
		$$
		\GDFI\big(U(\de)\1_{(\de<\eta)}\big)(x)=\sum_{j=1}^5 \int_{D_j}\GDFI(x,y)U(\de(y))\, dy =:\sum_{j=1}^5 I_j.
		$$

		\noindent
		{\bf Estimate of $I_1$:} We show that
		\begin{equation}\label{e:estimate-I1}
			I_1\asymp \frac{U(\de(x))}{\phi(\de(x)^{-2})}\lesssim \frac{1}{\de(x)^2\phi(\de(x)^{-2})}\int_0^{\de(x)}U(t)t\,dt.
		\end{equation}
		Indeed, let $y\in D_1$. Then $\de(y)>\de(x)/2>|y-x|$ implying that 
		\begin{align}\label{eq:G-I1}
			\GDFI(x,y)\asymp \frac{1}{|x-y|^{d}\phi(|x-y|^{-2})}.
		\end{align}
		Further, by using first \eqref{eq:wus-U} and then \eqref{eq:doubling-U} we have that
		\begin{equation}\label{eq:estimate-I1b}
			U(\de(y))\asymp  U(\de(x)).
		\end{equation}
		Therefore, 
		\begin{align*}
			I_1 &\asymp\int_{D_1}U(\de(y)) \frac{1}{|x-y|^{d}\phi(|x-y|^{-2})}\, dy \\
			&\asymp U(\de(x)) \int\limits_{|y-x|<\de(x)/2}  \frac{1}{|x-y|^{d}\phi(|x-y|^{-2})}\, dy\asymp \frac{U(\de(x))}{\phi(\de(x)^{-2})}.
		\end{align*}
		Finally,  by \eqref{eq:wus-U} we get
		\begin{align*}
			\frac{1}{\de(x)^2\phi(\de(x)^{-2})}  \int_0^{\de(x)}U(t) t\, dt &\gtrsim \frac{U(\de(x)))}{\de(x)^2\phi(\de(x)^{-2})}\int_0^{\de(x)}  t\, dt\\
			&\asymp  \frac{U(\de(x))}{\phi(\de(x)^{-2})}.
		\end{align*}
		
		\noindent
		{\bf Estimate of $I_2$:} Next, we show that
		\begin{equation}\label{eq:estimate-I2}
			I_2\asymp \de(x) \int_0^{\eta}U(t)t\, dt\, .
		\end{equation}
		Let $y\in D_2$. Then $r_0 <|y-x|< \mathrm{diam}(D)$ so that $|y-x|\asymp 1$. This implies that $\GDFI(x,y)\asymp \de(x)\de(y)$. Therefore
		$$
		I_2 \asymp \de(x)\int_{D_2} U(\de(y))\de(y)\, dy \asymp \de(x)\int_{\de(y)<\eta} U(\de(y))\de(y)\, dy \asymp \de(x) \int_0^{\eta}U(t)t\, dt,
		$$
		where the last approximate equality follows by the co-area formula. 
		
		In estimates for $I_3$, $I_4$ and $I_5$ we will use the change of variables formula based on a diffeomorphism $\Phi: B(x, r_0)\to B(0,r_0)$ satisfying
		\begin{align*}
			& \Phi(D\cap B(x,r_0))=B(0,r_0)\cap \{z\in \R^d: \, z\cdot e_d>0\},\\
			& \Phi(y)\cdot e_d= \de(y) \ \  \textrm{for any }y\in B(x,r_0), \quad \Phi(x)=\de(x)e_d, 
		\end{align*}
		cf. \cite[p. 38]{AGCV19}. For the point $z\in \R^d_+=\{z\in \R^d: \, z\cdot e_d>0\}$ we will write $z=(\wt{z}, z_d)$. Several times we also use the following integrals:
		\begin{align}
			\int_0^{a}\frac{s^{d-2}}{(1+s)^d}\, ds&= \frac{(1+1/a)^{1-d}}{(d-1)}\, ,\quad a>0,\label{eq:int1}\\
			\int_0^{a}\frac{s^{d-2}}{(1+s)^{d+2}}\, ds&= \frac{a^{d-1}}{(1+a)^{d+1}}\big(2a(1+a)+d+2ad+d^2\big)\, ,\quad a>0.\label{eq:int2}
		\end{align}
		
		\noindent 
		{\bf Estimate of $I_3$:} We prove that  
		\begin{equation}\label{eq:estimate-I3}
			I_3\, \asymp\,\frac{1}{\de(x)^2\phi(\de(x)^{-2})} \int_0^{\de(x)}U(t)t\, dt.
		\end{equation}
		To see this, take $y\in D_3$. Then $\de(y)\le \de(x)/2$ implying $|x-y|\ge \de(x)/2$, and thus
		\begin{equation}\label{eq:G-I3 1}
			\GDFI(x,y)\asymp \frac{\de(x)\de(y)}{|x-y|^{d+2}\phi(|x-y|^{-2})}.
		\end{equation}
		Therefore, by repeating the first five lines of the calculations in \cite[p. 34 for the integral $I_3$]{semilinear_bvw} with our bound \eqref{eq:G-I3 1} we get
		\begin{align}
			I_3 &\asymp \int_0^{r_0/\de(x)}s^{d-2}\int_0^{1/2}\frac{U(\de(x)h)h}{\big((1-h)+s\big)^{d+2}\phi\big(\de(x)^{-2}((1-h)+s)^{-2}\big)}dh\,ds\nonumber\\
			&\asymp \int_0^{r_0/\de(x)}s^{d-2}\int_0^{1/2}\frac{U(\de(x)h)h}{\big(1+s\big)^{d+2}\phi\big(\de(x)^{-2}(1+s)^{-2}\big)}dh\,ds,\label{eq:G-I3-middle1}
		\end{align}
		where the last line comes from $\frac12\le h\le 1$. Further, for $\phi$ it holds that
		\begin{align}\label{eq:G-I3-middle2}
			(1+s)^{-2}\phi(\de(x)^{-2})\le\phi\big(\de(x)^{-2}(1+s)^{-2}\big)\le \phi(\de(x)^{-2}),
		\end{align}
		see \eqref{eq:simple global scaling}. Since we have
		\begin{align}\label{eq:G-I3-middle3}
			\int_0^{r_0/\de(x)}\frac{s^{d-2}}{(1+s)^d}ds\asymp 1,\quad \int_0^{r_0/\de(x)}\frac{s^{d-2}}{(1+s)^{d+2}}ds\asymp 1,
		\end{align}
		by \eqref{eq:int1} and \eqref{eq:int2}, by applying the inequalities \eqref{eq:G-I3-middle2} to \eqref{eq:G-I3-middle1} and by using \eqref{eq:G-I3-middle3} we obtain
		\begin{align*}
			I_3 &\asymp \frac{1}{\phi(\de(x)^{-2})}\int_0^{1/2}U(\de(x)h)h\,dh\\
			&\asymp \frac{1}{\de(x)^2\phi(\de(x)^{-2})}\int_0^{\de(x)/2}U(t)t\,dt.
		\end{align*}
		This proves  \eqref{eq:estimate-I3} since the almost non-increasing condition \eqref{eq:wus-U} implies
		\begin{align*}
			\int_0^{\de(x)/2}U(t) t\, dt&=\int_0^{\de(x)} U(t/2)t\, dt \gtrsim \int_0^{\de(x)}  U(t)   t\,dt.
		\end{align*}
		
		\noindent
		{\bf Estimate of $I_4$:} By applying the same calculations as in $I_3$, we show that 
		\begin{equation}\label{eq:estimate-I4}
			I_4\,\asymp\, \de(x)\int_{3\de(x)/2}^{\eta} \frac{U(t)}{t^2\phi(t^{-2})}\, dt\,.
		\end{equation}
		Let $y\in D_4$. Then  $|x-y|\ge \de(x)/2$ and $|x-y|\ge \de(y)/3$, hence $\GDFI(x,y)$ is of the form \eqref{eq:G-I3 1}. By following the  computation in \cite[p. 35 for the integral $I_4$]{semilinear_bvw} with our bound \eqref{eq:G-I3 1}, we arrive at
		\begin{align}
			I_4 &\asymp \int_0^{r_0/\de(x)}s^{d-2}\int_{3/2}^{\eta/\de(x)}\frac{U(\de(x)h)h}{\big((h-1)+s\big)^{d+2}\phi\big(\de(x)^{-2}((h-1)+s)^{-2}\big)}dh\,ds\nonumber\\
			&\asymp \int_0^{r_0/\de(x)}s^{d-2}\int_{3/2}^{\eta/\de(x)}\frac{U(\de(x)h)h}{\big(h+s\big)^{d+2}\phi\big(\de(x)^{-2}(h+s)^{-2}\big)}dh\,ds\nonumber\\
			&\asymp \int_0^{r_0/(\de(x)h)}s^{d-2}\int_{3/2}^{\eta/\de(x)}\frac{U(\de(x)h)}{h^2\big(1+s\big)^{d+2}\phi\big((\de(x)h)^{-2}(1+s)^{-2}\big)}dh\,ds,\label{eq:G-I4-middle1}
		\end{align}
		where the second line comes from $\frac13 h\le h-1\le h$. By applying \eqref{eq:G-I3-middle2} in \eqref{eq:G-I4-middle1}, since the relations \eqref{eq:G-I3-middle3} also hold for $r_0/(\de(x)h)$ instead of $r_0/\de$, we get
		\begin{align*}
			I_4 &\asymp \int_{3/2}^{\eta/\de(x)}\frac{U(\de(x)h)}{h^2\phi\big((\de(x)h)^{-2}\big)}dh\\
			&\asymp \de(x)\int_{3\de(x)/2}^{\eta}\frac{U(t)}{t^2\phi\big(t^{-2}\big)}dt.
		\end{align*}
		
		\noindent
		{\bf Estimate of $I_5$:}  Under  the almost non-increasing condition \eqref{eq:wus-U} and the doubling condition \eqref{eq:doubling-U} we show that
		\begin{equation}\label{eq:estimate-I5}
			I_5 \,\lesssim\, \frac{U(\de(x))}{\phi(\de(x)^{-2})}\,\lesssim\, \frac{1}{\de(x)^2\phi(\de(x)^{-2})}\int_0^{\de(x)}U(t)t\, dt\, .
		\end{equation}
		Indeed, let $y\in D_5$. Then $|x-y|>\de(x)/2>\de(y)/3$, hence  $\GDFI(x,y)$ is of the form \eqref{eq:G-I3 1}. The estimate \eqref{eq:estimate-I1b} also  holds since $\de(y)\asymp \de(x)$. Therefore
		\begin{align*}
			I_5 &\asymp \de(x)\int_{D_5} \frac{U(\de(y))\de(y)}{|x-y|^{d+2}\phi(|x-y|^{-2})}\, dy\\
			&\lesssim \frac{U(\de(x))}{\phi(\de(x)^{-2})} \int_{D_5}\frac{1}{|x-y|^d}\, dy\, ,
		\end{align*}
		where the last line comes from $|x-y|^2\phi(|x-y|^{-2})\asymp \phi'(|x-y|^{-2})\ge \phi'(4\de(x)^{-2})\asymp\de(x)^2\phi(\de(x)^{-2})$ since $\phi'$ decreases.
		To end the calculations, it was shown in \cite[p. 42]{AGCV19} that the last integral is comparable to 1. This proves the first approximate inequality in \eqref{eq:estimate-I5}, while the second was already proved in the estimate of $I_1$. 
		
		The proof is finished by noting that $I_1+I_5\lesssim I_3$.
	\end{proof}
	
	\begin{lem}\label{apx:l:away-from-bdry}
		Let $\eta<\epsilon$ and assume that  conditions \ref{U1}-\ref{U4} hold true. There exists $c(\eta)>0$ such that for any $x\in D$ satisfying $\de(x)\ge\eta/2$, 
		\begin{equation}\label{eq:away-from-bdry}
			\GDFI\big(U(\de)\1_{(\de<\eta)}\big)(x)\le c(\eta)\, .
		\end{equation}
	\end{lem}
	
	\begin{proof}
		Fix $x\in D$ as in the statement and define
		\begin{align*}
			D_1=&\{y:\, \de(y)<\eta/4\},\\
			D_2=&\{y:\, \eta/4 \le \de(y) <\eta\}.
		\end{align*}
		Then
		$$
		\GDFI\big(U(\delta_D)\1_{(\delta_D<\eta)}\big)(x)=\sum_{j=1}^2 \int_{D_j}\GDFI(x,y)U(\delta_D(y))\, dy =:\sum_{j=1}^2 J_j.
		$$
		
		\noindent 
		{\bf Estimate of $J_1$:} 
		We show that 
		\begin{equation}\label{eq:estimate-J1}
			J_1\lesssim\frac{1}{\eta^2\phi(\eta^{-2})} \int_0^{\eta}U(t)t\, dt.
		\end{equation}
		Let $y\in D_1$. Then $\de(y)<\eta/4\le\de(x)/2$, hence by using $|x-y|\ge \de(x)-\de(y)$ we have that $|x-y|>\de(y)$ and $|x-y|>\de(x)/2$. This implies that $\GDFI(x,y)$ satisfies \eqref{eq:G-I3 1}. Therefore,
		$$
		J_1\asymp \de(x)\int_{D_1}\frac{U(\de(y))\de(y)}{|x-y|^{d+2}\phi(|x-y|^{-2})}\, dy\lesssim \frac{\de(x)}{\eta^2\phi(\eta^{-2})}\int_{D_1}\frac{U(\de(y))\de(y)}{|x-y|^{d}}\, dy,
		$$
		since on $D_1$ we have $|x-y|\ge\eta/4$, hence $|x-y|^2\phi(|x-y|^{-2})\gtrsim \eta^2\phi(\eta^{-2})$ by \eqref{eq:simple global scaling}.
		
		By using the co-area formula we get (below $dy$ denotes the $d-1$ dimensional Hausdorff measure on $\{\de(y)=t\}$)
		\begin{align}\label{eq:estimate-J1-inner}
			J_1\lesssim \frac{\de(x)}{\eta^2\phi(\eta^{-2})}\int_0^{\eta/4}U(t)t\left(\int_{\de(y)=t}\frac{1}{|x-y|^d}\, dy\right) dt. 
		\end{align}
		The inner integral in \eqref{eq:estimate-J1-inner} is estimated as follows: For $\de(y)=t$ it holds that $|x-y|\ge \de( x)-t$, hence $|x-y|^{-d}\le (\de(x)-t)^{-d}$. The $d-1$ dimensional Hausdorff measure of $\{\de(y)=t\}$ is larger than or equal to the $d-1$ dimensional  Hausdorff measure of the sphere around $x$ of radius $\de(x)-t$ which is comparable to $(\de(x)-t)^{d-1}$. This implies that the inner integral is estimated from above by a constant times $(\de(x)-t)^{-1}$. Thus
		$$
		J_1\lesssim \frac{1}{\eta^2\phi(\eta^{-2})}\de(x) \int_0^{\eta/4}U(t)\,t\, (\de(x)-t)^{-1}\, dt.
		$$
		If $t<\eta/4$, then $t<\de(x)/2$, implying $\de(x)/2<\de(x)-t<\de(x)$. Therefore,
		$$
		J_1\lesssim \frac{1}{\eta^2\phi(\eta^{-2})}\int_0^{\eta}U(t)t\, dt.
		$$
		
		\noindent
		{\bf Estimate of $J_2$:}
		It holds that 
		\begin{equation}\label{eq:estimate-J2}
			J_2\preceq U(\eta/4).
		\end{equation}
		Let $y\in D_2$. By the almost non-increasing condition \eqref{eq:wus-U} we have $U(\de(y))\le c_1 U(\eta/4)$, hence
		\begin{align*}
			J_2 &\lesssim \int_{\eta/4< \de(y) <\eta}\frac{U(\de(y))}{|x-y|^d\phi(|x-y|^{-2})}\, dy \lesssim U(\eta/4)\int_{\eta/4< \de(y) <\eta}\frac{1}{|x-y|^d\phi(|x-y|^{-2})}\, dy\\
			&\le U(\eta/4)\int_{B(x, 2\diam D)} \frac{1}{|x-y|^d\phi(|x-y|^{-2})}\, dy \lesssim U(\eta/4).
		\end{align*}
		The last estimate uses the fact that the integral is not singular.
		
		By putting together estimates for $J_1$ and $J_2$, we see that there exists $c_2>0$ such that
		$$
		\GDFI\big(U(\delta_D)\1_{(\delta_D<\eta)}\big)(x)\le c_2\left(\frac{1}{\eta^2\phi(\eta^{-2})}\int_0^{\eta}U(t)t\, dt +U(\eta/4)\right)=: c(\eta).
		$$
	\end{proof}
	
	\begin{proof}[\textbf{Proof of Theorem} \ref{t:boundary of GDFIU}]
		Fix some $\eta<\epsilon$ and treat it as a constant. Note that on $\{\de(y)\ge \eta\}$ it holds that $U$ is bounded (by the assumption  \ref{U4}). Therefore
		\begin{equation}\label{eq:gpe-b}
			\GDFI \big(U(\de)\1_{(\de\ge \eta)}\big)(x)\lesssim\GDFI \de(x)\asymp \de(x)\,,\quad x\in D,
		\end{equation}
		by Lemma \ref{l:rate of GDdelta}. For the lower bound of this term note that on $\{\de(x)\ge\eta/2\}$ we have
		\begin{align}\label{e:gpe-b1}
			\GDFI \big(U(\de)\1_{(\de\ge \eta)}\big)(x)&\gtrsim\int_{B(x,\eta/4)}\frac{1}{|x-y|^{d}\phi(|x-y|^{-2})}dy\asymp \frac{1}{\phi(16/\eta^2)}\gtrsim 1,
		\end{align}
		and on $\{\de(x)\le\eta/2\}$ we have
		\begin{align}
			\GDFI \big(U(\de)\1_{(\de\ge \eta)}\big)(x)&\gtrsim \de(x)\int_{\de(y)\ge \eta}\frac{\de(y)}{|x-y|^{d+2}\phi(|x-y|^{-2})}dy\gtrsim \de(x).
		\end{align}
		Since $\de(x)\asymp 1$ on $\{\de(x)\ge\eta/2\}$, we have just obtained $\GDFI \big(U(\de)\1_{(\de\ge \eta)}\big)(x)\asymp \de(x)$ in $D$. Further, by Lemma \ref{apx:l:away-from-bdry}, if $\de(x)\ge \eta/2$, then $\GDFI\big(U(\delta_D)\1_{(\delta_D<\eta)}\big)(x)\le c(\eta)$. Hence, 
		$$
		\GDFI(U(\de))(x)\asymp 1, \quad \de(x)\ge \eta/2.
		$$
		Since for $\de(x)\ge \eta/2$ the right-hand side of \eqref{eq:GDFI U sharp bound} is also comparable to 1, this proves the claim for the case $\de(x)\ge \eta/2$.
		
		Assume now that $\de(x)<\eta/2$. By Lemma \ref{apx:l:GDFI close-to-bdry} and \eqref{eq:gpe-b} we have that \eqref{eq:GDFI U sharp bound}  holds	where we clearly replaced $\eta$ of \eqref{eq:close-to-bdry} by $\diam D$ since we treat $\eta$ as a constant.
		
		Finally, we prove that $\GDFI(U(\de))(x)/\PDFI\sigma(x)\to 0$ as $x\to\partial D$. It is obvious that $\PDFI\sigma$ annihilates the first and the second term of  \eqref{eq:GDFI U sharp bound}. For the third term, note that on $\{t\ge \de(x)\}$ we have $t^2\phi(t^{-2})\gtrsim \de(x)^2\phi(\de(x)^{-2})$ and  $U(t)\de(x)\le U(t)t$. By applying the dominate convergence theorem we obtain
		\begin{align*}
			\frac{\de(x)\int_{\de(x)}^{\diam D}\frac{U(t)}{t^2\phi(t^{-2})}\, dt\,}{\PDFI\sigma(x)}\lesssim \int_{\de(x)}^{\diam D}U(t)\de(x)dt\to0,
		\end{align*}
		as $\de(x)\to 0$.
	\end{proof}

	\begin{lem}\label{apx:l:boundary operator of GDFI 2}
		Let $t<\epsilon$. There exists $C=C(d,D,\phi)>0$ such that for $\de(x)\ge \frac{t}{2}$ it holds that
		\begin{align*}
			\int_{\de(y)\le t}\frac{\GDFI(x,y)}{\PDFI\sigma(y)}dy\le C\,\wt f(x,t),
		\end{align*}
		where $0\le \wt f(x,t)\le t\,\de(x)$ on $\{\de(x)\ge t/2\}$ and $\wt f(x,t)/t\to 0$ as $t\to 0$ for every fixed $x\in D$.
	\end{lem}
	\begin{proof}
		We need a little adaptation of Lemma \ref{apx:l:away-from-bdry}. We break the set $D_2$  into  three pieces.	Fix $r_0<\epsilon$ and $x\in D$ as in the statement. Define
		\begin{align*}
			D_1&=\{y:\, \de(y)<t/4\},\\
			D_2&=\{y:\, t/4 \le \de(y) <t\}\cap B(x,t/4),\\
			D_3&=\{y:\, t/4 \le \de(y) <t\}\cap B(x,t/4)^c\cap B(x,r_0),\\
			D_4&=\{y:\, t/4 \le \de(y) <t\}\cap B(x,r_0)^c.
		\end{align*}
		Then 
		$$\int_{\de(y)\le t}\frac{\GDFI(x,y)}{\PDFI\sigma(y)}dy=\sum_{i=1}^4\int_{D_i}\frac{\GDFI(x,y)}{\PDFI\sigma(y)}dy=\sum_{i=1}^4J_i.$$
		
		\noindent
		{\bf Estimate of $J_1$:} We prove
		\begin{align}\label{apx:eq:J1}
			J_1\lesssim t^2.
		\end{align}
		Let $y\in D_1$. Then $\de(y)<t/4\le\de(x)/2$, hence  $|x-y|>\de(x)/2>\de(y)$. This implies that $\GDFI(x,y)$ satisfies \eqref{eq:G-I3 1}, and $\GDFI(x,y)/\PDFI\sigma(y)\lesssim \de(x)\de(y)/|x-y|^d$. Therefore, by using the co-area formula in the second comparison, we get
		\begin{align*}
			J_1\lesssim \de(x)\int_{D_1}\frac{\de(y)}{|x-y|^d}\asymp \de(x)\int_0^{t/4}h\left(\int_{\de(y)=h}\frac{\sigma(dy)}{|x-y|^d}\right)dh.
		\end{align*}
		The inner integral is estimated as before, see the paragraph under \eqref{eq:estimate-J1-inner}, i.e. the inner integral is bounded from above by a constant times $(\de(x)-h)^{-1}$. Thus
		\begin{align*}
			J_1\lesssim \de(x)\int_{0}^{t/4}\frac{h}{\de(x)-h}dh.
		\end{align*}
		However, when $h<t/4$ we have $\frac12\de(x)\le \de(x)-h\le \de(x)$, therefore
		\begin{align*}
			J_1\lesssim \int_{0}^{t/4}h\,dh\lesssim t^2.
		\end{align*}

		In the following integral estimates we have $y\in D$ such that $t/4\le\de(y)\le t$ so $\PDFI\sigma(y) \asymp \frac{1}{t^2\phi(t^{-2})}$. 
		
		\noindent
		{\bf Estimate of $J_2$:}
		We prove
		\begin{align}\label{apx:eq:J2}
			J_2\lesssim t^2.
		\end{align}
		On $D_2$ we obviously have $\GDFI(x,y)\asymp \frac{1}{|x-y|^{d}\phi(|x-y|^{-2})}$, hence
		\begin{align*}
			J_2\lesssim t^2\phi(t^{-2})\int_{B(x,t/4)}\frac{1}{|x-y|^{d}\phi(|x-y|^{-2})}dy\asymp t^2\phi(t^{-2})\frac{1}{\phi(t^{-2})} \asymp t^2.
		\end{align*}
		
		\noindent
		{\bf Estimate of $J_3$:} We prove that $J_3\lesssim f(x,t)$ for a function $f$ which satisfies $0\le f(x,t)/t\lesssim \de(x)$ and $f(x,t)/t\to0$ as $t\to 0$ for every fixed $x\in D$.
		
		To this end, since $y\in D_3$, hence $|x-y|\ge t/4$, it holds that $\GDFI(x,y) \asymp \frac{\de(x)t}{|x-y|^{d+2}\phi(|x-y|^{-2})}$. Hence,
		\begin{align}\label{apx:eq:J3-1}
			J_3\asymp t^3\phi(t^{-2})\de(x) \int_{D_3}\frac{1}{|x-y|^{d+2}\phi(|x-y|^{-2})}dy\eqqcolon f(x,t).
		\end{align}
		Since $|x-y|\ge t/4$, we have $|x-y|^{2}\phi(|x-y|^{-2})\gtrsim t^2\phi(t^{-2})$ by \eqref{eq:simple global scaling}, hence
		\begin{align}\label{apx:eq:J3-2}
			f(x,t)/t\lesssim \de(x)\int_{D_3}\frac{1}{|x-y|^{d}}dy.
		\end{align}
		Also, by reducing to the flat case we have
		\begin{align*}
			\int_{D_3}\frac{1}{|x-y|^{d}}dy&\asymp\int_{t/4}^{r_0}\int_{t/4}^t\frac{r^{d-2}}{(|\de(x)-h|+r)^{d}}dh\,dr\\
			&\asymp \int_{t/4}^{r_0}\int_{(t/4-\de(x))/r}^{(t-\de(x))/r}\frac{1}{r(|\rho|+1)^d}d\rho\,dr.
		\end{align*}
		Since $\rho\mapsto 1/(|\rho|+1)$ is bell-shaped, and the inner interval $[(t/4-\de(x))/r,(t-\de(x))/r]$ has fixed length, the inner integral is maximal when the inner interval is symmetric (which is when $\de(x)=\frac58 t$), thus, we get
		\begin{align*}
			\int_{D_3}\frac{1}{|x-y|^d}&\lesssim \int_{t/4}^{r_0}\int_{-3t/(8r)}^{3t/(8r)}\frac{1}{r(|\rho|+1)^d}d\rho\,dr\\
			&=2\int_{t/4}^{r_0}\int_{0}^{3t/(8r)}\frac{1}{r(\rho+1)^d}d\rho\,dr.
		\end{align*}
		Further, $1\le \rho+1\le 3t/(8r)+1\le 3$ so we get
		\begin{align}\label{apx:eq:J3-3}
			\int_{D_3}\frac{1}{|x-y|^d}&\lesssim 
			\int_{t/4}^{r_0}\frac{t}{r^2}dr\lesssim 1.
		\end{align}
		Inserting the bound \eqref{apx:eq:J3-3} into \eqref{apx:eq:J3-2}, we get
		that $0\le f(x,t)/t\lesssim \de(x)$ where the constant of comparability depends only on $d$, $D$ and $\phi$.
		
		Further, if we fix $x$ and let $t\to 0$, then it is clear that  $\1_{D_3}\to 0$, and that $|x-y|^{-d-2}\phi(|x-y|^{-2})^{-1}\le c$ for every $y\in D_3$ for all small enough $t>0$. Hence, $f(x,t)/t\to0$ as $t\to 0$.
		
		\noindent
		{\bf Estimate of $J_4$:} We prove
		\begin{align}\label{apx:eq:J4}
			J_4\lesssim t^3\phi(t^{-2})\,\de(x).
		\end{align}
		For $y\in D_4$ we have $\GDFI(x,y)\lesssim \frac{\de(x)\de(y)}{|x-y|^{d+2}\phi(|x-y|^{-2})}\lesssim \de(x)t$ since $r_0\le |x-y|\le \diam D$. Hence, $J_4\lesssim t^3\phi(t^{-2})\de(x)$.
		
		To finish the proof, note that we can take $\wt f(x,t)=c \big(t^2 + f(x,t)+t^3\phi(t^{-2})\de(x)\big)$ for some constant $c=c(d,D,\phi)>0$.
	\end{proof}
	
	\subsection{Uniform integrability of some classes of functions}
	
	\begin{lem}\label{apx:l:Montenegro and Ponce Prop 2.1}
		Let $f:D\times \R\to\R$ be continuous in the second variable and $u_1,u_2\in\LLL$ such that $u_1\le u_2$. Assume that for every $u\in\LLL$ such that $u_1\le u \le u_2$ a.e. in $D$ it holds that $x\mapsto f(x,u(x))\in\LLL$. Then the family
		\begin{align*}
			\mathcal{F}\coloneqq \{f(\cdot,u(\cdot))\in\LLL:\:u_1\le u \le u_2\text{ a.e. in $D$} \}
		\end{align*}
		is uniformly integrable in $D$ with respect to the measure $\de(x)dx$, hence bounded in $\LLL$.
	\end{lem}
	\begin{proof}
		Before we start the proof, we refer the reader to \cite[Chapter 16]{schilling2017measures} for details on the uniform integrability. Also, the proof is motivated by the proof of  a  similar claim which can be found in \cite[Section 2]{montenegro2008sub}.
		
		Suppose that the family $\mathcal{F}$ is not uniformly integrable. Then there is $\varepsilon>0$, a sequence $(v_n)_n\subset \LLL$ such that $u_1\le v_n\le u_2$ a.e. in $D$, and a sequence $(E_n)_n$ consisting of measurable subsets of $D$ with property
		\begin{align*}
			\int_{E_n}|f(x,v_n(x))|\de(x)dx\ge\varepsilon,\quad n\in\N.
		\end{align*}
		Now use \cite[Lemma 2.1]{montenegro2008sub} with $w_n(\cdot)=|f(\cdot,v_n(\cdot))|\de(\cdot)/\varepsilon\in L^1(D)$ to extract a subsequence $(v_{n_k})_k$ of $(v_n)_n$ and disjoint sets $F_k\subset E_{n_k}$ such that
		\begin{align*}
			\int_{F_k}|f(x,v_{n_k}(x))|\de(x)dx\ge\frac\varepsilon2,\quad k\in\N.
		\end{align*}
		To finish the proof, define 
		\begin{align*}
			v(x)=\begin{cases}
				v_{n_k}(x),&x\in F_k,\\
				u_1(x), &x\in \cap_{k=1}^\infty F_k^c.
			\end{cases}
		\end{align*}
		We have $u_1\le v\le u_2$ in $D$, hence $v\in \LLL$. Further,
		\begin{align*}
			\int_D |f(x,v(x))|\de(x)dx\ge \sum_{k=1}^\infty\int_{F_k} |f(x,v_{n_k}(x))|\de(x)dx=\infty,
		\end{align*}
		which is a contradiction.
	\end{proof}
	
	\subsection{Regularity of transition densities}\label{ss:heat kernel regularity}
	The following result on the regularity up to the boundary of the transition kernel of the killed Brownian motion appears to be well known, but we were unable to find an exact reference. In the article we assumed that $D$ is a $C^{1,1}$ bounded domain, but this result we give for a slightly more general open set since the claim is important in itself.
	\begin{lem}\label{apx:l:regularity of heat kernel}
		Let $D$ be an open bounded $C^{1,\alpha}$ domain for some $\alpha\in (0,1]$. For the transition density $p_D(\cdot,\cdot,\cdot)$ of the killed Brownian motion upon exiting the set $D$ it holds that $p_D\in C^1((0,\infty)\times \overline D\times \overline D)$.
	\end{lem}
	\begin{rem}\label{apx:r:regularity of heat kernel}
		Moreover, we will see in the proof of the previous lemma that $p_D$ is somehow independently regular,  variable by variable. E.g. we can differentiate $p_D(t,x,y)$ in $x$ up to the boundary, then differentiate the obtained function in $y$ up to the boundary, and then differentiate in $t$ as many times as we want. This can be done up to $C^{1,\alpha}(\overline D)$  regularity in the second and the third variable and up to $C^\infty(0,\infty)$ regularity in the first variable.
	\end{rem}
	\begin{proof}[\textbf{Proof of Lemma \ref{apx:l:regularity of heat kernel}}]
		Note that $p_D(t,x,y)\le p(t,x,y)$ everywhere by \eqref{eq:trans.dens. KBM} so for fixed $t>0$ and $x\in D$ we have that the mapping $y\mapsto p_D(t,x,y)$ is in $L^\infty(D)\subset L^2(D)$. Hence, by the spectral representation of $L^2(D)$ functions we have
		\begin{align}\label{eq:heat spectral rep}
			p_D(t,x,y)=\sum_{j=1}^\infty e^{-\lambda_j t}\varphi_j(x)\varphi_j(y),
		\end{align}
		where we have used \eqref{eq:Pt=e{-lt}}.
		
		Now we show that the sum in \eqref{eq:heat spectral rep} converges uniformly and is bounded in  a  certain strong sense. First  note that $\varphi_j\in C^{1,\alpha}(\overline D)$ by \cite[Theorem 8.34]{gilbarg_pde}. Furthermore, by \cite[Theorem 8.33]{gilbarg_pde} the following estimate holds 
		\begin{align}\label{eq:eigen estimates}
			\|\varphi_j\|_{C^{1,\alpha}(D)}\le c_1(1+\lambda_j)\|\varphi_j\|_{L^\infty(D)},
		\end{align}
		where $\|\cdot \|_{C^{1,\alpha}(D)}$ is the standard ${C^{1,\alpha}(D)}$ H\"older norm and $c_1=c_1(d,D)>0$.
		Also, the eigenvalues satisfy the well known estimate
		\begin{align}\label{eq:eigen estimate 1}
			\|\varphi_j\|_{L^\infty(D)}\le c_2 \lambda_j^{d/4} \|\varphi_j\|_{L^2(D)}=c_2 \lambda_j^{d/4},
		\end{align}
		see e.g. \cite[Theorem 1.6]{BHV_2015_eigen_estimate}, where $c_2=c_2(d)>0$. In particular, this inequality and the inequality in \eqref{eq:eigen estimates} imply
		\begin{align}\label{eq:eigen estimates 2}
			\|\nabla \varphi_j\|_{L^\infty(D)}\le  C_1(1+\lambda_j)\|\varphi_j\|_{L^\infty(D)}\le c_3(1+\lambda_j)^{d/4+1},
		\end{align}
		for $c_3=c_3(d,D)>0$.
		Also note that since $\varphi_j$ vanishes on the boundary, by the mean-value theorem for every $x\in D$ there is some $\tilde x$ between $x$ and the closest boundary point to $x$ such that $$\left|\frac{\varphi_j(x)}{\de(x)}\right|=\left|\nabla\varphi_j(\tilde x)\right|\le \|\nabla \varphi_j\|_{L^\infty(D)}.$$
		Hence,  for the sum in \eqref{eq:heat spectral rep} the following uniform bound holds
		\begin{align}
			\begin{split}\label{eq:sum bound1}
				\sum_{j=1}^\infty e^{-\lambda_j t}|\varphi_j(x)||\varphi_j(y)|&\le \sum_{j=1}^\infty e^{-\lambda_j t}\|\varphi_j\|_{L^\infty(D)}\|\varphi_j\|_{L^\infty(D)}\\&\le c_2^2 \sum_{j=1}^\infty e^{-\lambda_j t} \lambda_j^{d/2}<\infty,
			\end{split}\\
			\begin{split}\label{eq:sum bound2}
				\sum_{j=1}^\infty e^{-\lambda_j t}|\varphi_j(x)|\left|\frac{\varphi_j(y)}{\de(y)}\right|&\le \sum_{j=1}^\infty e^{-\lambda_j t}\|\varphi_j\|_{L^\infty(D)}\|\nabla\varphi_j\|_{L^\infty(D)}\\&\le c_4 \sum_{j=1}^\infty e^{-\lambda_j t} (1+\lambda_j)^{d/2+1}<\infty,
			\end{split}
		\end{align}
		where $c_4=c_4(d,D)>0$ and the sums converge by Weyl's law, see \eqref{eq:Weyl's law}. Similar bounds hold if we take the derivate by the variable $t$ or by the variable $y$.
		
		Since $\varphi_j\in C^{1,\alpha}(\overline D)$ and since the bounds \eqref{eq:sum bound1} and \eqref{eq:sum bound2} hold, we can pass the needed limits inside the sum \eqref{eq:heat spectral rep} to get $p_D\in C^1((0,\infty)\times \overline D\times \overline D)$.
		
		In addition, since the bounds \eqref{eq:eigen estimates}-\eqref{eq:eigen estimates 2} hold, we can pass the limits inside the sum in the representation \eqref{eq:heat spectral rep} to get that the density $p_D$ is regular,  variable by variable up to $C^{1,\alpha}(\overline D)$  regularity in the second and the third variable and up to $C^\infty(0,\infty)$ regularity in the first variable, see Remark \ref{apx:r:regularity of heat kernel}.
	\end{proof}
	\subsection*{Acknowledgement}
	
	This research was supported in part by the Croatian Science Foundation under the project 4197. The author would like to thank Prof. Zoran Vondraček and Prof. Vanja Wagner for many discussions on the topic and for their helpful comments on the presentation of the results.


	\bibliographystyle{abbrv}
	\bibliography{bibliography_ivan}
	
	\bigskip
	
	\noindent{\bf Ivan Bio\v{c}i\'c}
	
	\noindent Department of Mathematics, Faculty of Science, University of Zagreb, Zagreb, Croatia,
	
	\noindent Email: \texttt{ibiocic@math.hr}
\end{document}